\documentclass[11pt, letter]{article}

\usepackage{amsfonts,amssymb, latexsym, amsmath, amsthm,mathrsfs, verbatim,geometry}
\usepackage{graphics}
\usepackage{slashed}
\usepackage{url,color}
\usepackage{enumerate}
\usepackage{esint}
\usepackage{pifont}
\usepackage{upgreek}
\usepackage{times}
\usepackage{calligra}
\usepackage{graphicx}
\usepackage{caption}
\usepackage{tikz}
\usetikzlibrary{calc}
\usetikzlibrary{arrows.meta}

\newcommand\J{\mathscr J}
\def\A{\mathscr{A}}

\def\energy{\mathcal{E}}

\def\be{\begin{equation}}
\def\ee{\end{equation}}
\def\bea{\begin{eqnarray}}
\def\eea{\end{eqnarray}}
\def\beas{\begin{eqnarray*}}
\def\eeas{\end{eqnarray*}}

\def\g{\partial}

\def\pa{\partial }

\def\der{\pa_r^a\slashed\nabla^\beta}
\def\car{\pa^\nu}

\def\sn{\slashed\nabla}

\def\norm{\mathcal S^N}
\def\vortnorm{\mathcal B^N}
\def\energy{\mathcal E^N}

\newtheorem{theorem}{Theorem}[section]

\newtheorem{proposition}[theorem]{Proposition}

\newtheorem{lemma}[theorem]{Lemma}
\newtheorem{remark}[theorem]{Remark}
 
\renewcommand{\theequation}{\arabic{section}.\arabic{equation}}

\title{Expanding large global solutions of the equations of compressible fluid mechanics}

\author{Mahir Had\v zi\'c\thanks{Department of Mathematics, King's College London, London, WC2S 2LR, UK. Email: mahir.hadzic@kcl.ac.uk.} \ and Juhi Jang\thanks{Department of Mathematics, University of Southern California, Los Angeles, CA 90089, USA and Korea Institute for Advanced Study, Seoul, Korea.  Email: juhijang@usc.edu.}}

\date{}
\begin{document}

\maketitle

\abstract{
Inspired by a recent work of Sideris on affine motions of compactly supported moving ellipsoids, we construct global-in-time solutions to the vacuum free boundary  three-dimensional isentropic compressible Euler equations when $\gamma\in(1,\frac53]$ for initial configurations that are sufficiently close to the affine motions, and satisfy the physical vacuum boundary condition. The support of these solutions expands at a linear rate in time, they remain smooth in the interior of their support, and no shocks are formed in the evolution. We impose no symmetry assumptions on our initial data.
We prove the existence of such solutions by reformulating the problem as a nonlinear stability question in suitably rescaled variables, wherein the stabilizing effect of the fluid expansion becomes visible in the range $\gamma\in(1,\frac53]$.
%In a recent work Sideris constructed a finite-parameter family of compactly supported affine solutions to the three-dimensional isentropic compressible Euler equations 
% satisfying the physical vacuum condition.  The support of these solutions expands at a linear rate in time. 
% We show that if the adiabatic exponent $\gamma$ belongs to the interval $(1,\frac53]$, then these affine motions are nonlinearly stable without any symmetry assumptions on the initial data. 
% Small perturbations lead to globally-in-time defined solutions that remain in the vicinity of the manifold of affine motions, they remain smooth in the interior of their support, and no shocks are formed in the process.
%
%Our strategy relies on two key ingredients. 
%We first provide a new interpretation of the affine motions using an (almost) invariant 
%action of GL$^+(3)$ on the compressible Euler system. This transformation dictates a particular rescaling of time and a change of variables, which in turn exposes a stabilization mechanism
% induced by the expansion of the background affine motions when $\gamma\in(1,\frac53]$. We then switch to a Lagrangian description of the rescaled Euler system which reflects the geometry of expanding affine motions.  
% We introduce new ideas with respect to the existing well-posedness frameworks and build high-order energy spaces to  prove global-in-time stability, thereby making crucial use of the new stabilization effect.
}

\section{Introduction}
 
We consider the dynamics of moving gases in three dimensions as described by compressible isentropic Euler system. We are interested in fluids surrounded by vacuum and therefore the unknowns are the density 
$\rho$, fluid velocity vector-field ${\bf u}$, and the free boundary of the support of $\rho$, denoted by $\pa\Omega(t)$. The resulting initial value problem takes the form:
\begin{subequations}
\label{E:EULER}
\begin{alignat}{2}
\g_t\rho + \text{div}\, (\rho \mathbf{u})& = 0 &&\ \text{ in } \ \Omega(t)\,;\label{E:CONTINUITYE}\\
\rho\left(\g_t  \mathbf{u}+ ( \mathbf{u}\cdot\nabla) \mathbf{u}\right) +\nabla p &= 0 &&\ \text{ in } \ \Omega(t)\,;\label{E:VELOCITYE}\\
p&=0&& \ \text{ on } \ \partial\Omega(t)\,; \label{E:VACUUME} \\
\mathcal{V} (\pa\Omega(t))&= \mathbf{u}\cdot \mathbf{n}(t)  && \ \text{ on } \ \partial\Omega(t)\,;\label{E:VELOCITYBDRYE}\\
(\rho(0,\cdot),  \mathbf{u}(0,\cdot))=(\rho_0,  \mathbf{u}_0)\,, & \ \Omega (0)=\Omega _0&&\,.\label{E:INITIALE}
\end{alignat}
\end{subequations}
Here $\mathcal{V} (\partial\Omega(t))$ denotes the normal velocity of $\pa\Omega(t)$ and $\mathbf{n}(t) $ denotes the outward unit normal vector to $\pa\Omega(t)$. 
Equation~\eqref{E:CONTINUITYE} is the well-known continuity equation, while~\eqref{E:VELOCITYE} expresses the conservation of momentum.
Boundary condition~\eqref{E:VACUUME} is the vacuum boundary condition, while~\eqref{E:VELOCITYBDRYE} is the kinematic boundary condition stating that the boundary
movement is tangential to the fluid particles.

In this article, we shall only consider ideal barotropic fluids, where the pressure depends only on the density, expressed through the following equation of state
\be\label{E:EQUATIONOFSTATE}
p = \rho^{\gamma}, \ \ \gamma>1,
\ee
where we have set the entropy constant to be 1. We additionally demand that the initial density satisfies the {\em physical vacuum boundary condition} \cite{JM1,LY2}:
\begin{align}\label{E:PHYSICALVACUUM}
-\infty < \frac{\pa\, c_s^2}{\pa n}\Big|_{\pa\Omega(0)}<0
\end{align}
where $c_s=\sqrt{\frac{d}{d\rho}p(\rho)}$ is the speed of the sound. 
We shall refer to the system of equations~\eqref{E:EULER}--\eqref{E:PHYSICALVACUUM} as the {\em Euler system} and denote it by E$_\gamma$.

Due to the inherent lack of smoothness of the enthalpy $c_s^2$ at the vacuum boundary (implied by the assumption~\eqref{E:PHYSICALVACUUM}), a rigorous understanding of 
the existence of physical vacuum states in compressible fluid dynamics has been a challenging problem. Only recently, a successful local-in-time well-posedness theory for the E$_\gamma$ system 
was developed in~\cite{CoSh2012, JaMa2015} using the Lagrangian formulation of the Euler system in the vacuum free boundary framework. The fundamental unknown is the flow map $\zeta$ defined as a solution of the ordinary differential equations (ODE)
\[
\pa_t\zeta = {\bf u}\circ \zeta.
\]

Very recently Sideris~\cite{Sideris} constructed a special class of {\em affine} fluid motions\footnote{Affine motions are often used to understand a qualitative behavior in fluid mechanics, for instance see Majda \cite{Majda86}.} that solve~\eqref{E:EULER}--\eqref{E:PHYSICALVACUUM} globally-in-time  
in the vacuum free boundary setting. 
The flow map and the velocity field of an affine solution by definition take the form 
\begin{align}\label{E:affine}
\zeta_A(t,x) = A(t)x, \ \ {\bf u}(t,x)=\dot A(t)A^{-1}(t)x, \ \ A(t)\in\text{GL}^+(3).
\end{align}
Plugging this ansatz into~\eqref{E:EULER} one can effectively separate variables and discover a family of evolving vacuum states of the form
\begin{align}\label{E:RHOAUA}
\rho_A(t,x) = \det A(t)^{-1} \left[\frac{(\gamma-1)}{2\gamma}(1-|A^{-1}(t)x|^2)\right]^{\frac{1}{\gamma-1}}, \ \ {\bf u}_A(t,x) = \dot{A}(t)A^{-1}(t)x, 
\end{align}
whereby the flow matrix $t\to A(t)$ solves the following Cauchy problem for a system of ODEs:
\begin{align}
\ddot A(t) & = \det A(t)^{1-\gamma}A(t)^{-\top}, \label{E:AEQUATION}\\ 
(A(0),\dot A(0)) & = (A_0,A_1) \in \text{GL}^+(3)\times \mathbb M^{3\times3}.\label{E:AEQUATIONINITIAL}
\end{align} 
The density and the velocity field are both supported on a moving ellipse $\Omega(t)$ of the form $A(t)\Omega$,
where $\Omega=B_1({\bf 0})$ is the unit ball in $\mathbb R^3$. We denote the set of such affine motions by $\mathscr S$.
Our main theorem states that the elements of $\mathscr S$ are stable under small perturbations if $\gamma\in(1,\frac53]$.

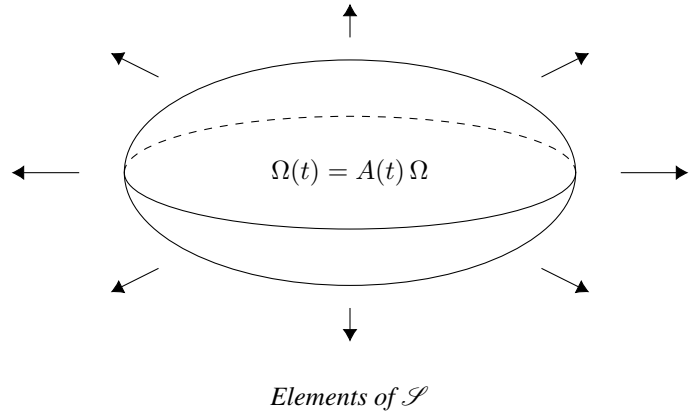
\begin{center}
\begin{tikzpicture}[>=stealth]
\def\wx{3}      % Width of outer ellipse.
\def\wy{1.5}    % Height of outer ellipse.
\def\wz{0.75}   % Height of inner ellipse.

% Draw outer ellipse.
\draw (0,0) ellipse (\wx cm and \wy cm);

% Draw top half of inner ellipse.
\begin{scope}
  \clip (-\wx-0.1, 0) rectangle (\wx+0.1, \wy);
  \draw[dashed] (0, 0) ellipse (\wx cm and \wz cm);
\end{scope}

% Draw bottom half of inner ellipse.
\begin{scope}
  \clip (-\wx-0.1, 0) rectangle (\wx+0.1, -\wy);
  \draw (0, 0) ellipse (\wx cm and \wz cm);
\end{scope}

% Draw text.
\node at (0, 0) {$\Omega(t) = A(t) \, \Omega$};
\node at (0,-3) {{\em Elements of $\mathscr S$}};

% Draw arrows.
\def\r{1.2}    % Radius multiplier for start of arrows.
\def\s{1.5}    % Radius multiplier for end of arrows.
\foreach \x in {45, 90, ..., 360} {
  \draw[-{Latex[width=2mm]}] ({sin(\x)*\wx*\r}, {cos(\x)*\wy*\r}) -- ({sin(\x)*\wx*\s}, {cos(\x)*\wy*\s});
}
\end{tikzpicture}
\end{center}

\begin{theorem}[Global existence in the vicinity of affine motions]\label{T:MAINTHEOREM}
Assume that $\gamma\in(1,\frac53]$. 
Then small perturbations 
of the expanding affine motions given by \eqref{E:RHOAUA}--\eqref{E:AEQUATIONINITIAL}  give rise to unique globally-in-time defined solutions of the Euler system E$_\gamma$. 
Moreover, their support expands at a linear rate and they remain close to the underlying ``moduli" space $\mathscr S$ of affine motions.
\end{theorem}

\begin{remark}
A precise statement of Theorem~\ref{T:MAINTHEOREM} specifying the function spaces and the notions of ``small" and ``close" in the statement above is provided in Theorem~\ref{T:MAINLAGR}.
\end{remark}

\begin{remark}
Theorem~\ref{T:MAINTHEOREM} covers a range of adiabatic exponents of physical importance. The exponent $\gamma=\frac53$ is commonly used in the description of a monatomic gas, $\gamma=\frac75$ corresponds to a diatomic gas, and $\gamma=\frac43$ is often referred to as the radiative case.
\end{remark}

\begin{remark}\label{R:RESTRICTIONS} 
As shown in~\cite{Sideris}, the affine motions exist even when $\gamma>\frac53$. 
It would be interesting 
to understand whether one can go beyond the $\gamma=\frac53$ threshold in our theorem.  The assumption $\gamma\le\frac53$ is certainly optimal for our method, as the case $\gamma>\frac53$
yields an anti-damping effect that can in principle lead to an instability.
\end{remark}

\begin{remark}
The theorem shows that in the forward time direction, in the vicinity of $\mathscr S$ the solution naturally splits into an affine component (i.e. an element of $\mathscr S$) and a remainder, which as we shall show, remains small due to the underlying expansion associated with the elements of $\mathscr S$.
\end{remark}

To the best of our knowledge, Theorem \ref{T:MAINTHEOREM} is the first global-in-time existence result for {\em non-affine} expanding solutions to the Euler system E$_\gamma$. These solutions are unique in the vacuum free boundary framework, and they are simultaneously weak solutions in $\mathbb R^3$ obtained by extending $\rho$ to 0 in the vacuum region. In fact,  
Theorem~\ref{T:MAINTHEOREM} is consistent with another result of Sideris~\cite{Sideris2014}: {\em if the Euler system E$_\gamma$ admits global-in-time compactly supported solutions then the diameter of the support has to grow at least linearly in time. In other words, $\sup_{x\in \Omega(t)}|x| \gtrsim  t$.}   
The existence of such expanding global-in-time solutions has been shown only recently~\cite{Sideris} in the context of a finite parameter family of affine motions $\mathscr S$ introduced above.  One important novelty of our result is the introduction of a rigorous mathematical framework that allows us to describe the nonlinear stable expansion phenomenon of the Euler flows for all time around the set $\mathscr S$.  Theorem~\ref{T:MAINTHEOREM} provides an evidence that the expansion of the gas at a linear rate is a mechanism that prevents the development of shock singularities\footnote{In the absence of free boundaries, gas expansion also plays an important role in the global results in~\cite{Se1997,Grassin98}. The full nonlinear analysis of stabilizing effects of the fluid expansion in the context of general relativistic cosmological models was initiated by Rodnianski \& Speck~\cite{iRjS2012} and extended further in~\cite{jS2012,cLjVK,HaSp,Ol,Fr}}.

\subsection{Existence theories for compressible Euler flows}

We briefly review some existence theories of compressible  flows, where we focus on the works most relevant to our problem. We will not attempt to provide an exhaustive overview, but for more references we refer the reader to~\cite{Daf}. 

It is well-known~\cite{Daf} that the system \eqref{E:EULER} is strictly hyperbolic if the density is bounded below away from zero.   
In the absence of vacuum, the theory of symmetric
hyperbolic systems developed by Friedrichs-Lax-Kato applies and one can construct local-in-time smooth solutions, see for instance Majda's book~\cite{Majda84}. 
These smooth solutions break down in a finite time for generic initial data, although global-in-time classical solutions can exist for special initial data with $c_s$ sufficiently smooth and small, see for instance Serre~\cite{Se1997} and Grassin \cite{Grassin98}. In particular, in two space dimensions there exist initial data leading to so-called eternal solutions for all $t\in\mathbb R$~\cite{Se,Grassin98}.
The first proof of formation of singularities for classical Euler flow with a nonzero constant density at infinity was given by Sideris~\cite{Sideris2}. A similar result was obtained by Makino-Ukai-Kawashima \cite{MUK86} for smooth compactly supported disturbance moving into vacuum.  
In an important work~\cite{Chr2007}, Christodoulou gave a precise description of shock formation for irrotational relativistic fluids starting with small smooth initial data. 
We refer to the work by Luk-Speck~\cite{jS2016} for more recent developments in this direction. We remark that the above results on singularity formation do not apply to the physical vacuum free boundary problem. 
In fact Liu-Smoller \cite{LS} showed that the shock waves vanish at the vacuum and the singular 
behavior is reminiscent of the centered rarefaction waves, which suggests that the vacuum has a regularizing effect in that setting~\cite{LY2}. 

In the context of weak solutions it has been  known since the work of DiPerna~\cite{DiPerna1983} and subsequently Chen~\cite{Chen1997}, Lions-Perthame-Souganidis~\cite{LiPeSo}
that the one-dimensional isentropic Euler system allows for a globally defined notion of a weak solution. We mention that this theory requires the initial density $\rho_0$ to be strictly positive and supported uniformly away from zero, although a formation of vacuum regions is not necessarily dynamically precluded.
We note that for multi-dimensional flows there is no entropy criterion that has so far been imposed so that the uniqueness of solutions is ensured and this question is still open. 
Recent results of Chiodaroli-DeLellis-Kreml~\cite{ChDeKr} show that the two-dimensional compressible Euler system is in fact strongly
ill-posed with respect to a commonly used notion of an admissible weak solution. 

In the framework of the vacuum free boundary, the existence theory depends strongly on the behavior of initial data (containing vacuum). 
When the sound speed $c_s$ is smooth across the vacuum boundary, 
Liu-Yang~\cite{LY1} constructed local-in-time solutions to one-dimensional Euler system with
damping and showed that $c_s^2$ cannot be smooth across the vacuum after a finite time. 
Our physical vacuum assumption~\eqref{E:PHYSICALVACUUM} corresponds precisely to the requirement that $c_s$ is $\frac12-$H\"{o}lder continuous across the initial vacuum interface.
Local-in-time existence theory for the Euler system with physical vacuum was developed by  Coutand-Shkoller~\cite{CoSh2011,CoSh2012} and Jang-Masmoudi~\cite{JaMa2009,JaMa2015}, where substantial new ideas with respect to prior works were introduced in order to handle the above mentioned vacuum degeneracy. See also a subsequent work by Luo-Xin-Zeng~\cite{LXZ}. 
 On the other hand, if $c_s$ is H\"{o}lder continuous but does not satisfy the physical vacuum condition, some ill-posedness result can be found in \cite{JM2012}. In this regime, a satisfactory theory is still far from being complete.
 We also mention the work by Lindblad~\cite{L1} on the vacuum free boundary problem for compressible liquids when the fluid is in contact with vacuum discontinuously, namely the density is positive at the vacuum boundary.  For an elegant discussion on the expansion of a gas in vacuum and a motivation for the introduction of the physical vacuum condition~\eqref{E:PHYSICALVACUUM} see also a recent work by Serre~\cite{Se}.

In the presence of damping, Liu~\cite{L2} constructed explicit spherically symmetric self-similar solutions that satisfy the physical vacuum condition and asymptotically converge to Barenblatt solutions of the porous media equation. 
Liu conjectured that  solutions of the Euler system with damping in vacuum converge to solutions of the porous media equation.  Huang-Marcati-Pan \cite{HMP} established the conjecture in the
 entropy solution framework, and in a recent work \cite{LZ}, Luo-Zheng justified the convergence for one-dimensional Euler system with damping in the physical vacuum free boundary setting.

\subsection{Methodology and plan of the paper}

At the heart of our approach is an almost invariant action of $\text{GL}^+(3)$ on the solutions of the Euler system E$_\gamma$.
Namely, for any given $(\rho,{\bf u })$ and $A\in \text{GL}^+(3)$ 
we consider a transformation 
\be\label{E:TRANSFORMATION}
(\rho,{\bf u }) \mapsto (\tilde\rho,\tilde{\bf u })
\ee
defined  by
\begin{align}
\rho(t,x) & = \det A^{-1}\tilde\rho\left(\det A^{\frac{1-3\gamma}{6}}t,\, A^{-1}x\right) \label{E:ALMOSTRHO}\\
{\bf u}(t,x) & = \det A^{\frac{1-3\gamma}{6}}A\tilde{\bf u}\left(\det A^{\frac{1-3\gamma}{6}}t,\, A^{-1}x\right) \label{E:ALMOSTVELOCITY}
\end{align}
It is straightforward to check that  if $(\rho,{\bf u})$ solve~\eqref{E:CONTINUITYE}--\eqref{E:VELOCITYE}, 
then the pair $(\tilde\rho,\tilde{\bf u })$ solves the following generalized Euler system:
\begin{subequations}
\label{E:GEULER}
\begin{alignat}{2}
\g_s\tilde\rho + \text{div}\, (\tilde\rho \tilde{\mathbf{u}})& = 0 &&\ \text{ in } \ \tilde\Omega(s)\,;\label{E:CONTINUITYG}\\
\tilde\rho\left(\pa_s  \tilde{\mathbf{u}}+ (\tilde{\mathbf{u}}\cdot\nabla) \tilde{\mathbf{u}}\right) +\Lambda\nabla (\tilde\rho^{\gamma}) &= 0 &&\ \text{ in } \ \tilde\Omega(s)\,;\label{E:VELOCITYG}
\end{alignat}
\end{subequations}
where 
\begin{align}\label{E:NEWQUANITITIES}
\Lambda: = \det A^{\frac23}A^{-1}A^{-\top}, \ \ \tilde\Omega(s) = A^{-1}\Omega(s), \ \ s = \det A^{\frac{1-3\gamma}{6}}t.
\end{align}
Matrix $\Lambda$ is clearly symmetric, positive definite and  belongs to  $\text{SL}(3)$. Here $A^{-\top}$ is by definition the transpose of the inverse of $A$.
With respect to the original momentum equation~\eqref{E:VELOCITYE}, a structural novelty  
is the presence of the matrix $\Lambda$ in~\eqref{E:VELOCITYG}. We should stress here that a related similarity transformation appears in~\cite{Se1997}.

In the special case when $A$ is a {\bf conformal} matrix (i.e. $A(\det A)^{-\frac13}\in \text{SO}(3)$) $\Lambda$ is the identity matrix and the transformation~\eqref{E:TRANSFORMATION}
is an exact invariance. This invariance simply expresses the fact that the problem possesses both a scaling and a rotational symmetry. Conformal case corresponds to arbitrary compositions
of these symmetries acting on the solution space of the Euler system {\em E}$_\gamma$.

A rich scaling freedom of compressible Euler equations is  
responsible for the presence of a finite parameter family of global-in-time vacuum states discovered by Sideris~\cite{Sideris}. We now proceed to explain this in a little more detail, while the rigorous arguments can be found in Section~\ref{S:MAINRESULT0}. 
Since the transformation~\eqref{E:ALMOSTRHO}--\eqref{E:ALMOSTVELOCITY} leaves the total mass $M(\rho)(t)=\int_{\mathbb R^3}\rho(t,x) \,dx$ unchanged we refer to it as
{\em mass-critical}. It is thus plausible to look for paths $t\mapsto A(t)\in\text{GL}^+(3)$ that yield special solutions of the E$_\gamma$-system of the form~\eqref{E:ALMOSTRHO}--\eqref{E:ALMOSTVELOCITY}.
After rescaling time and space according to the self-similar change of variables
\be\label{E:SSINTRO}
\frac{ds}{dt} = \det A(t)^{\frac{1-3\gamma}{6}}, \ \ y =A(t)^{-1}x
\ee
we discover that the unknowns $(\tilde\rho(s,y),\tilde{\bf u}(s,y))$ solve a new and more complicated looking system of equations. However, after introducing a modified velocity
\be\label{E:CONFORMAL}
{\bf U} = \tilde{\bf u} - A^{-1}A_s y,
\ee
the unknowns $(\tilde\rho,{\bf U})$ in turn solve a system with a much more agreeable structure, reminiscent of E$_\gamma$. The system takes a schematic form
\begin{align}
&\pa_s\tilde\rho +  \text{div}\,\left(\tilde\rho{\bf U}\right)  = 0 \label{E:CONTINUITYNEWINTRO}\\
& \pa_s{\bf U} +({\bf U}\cdot\nabla){\bf U}+\mathcal F_1(A,A_s){\bf U} + \frac\gamma{\gamma-1}\Lambda\nabla(\tilde\rho^{\gamma-1})
 = \mathcal F_2(A,A_s,A_{ss}) y, \label{E:VELOCITYNEWINTRO}
\end{align}
where $\mathcal F_i$, $i=1,2$, are some explicitly given matrix-valued smooth functions (see Section~\ref{S:SS}) and $\Lambda = \Lambda(A)$ is defined in~\eqref{E:NEWQUANITITIES}.
For any $\delta>0$ we can find a special solution of~\eqref{E:CONTINUITYNEWINTRO}--\eqref{E:VELOCITYNEWINTRO} by setting
${\bf U}=0$, $\frac\gamma{\gamma-1}\nabla(\tilde\rho^{\gamma-1})=-\delta y$ and 
\be\label{E:SIDERISODEINTRO}
\mathcal F_2(A,A_s,A_{ss}) = -\delta\Lambda.
\ee
A simple calculation shows that the equation~\eqref{E:SIDERISODEINTRO} is just a restatement of the ODE~\eqref{E:AEQUATION} in the rescaled time variable $s$ and this way we rediscover all the the affine motions from~\cite{Sideris}! 
This situation exhibits a certain analogy to the work of Merle-Rapha\"el-Szeftel~\cite{MeRaSz} wherein the stable self-similar behavior of solutions to the slightly supercritical Schr\"odinger equation is investigated, see Remark~\ref{R:MRS}

The question of stability of a given element of $\mathscr S$ reduces to a problem of nonlinear stability
of a particular steady state of~\eqref{E:CONTINUITYNEWINTRO}--\eqref{E:VELOCITYNEWINTRO} where $\mathcal F_2$ is replaced by  $-\delta\Lambda$ and $\mathcal F_1$ is a prescribed $s$-dependent matrix-valued function. However, there are several caveats. First, the solutions of~\eqref{E:AEQUATION} blow-up in finite $s$-time and therefore the self-similar time variable $s$ is not well-suited for the study of stability. This is related to the fact that the affine motions from $\mathscr S$, even though of Type I\footnote{By definition, Type I self-similar solutions approach an equilibrium exponentially fast in the logarithmic time variable~\eqref{E:TAUINTRO}. We refer to the review article~\cite{EgFo} for a detailed discussion of this terminology.}, are {\em not} self-similar with respect to the rescaling~\eqref{E:SSINTRO}, but instead the associated solutions $A(t)$ grow linearly in $t$ independently of $\gamma>1$. We change to a new time scale $\tau$ 
\be\label{E:TAUINTRO}
\tau\sim_{t\to\infty}\log t,
\ee
which resolves the above issue and simultaneously elucidates the stabilizing effect of the expansion of the background affine motion. 
In the new variables the $\tau$-dependent 
coefficient $\mathcal F_1$ in~\eqref{E:VELOCITYNEWINTRO} gives a manifestly dissipative energy contribution in the range $\gamma\in(1,\frac53)$ and suggests a possible nonlinear stability mechanism.
The logarithmic change of time-scale~\eqref{E:TAUINTRO} is 
often used in the study of Type I behavior and has been crucially exploited in authors' work on stable expanding solutions for the mass-critical Euler-Poisson system~\cite{HaJa}.

To turn the above strategy to a rigorous proof, we need a good well-posedness theory. Due to the physical vacuum condition~\eqref{E:PHYSICALVACUUM} the initial enthalpy is not smooth across the vacuum interface and this makes the question of (even) local-in-time existence of solutions rather challenging.  Only recently, the problem of local-in-time well-posedness in the presence of the physical vacuum condition was fully solved~\cite{CoSh2012,JaMa2015}. At the core of both approaches is the use of Lagrangian coordinates which is particularly convenient as the problem is automatically pulled back to a fixed domain.
Our work draws especially on the approach from~\cite{JaMa2015} where a high-order energy method relying on only spatial differential operators is developed.

In our case we work with a Lagrangian formulation of~\eqref{E:CONTINUITYNEWINTRO}--\eqref{E:VELOCITYNEWINTRO} with respect to the modified velocity ${\bf U}$,
where the basic unknown is the flow map defined by the ODE
\be\label{E:LAGRINTRO}
\dot\eta(t,y) = {\bf U}(t,\eta(t,y))
\ee
such that $\eta\equiv y$ represents the expanding background solution. 
While the continuity equation~\eqref{E:CONTINUITYNEWINTRO} retains the same form as~\eqref{E:CONTINUITYE}, the presence of the matrix $\Lambda$ in the momentum equation~\eqref{E:VELOCITYNEWINTRO}   
introduces two important difficulties compared  
 to the work~\cite{JaMa2015}:
\begin{enumerate}
\item[(i)]  Since the curl of the modified pressure term $\Lambda\nabla(\tilde\rho^{\gamma-1})$ {\em does not} vanish, the curl of the velocity field {\em does not} satisfy a simple transport equation, a crucial tool in implementing both strategies in~\cite{CoSh2012, JaMa2015};
\item[(ii)] It is a priori not obvious  
how to extract a positive definite energy contribution from~\eqref{E:VELOCITYNEWINTRO}.
\end{enumerate}
To handle the first issue we note that an adapted $\Lambda$-curl operator
defined via
\[
\left[\text{curl}_{\Lambda}{\bf F}\right]_i := \epsilon_{ijk}\Lambda^s_{j}{\bf F}^k,_s \ \ (\epsilon_{ijk} \ \text{is the usual permutation symbol})
\] 
annihilates the $\Lambda$-gradient $\Lambda\nabla$.\footnote{One could introduce a more geometric language so that $\text{curl}_{\Lambda}$ and $\Lambda\nabla$ are ``natural" operators with respect
to a given metric structure, but for the sake of conciseness we choose not to.} Our key observation is that the $\Lambda$-curl of ${\bf U}$ satisfies a transport equation which at the top order decouples from the rest of the dynamics, allowing us to obtain ``good" estimates for the (Lagrangian pull-back of) $\Lambda$-curl of ${\bf U}$. 
The perturbation around the steady state is given by 
\[
\uptheta(\tau,y):=\eta(\tau,y)- y.
\] 
A simplified schematic form of the equation for $\uptheta$  is 
\begin{align}\label{E:THETAINTRO}
w^\alpha e^{\mu(\gamma)\tau} \left(\pa_{\tau\tau}\uptheta + c\pa_{\tau}\uptheta\right) + \left(\Lambda\nabla\right)^* \mathcal F(w^{1+\alpha}D\uptheta) = 0 
\ \ \text{ in } \ \Omega=B_{1}({\bf 0}), 
\end{align}
where $\left(\Lambda\nabla\right)^*$ is the Lagrangian pull-back of $\Lambda\nabla$, $\mathcal F:\mathbb M^{3\times3}\to\mathbb R$ is a smooth function, $\mu(\gamma),c,\alpha=\alpha(\gamma)>0$ given constants, and $w:\Omega\to\mathbb R$ is the equilibrium fluid enthalpy satisfying $w^\alpha = \rho_A(0)$ (initial density of the background expanding solution). See~\eqref{E:THETAEQUATION} for a precise version. 
Starting from~\eqref{E:THETAINTRO} we derive a high-order energy by integrating the equation against spatial derivatives of  $\Lambda^{-1}\uptheta_\tau$.
To resolve the issue (ii) mentioned above we crucially use the fact that $\Lambda$ is a real symmetric non-degenerate matrix depending only on $\tau$. If $\pa$ denotes a generic high-order differential operator, 
we apply it to the $\Lambda$-gradient term in~\eqref{E:THETAINTRO} and use the Leibniz rule to isolate the leading order term. After a  suitable anti-symmetrization we can extract a coercive energy contribution at the expense of creating a new top order term containing 
$\text{curl}^*_{\Lambda}(\pa\uptheta)$ where $\text{curl}^*_{\Lambda}$ is the Lagrangian pull-back of the operator $\text{curl}_{\Lambda}$. 
This procedure relies on a delicate algebraic structure of the top-order term and requires diagonalizing the matrix $\Lambda$, see Lemma~\ref{L:KEYLEMMA}. 
The $\text{curl}^*_{\Lambda}$ in turn satisfies a ``good" transport equation 
and we are able to
estimate it in terms of our natural high-order energy. 

The leading order energy contribution associated with $\pa_\tau$-derivatives in~\eqref{E:THETAINTRO} takes the following schematic form
\be\label{E:TAUNORM}
\frac12\frac{d}{d\tau}\left(e^{\mu(\gamma)\tau} \|w^{\frac\alpha2}\sqrt{\Lambda^{-1}}\pa_\tau\uptheta\|_{L^2(\Omega)}^2 \right) 
+ (c-\frac12\mu(\gamma))e^{\mu(\gamma)\tau}\|w^{\frac\alpha2}\sqrt{\Lambda^{-1}}\pa_\tau\uptheta\|_{L^2(\Omega)}^2.
\ee
As shown in Proposition~\ref{P:ENERGYESTIMATE1}, the analogue of the requirement $c-\frac12\mu(\gamma)\ge0$ amounts to the condition $\gamma\le\frac53$ which is assumed in Theorem~\ref{T:MAINTHEOREM}. Similarly, 
an exponentially growing term $e^{\mu(\gamma)\tau}$ in~\eqref{E:THETAINTRO} acts as a medium for stabilization 
and we make crucial use of it in closing the high-order estimates.

After introducing a high-order weighted norm $\norm$, we prove an energy inequality of the form
\begin{align}\label{E:ENERGYINTRO}
\norm(\tau) \le C_0 + C \int_0^\tau e^{-\mu_\ast\tau'}\norm(\tau')\,d\tau' \ \ \text{ for some $\mu_\ast>0$}.
\end{align}
The constant $C_0>0$ depends only on the initial data while the exponentially decaying factor inside the integral reflects the presence of exponential-in-$\tau$ weights in~\eqref{E:TAUNORM}.
The basic technical tool in the proof of~\eqref{E:ENERGYINTRO} is the use of Hardy-Sobolev embeddings between weighted Sobolev spaces explained in Appendix~\ref{A:HARDY}. The global existence for sufficiently small perturbed initial data follows from a standard continuity argument applied to~\eqref{E:ENERGYINTRO}.

Last but not least, it is important for our argument 
to use only spatial differential operators that do not destroy the structure of the time weight in  \eqref{E:THETAINTRO}.  
If we use vector fields containing $\tau$-derivatives, we encounter unpleasant commutator terms stemming from the exponential weights in~\eqref{E:THETAINTRO}.  
By employing the approach of Jang-Masmoudi~\cite{JaMa2015} with a significant refinement as explained above, we are able to control the high-order energy of spatial derivatives with powers of $w$ 
as weights {\it uniformly-in-time}, see the definition of our high-order norm~\eqref{E:SNORM}. 
The key weighting pattern~\cite{JaMa2015} is that every normal derivative with respect to the boundary $\pa\Omega=\mathbb S^2$ requires adding a multiple of the fluid enthalpy $w$
in the weight. If not for this weighting structure, it would be impossible to close the nonlinear energy estimates, see Remark~\ref{R:Comm}. This is consistent with the equation~\eqref{E:THETAINTRO} and such weighted spaces relate to the usual Sobolev spaces via Hardy-Sobolev embeddings, as laid out in Appendix~\ref{A:HARDY}.

We conclude the introduction with the plan for the rest of the paper. In Section~\ref{S:MAINRESULT0} we explain in detail  how the affine motions relate to the symmetries of E$_\gamma$-system, Section~\ref{S:SS}, and we then 
formulate the stability problem and state the main result in Lagrangian coordinates, Sections~\ref{S:LAGR}--\ref{S:MAINRESULT}. Vorticity bounds are explained in Section~\ref{S:VORTICITY},
main energy estimates in Section~\ref{S:ENERGY},
and proof of the main theorem in Section~\ref{S:MAINPROOF}. In Appendix~\ref{A:ASYMPTOTIC} we summarize the properties of the background solution using the results from~\cite{Sideris}. Appendix~\ref{A:COMM} contains the basic properties of commutators between various differential operators used in the paper, while Appendix~\ref{A:HARDY} contains the statements of frequently used Hardy-Sobolev embeddings. In Appendix~\ref{A:LAG}, we give an alternative and equivalent formulation of the problem starting from the Lagrangian coordinates. 

\section{Formulation and Main result}\label{S:MAINRESULT0}

\subsection{A self-similar rescaling}\label{S:SS}

Motivated by an (almost) invariant transformation~\eqref{E:ALMOSTRHO}--\eqref{E:ALMOSTVELOCITY} we 
are seeking for a path 
\[
\mathbb R_+\ni t\mapsto A(t) \in \text{GL}^+(3)
\]
of rescaling matrices such that the unknowns $(\tilde\rho,\tilde{\bf u})$ defined by  
\begin{align}
\rho(t,x) & = \det A(s)^{-1} \tilde{\rho}(s, y) , \label{E:DENSITY2}\\
\mathbf{u}(t,x) & = \det A(s)^{\frac{1-3\gamma}{6}}A(s)\tilde{\mathbf{u}}(s,y) , \label{E:VELOCITY2}
\end{align}
solve the Euler system E$_\gamma$. Here the new time and space coordinates $s$ and $y$ are given by 
\be\label{E:TIMESCALE}
\frac{ds}{dt} = \frac{1}{ \det A(t)^{\frac{3\gamma-1}{6}}}, \ \ y = A(t)^{-1}x,
\ee
motivated by the transformation~\eqref{E:ALMOSTRHO}--\eqref{E:ALMOSTVELOCITY}.
Introducing the notation
\be\label{E:MUDEF}
\mu(s) : = \det A(s)^{\frac13},
\ee 
a simple application of the chain rule transforms the equations~\eqref{E:CONTINUITYE}--\eqref{E:VELOCITYE} into
\begin{align}
\tilde\rho_s - 3\frac{\mu_s}{\mu}\tilde\rho - \left(A^{-1}A_sy\cdot\nabla\right)\tilde\rho + \text{div}\,\left(\tilde\rho\tilde{\bf u}\right)& =0, \label{E:CONTINUITYHAT2}\\
\pa_s\tilde{\bf u} - \frac{3\gamma-1}{2}\frac{\mu_s}{\mu}\tilde{\bf u} + A^{-1}A_s\tilde {\bf u} + ( \tilde{\mathbf{u}}\cdot\nabla) \tilde{\mathbf{u}} \label{E:DENSITYHAT2}
-A^{-1}A_s y\cdot \nabla\tilde{\bf u} 
+\frac\gamma{\gamma-1}\Lambda\nabla(\tilde\rho^{\gamma-1}) &=0,
\end{align}
where we recall that
$
\Lambda(s) = \det A(s)^{\frac23}A(s)^{-1}A(s)^{-\top}.
$
A structural miracle of the new system~\eqref{E:CONTINUITYHAT2}--\eqref{E:DENSITYHAT2} is that it can be recast in a simpler form, close to the original Euler system E$_\gamma$. 
To that end we introduce a modified velocity
\be
{\bf U}(s,y)  : = \tilde{\bf u}(s,y) +B(s) y, \label{E:MODIFIEDVELOCITY}
\ee
where 
\be
 B(s) :=- A^{-1}A_s, \label{E:MATRIXB}
\ee
and note that 
\be\label{E:JACOBI}
\text{div}\,\left( A^{-1}A_s  y\right) = \text{Tr}\left( A^{-1}A_s\right) = \pa_s\det A(s) \det A(s)^{-1} = \frac{3\mu_s}{\mu}.
\ee
Then the system~\eqref{E:CONTINUITYHAT2}--\eqref{E:DENSITYHAT2} can be rewritten as
\begin{align}
&\pa_s\tilde\rho +  \text{div}\,\left(\tilde\rho{\bf U}\right)  = 0 \label{E:CONTINUITYNEW}\\
& \pa_s{\bf U} +({\bf U}\cdot\nabla){\bf U}+\left(-\frac{3\gamma-1}{2}\frac{\mu_s}{\mu}\,\text{{\bf Id}}-2B(s)\right){\bf U} + \frac\gamma{\gamma-1}\Lambda\nabla(\tilde\rho^{\gamma-1}) \notag \\
& \ \ \ = \left[B_s +\left(-\frac{3\gamma-1}{2}\frac{\mu_s}{\mu}\,\text{{\bf Id}}-B\right)B\right] y , \label{E:VELOCITYNEW}
\end{align}
where we used~\eqref{E:JACOBI} in verifying~\eqref{E:CONTINUITYNEW}. 

For any given $\delta\in\mathbb R$ we look 
for solutions $({\bf U},\tilde \rho,B)$ of~\eqref{E:CONTINUITYNEW}--\eqref{E:VELOCITYNEW} of the form
\begin{align}
{\bf U} &={\bf 0} \ \ \text{ in } \Omega,\label{E:UEQUATION}\\
B_s +\left(-\frac{3\gamma-1}{2}\frac{\mu_s}{\mu}\,\text{{\bf Id}}-B\right)B & =-\delta\Lambda, \ \ s\ge0, \label{E:BEQUATION}\\
 \frac\gamma{\gamma-1}\nabla(\rho^{\gamma-1}) & = - \delta y \ \ \text{ in } \Omega. \label{E:RHOEQUATION}
\end{align}
Equation~\eqref{E:RHOEQUATION} is easy to solve and leads to 
the enthalpy profile $ w:=\tilde \rho^{\gamma-1}$ given by
\begin{align}\label{E:BARW}
 w( y) = \frac{\delta(\gamma-1)}{2\gamma}\left(1-| y|^2\right), \quad  y\in \Omega,
\end{align}
where we consider only the case $\delta>0$ to ensure that the physical vacuum condition~\eqref{E:PHYSICALVACUUM} holds true.
Equation~\eqref{E:UEQUATION} gives
\be\label{E:HATUSTEADY}
{\bf u} (s, y) = A^{-1}A_s y.
\ee
Interestingly, equation~\eqref{E:BEQUATION} does not posses globally defined solutions in the $s$-variable; instead we must rescale back to the original time variable $t$. 
Recalling that $\frac{dt}{ds} = \det {A}^{\frac{3\gamma-1}{6}}$ and~\eqref{E:MATRIXB} we can re-express~\eqref{E:BEQUATION} as a differential equation for $A(t)$. Note that 
\begin{align*}
B_s &= -\det A^{\frac{3\gamma-1}{6}} \partial_t(A^{-1} A_t\det{A}^{\frac{3\gamma-1}{6}}) \\
& = \det{A}^{\frac{3\gamma-1}3}A^{-1} A_t A^{-1}A_t - \det{A}^{\frac{3\gamma-1}3} A^{-1}A_{tt} +\frac{3\gamma-1}{6} \det{A}^{\frac{3\gamma-1}3}
\frac{\pa_t\det{A}}{\det{A}}\\
& = B^2 +\frac{3\gamma-1}{2} \frac{\mu_s}{\mu} B - \det{A}^{\frac{3\gamma-1}3} A_{tt}.
\end{align*}
Therefore, from~\eqref{E:BEQUATION} and the definition of $\Lambda$ it follows that 
\begin{align*}
- \det{A}^{\frac{3\gamma-1}3} A^{-1}A_{tt} = -\delta \det{A}^{\frac23}A^{-1}A^{-\top},
\end{align*}
which is equivalent to 
\be\label{E:SIDERISODE}
A_{tt} = \delta \det{A}^{1-\gamma} A^{-\top}.
\ee
We have recovered~\eqref{E:AEQUATION} (modulo a harmless parameter $\delta>0$).  
This is exactly the ODE describing the affine motion of a fluid blob from~\cite{Sideris}.
By Theorem 6~\cite{Sideris} for any $\gamma>1,\delta>0$ there exists a global-in-$t$ solution to~\eqref{E:SIDERISODE},
and the family of solutions $({\bf U},\tilde\rho,B)$ comprises all the affine motions constructed by Sideris in~\cite{Sideris}.  
We observe that  the space of all such solutions denoted by $\mathscr S$ 
forms a finite-parameter family parametrized by the choice of the parameter $\delta$ and the initial conditions for the ODE~\eqref{E:SIDERISODE}, 
\[
(A_0,A_1,\delta)\in\text{GL}^+(3)\times \mathbb M^{3\times3}\times \mathbb R_+.
\]
Here the parameter $\delta>0$ dictates the total mass of the gas. From \eqref{E:BARW}, one can easily see that the total mass can be viewed as a function of $\delta$ for any fixed $\gamma>1$. In particular, since $\delta>0$ is arbitrary, any positive finite total mass is allowed for the expansion of the gas. 

\begin{remark}\label{R:MRS}
In~\cite{MeRaSz} stable self-similar behavior of solutions to the slightly supercritical Schr\"odinger equation is investigated. There the analogue of the key transformation~\eqref{E:MODIFIEDVELOCITY} leads to a reformulation of the nonlinear Schr\"odinger equation, analogous to~\eqref{E:CONTINUITYNEW}--\eqref{E:VELOCITYNEW}.
In the setting of~\cite{MeRaSz} explicit formula for self-similar solutions is not available, but the new formulation is ideally suited for the study of their existence and stability.
The modified velocity ${\bf U}$ defined by~\eqref{E:MODIFIEDVELOCITY} differs from $\tilde{\bf u}$ by an affine transformation in analogy to the quantity $P_b$ from~\cite{MeRaSz} which arises by a suitable ``conformal" change of variables honoring the symmetries of the nonlinear Schr\"odinger equation.
\end{remark}

%%%%%%%%%%%%%%%%%%%%%%%%%%%%%%%%%%%%%%%%%%

\subsection{Lagrangian formulation}\label{S:LAGR}

Let us fix an element of $\mathscr S$ given by the triple $(\rho_A,\,{\bf u}_A,\, A)$, where $(\rho_A,{\bf u}_A)$ are given by~\eqref{E:RHOAUA}
and $A\in \text{GL}^+(3)$ is a unique solution to the initial value problem~\eqref{E:AEQUATION}--\eqref{E:AEQUATIONINITIAL}.
Note that these solutions are fixed by a choice of free parameters $(A_0,A_1,\delta)$ as described above and it is in one-to-one correspondence with 
solutions to~\eqref{E:CONTINUITYNEW}--\eqref{E:VELOCITYNEW} given by the triple
\be\label{E:SS}
{\bf U}=0, \ \  w( y) = \frac{\delta(\gamma-1)}{2\gamma}(1-| y|^2), \ \ \ B(s) = -A^{-1} A_s,
\ee
where we recall~\eqref{E:TIMESCALE} for the definition of $s$ and the relationship 
$$
\tilde\rho=w^{\alpha}, \ \ \alpha:=\frac1{\gamma-1}.
$$

Upon decomposing
\be\label{E:DECOMPOSITION}
A = \mu O, \ \ \mu=\det A^{\frac13}, \ O\in \text{SL}(3)
\ee
we see that 
\be\label{E:BSIMPLER}
B(s) = -\mu^{-1}O^{-1}(\mu_s O + \mu O_s) = -\frac{\mu_s}{\mu}\,\text{{\bf Id}}  - \Gamma, \ \ \Gamma : =O^{-1} O_s.
\ee

To study the stability of the  steady state given~\eqref{E:SS} we shall reformulate the problem in Lagrangian coordinates. 
Letting 
\[
\Omega : = B_1({\bf{0}})
\]
be the unit ball in $\mathbb R^3$,
we introduce $\eta:\Omega\to\tilde\Omega(s)$ as a flow map associated with the renormalized velocity field ${\bf U}:$
\begin{align}
\eta_s(s, y) &= \mathbf{U}(s,\eta(s, y)), \label{E:FLOWMAPV} \\
\eta(0, y) &= \eta_0( y), \label{E:ETAINITIALV}
\end{align}
where $\eta_0:\Omega\to\tilde\Omega(0)$ is a sufficiently smooth diffeomorphism to be specified later. 
To pull-back~\eqref{E:CONTINUITYNEW}-\eqref{E:VELOCITYNEW} to the fixed domain $\Omega$, we introduce the notation
\begin{align}
\A : = [D\eta]^{-1} \ \ &\text{ (Inverse of the Jacobian matrix)},\\
\J  : = \det [D\eta] \ \ &\text{ (Jacobian determinant)},\\
f : = \tilde\rho\circ \eta \ \ &\text{ (Lagrangian density)}, \\ 
{\bf V} : = {\bf U}\circ \eta \ \ &\text{ (Lagrangian modified velocity)},\\
a : = \J\A \ \ &\text{ (Cofactor matrix)}.
\end{align}
From $\A [D\eta] = \text{{\bf Id}}$, one can obtain the differentiation formula for $\A$ and $\J$: 
\[
\pa \A^k_i = - \A^k_\ell \pa \eta^\ell,_s \A^s_i \  ; \quad \pa \J = \J \A^s_\ell\pa\eta^\ell,_s 
\]
for $\pa=\pa_\tau$ or $\pa=\pa_i$, $i=1,2,3$. Here we have used the Einstein summation convention and the notation $F,_k$ to denote the $k^{th}$ partial  derivative of $F$. Both expressions will be used throughout the paper.

It is well-known~\cite{CoSh2012,JaMa2015} that the continuity equation~\eqref{E:CONTINUITYNEW} reduces to the relationship
\[
f \J = f_0\J_0. 
\]
We choose $\eta_0$ such that  
$ w  = (f_0 \J_0)^{\gamma-1}$ where $ w$ is given in \eqref{E:BARW}. For given initial density function $\rho_0$ so that $\rho_0/\rho_A$ is smooth, the existence of such $\eta_0$ follows from the Dacorogna-Moser theorem \cite{DM}. As a consequence the Lagrangian density can be expressed as
\be \label{E:CONTINUITYLAGRV}
f=  w^\alpha \J^{-1}
\ee
where we recall that $\alpha =\frac1{\gamma-1}.$  We remark that this specific choice of $\eta_0$ (gauge fixing) is important for our analysis; $\eta_0$ being the identity map corresponds to the background expanding solution, $\eta_0$ measures how the initial density is distributed with respect to the background equilibrium, and hence it provides a natural framework for stability analysis around the identity map.  

Using~\eqref{E:BSIMPLER} equation~\eqref{E:VELOCITYNEW} takes the following form in Lagrangian coordinates: 
\be\label{E:VELOCITYLAGR}
\eta_{ss} +  \frac{5-3\gamma}2\frac{\mu_s}{\mu}\eta_s+ 2\Gamma\eta_s+\delta\Lambda \eta + \frac\gamma{\gamma-1}\Lambda \A^\top \nabla(f^{\gamma-1}) = 0
\ee
which can be expressed in the component form~\cite{CoSh2012,JaMa2015}
\be\label{E:ETAPDE1}
 w^\alpha \left(\pa_{ss}\eta_i-\frac{5-3\gamma}2b\pa_s\eta_i+2\Gamma_{ij}\pa_s\eta_j+\delta\Lambda_{ij}\eta_j\right) + \Lambda_{ij}( w^{1+\alpha} \A^k_j\J^{-\frac1\alpha}),_k = 0.
\ee
Since in the $s$-time variable the matrix $B(s)$ blows up in finite time, we introduce a new time variable $\tau$ via 
\be\label{E:TIMESCALE2}
\frac{d\tau}{ds} = \det{A}^{\frac{3\gamma-3}{6}} = \mu^{\frac{3\gamma-3}{2}} \ \text{ or equivalently } \ \frac{d\tau}{dt} = \frac1{\mu}. 
\ee
Since $\mu$ grows linearly in $t$ by results from~\cite{Sideris} (see Lemma~\ref{L:GAMMASTARASYMP}), the new time $\tau$ grows like $\log t$ as $t\to\infty$ and therefore corresponds to 
a logarithmic time-scale with respect to the original time variable $t$.
Equation~\eqref{E:VELOCITYLAGR} takes the form
\be\label{E:VELOCITYLAGR2}
\mu^{3\gamma-4}\left(\mu\eta_{\tau\tau} +  \mu_\tau\eta_\tau+ 2\mu\Gamma^*\eta_\tau \right)+\delta\Lambda \eta + \frac\gamma{\gamma-1} \Lambda\A^\top \nabla(f^{\gamma-1}) = 0,
\ee
where
\[
\Gamma^* =  O^{-1}O_\tau.
\]
In coordinates, for any $i\in\{1,2,3\}$ it reads 
\begin{align}\label{E:ETAPDE2}
 w^\alpha\mu^{3\gamma-4}\left(\mu\pa_{\tau\tau}\eta_i+\mu_\tau\pa_\tau\eta_i+2\mu\Gamma^*_{ij}\pa_\tau\eta_j\right)+\delta  w^\alpha\Lambda_{i\ell}\eta_\ell  + ( w^{1+\alpha} 
\Lambda_{ij}\A^k_j\J^{-\frac1\alpha}),_k = 0.
\end{align} 
Defining the perturbation 
\begin{align} 
\uptheta( \tau, y):= \eta( \tau,y) -  y
\end{align}
equation~\eqref{E:ETAPDE2} takes the form:
\begin{align}
& w^\alpha\mu^{3\gamma-3} \left(\pa_{\tau\tau}\uptheta_i+\frac{\mu_\tau}{\mu}\pa_\tau\uptheta_i+2\Gamma^*_{ij}\pa_\tau\uptheta_j\right) +\delta w^\alpha \Lambda_{i\ell}\uptheta_\ell +  \left( w^{1+\alpha} \Lambda_{ij}\left(\A^k_j\J^{-\frac1\alpha}-\delta^k_j\right)\right),_k = 0, \label{E:THETAEQUATION}
\end{align}
equipped with the initial conditions
\begin{align}
& \uptheta(0,  y) = \uptheta_0(  y), \ \ \uptheta_\tau(0,  y) ={\bf V}(0, y) = {\bf V}_0(  y), \ \ \left( y\in \Omega=B_1({\bf 0})\right).  \label{E:THETAINITIAL}
\end{align}

We will use both $\uptheta_\tau=\pa_\tau\uptheta$ and ${\bf V}$ to denote the velocity field throughout the paper. 

\begin{remark} The Lagrangian formulation~\eqref{E:FLOWMAPV}--\eqref{E:ETAINITIALV} uses the modified velocity 
${\bf U}$ 
to define the flow map $\eta$; this is a coherent strategy as we identified ${\bf U}$ as a natural perturbed velocity around the set $\mathscr S$ of Sideris' affine motions in Eulerian coordinates. 
We remark that the order of this derivation can be reversed. We can start with the Lagrangian formulation of the E$_\gamma$-system using the original velocity ${\bf u}$ and a flow map $\zeta$, wherein the affine motions are identified as solutions of the form $\zeta(t,y) = A(t) y$ as  in \eqref{E:affine}. 
By suitably conjugating the flow map $\zeta$ by $A(t)^{-1}$ and introducing the time rescaling \eqref{E:TIMESCALE2} we can arrive at  \eqref{E:VELOCITYLAGR2} and subsequently~\eqref{E:THETAEQUATION}. See Appendix \ref{A:LAG} for such a derivation.  
However, we have decided on presenting the formulation derived from Eulerian coordinates as it links the symmetries of the E$_\gamma$-system to the existence of affine motions and provides a systematic approach to the question of global existence of solutions in their vicinity without any symmetry assumptions.
\end{remark}

\subsection{Notation}

\noindent
{\bf Lie derivative of the flow map and modified vorticity.}
For vector-fields ${\bf F}:\Omega\to\mathbb R^3$, we introduce the Lie derivatives: full gradient along the flow map $\eta$
\begin{align}
[\nabla_\eta {\bf F}]^i_r := \A^s_r {\bf F}^i,_s 
\end{align}
the divergence
\begin{align}
\text{div}_\eta{\bf F} := \A^s_\ell {\bf F}^\ell,_s,
\end{align}
the modified curl-operator
\begin{align}\label{E:LITTLECURLDEF}
\left[\text{curl}_{\Lambda\A}{\bf F}\right]^i := \epsilon_{ijk}(\Lambda\A)^s_{j}{\bf F}^k,_s = \epsilon_{ijk}\Lambda_{jm}\A^s_m{\bf F}^k,_s,
\end{align}
(where $\epsilon_{ijk}$ is the standard permutation symbol)
and the anti-symmetric curl matrix 
\begin{align}\label{E:BIGCURLDEF}
\left[\text{Curl}_{\Lambda\A}{\bf F}\right]^i_j :=\Lambda_{jm}\A^s_m{\bf F}^i,_s- \Lambda_{im}\A^s_m{\bf F}^j,_s
\end{align}
We will also use $D_{\,}{\bf F}$, $\text{div}_{\,}{\bf F}$, $\text{curl}_{\,}{\bf F}$ to denote its full gradient, its divergence, and its curl:
\[
[D_{\,}{\bf F}]^i_j = {\bf F}^i,_j; \ \ \ \text{{div}}_{\,}{\bf F} ={\bf F}^\ell,_\ell; \ \  \ \left[\text{{curl}} _{\,}{\bf F}\right]^i =  \epsilon_{ijk} {\bf F}^k,_j, \ \ i,j=1,2,3.
\]

\

\noindent
{\bf Function spaces.} For any given integer $k\in\mathbb N$, a function $f:\Omega\to\mathbb R$ and any smooth non-negative function $\varphi:\Omega\to\mathbb R^+$ we introduce the
notation
\begin{align}
\|f\|_{k,\varphi}^2 : = \int_{\Omega}\varphi w^{k} |f(y)|^2\,dy.
\end{align} 
For vector-fields ${\bf F}:\Omega\to\mathbb R^3$ or $2$-tensors ${\bf T}:\Omega\to\mathbb M^{3\times3}$ we analogously define
\begin{align}
\|{\bf F}\|_{k,\varphi}^2 : =\sum_{i=1}^3 \|{\bf F}_i\|_{k,\varphi}^2,  \ \ \|{\bf T}\|_{k,\varphi}^2 : =\sum_{i,j=1}^3 \|{\bf T}_{ij}\|_{k,\varphi}^2.
\end{align}
In this paper we will always choose the weight function $\varphi$ to be either $\psi$ or $1-\psi$ where $\psi: B_1({\bf 0})\to\mathbb [0,1]$ is a smooth cut-off function satisfying 
\be\label{E:PSIDEF}
\psi= 0 \ \text{ on } \, B_{\frac14}({\bf 0}) \ \text{ and } \ \psi =1 \ \text{ on }\, B_1({\bf 0})\setminus B_{\frac34}({\bf 0}).
\ee
It will allow us to localize the estimates to two different regions of the domain $\Omega$.

\

\noindent 
{\bf Angular gradient.}
For any $f\in H^1(\Omega)$ we define the {\em angular gradient} $\slashed\nabla$ as the projection of the usual gradient onto the plane tangent to the sphere of constant $(t,r)$.
In other words
\begin{align}\label{E:TANGENTIAL}
\sn f(y) = \nabla f(y) - \frac{y}{r}(\frac{y}{r}\cdot\nabla)f(y) = \nabla f(y) - \frac yr\pa_r f(y),
\end{align}
where $\pa_r  : = \frac yr\cdot\nabla$ and $r=|y|$.
For $i\in\{1,2,3\}$ we can evaluate the components of $\sn_i$ using~\eqref{E:TANGENTIAL}:
\begin{align}\label{E:ANGGRAD}
\sn_i= \frac{y_j(y_j\pa_i-y_i\pa_j)}{r^2}. 
\end{align}
The angular gradients~\eqref{E:ANGGRAD} are the natural, geometric tangential gradients near the boundary of our domain and they will be crucially used in the analysis. For the commutation rules between $\slashed\nabla$ and other vector fields, see Lemma \ref{L:COMMUTATORS}. 

For any multi-index $\beta = (\beta_1,\beta_2,\beta_3)\in \mathbb Z_{\ge0}^3$ and any $a\in\mathbb Z_{\ge0}$ we introduce the following differential operator
\begin{align}
\der : = \pa_r^a \sn_1^{\beta_1}\sn_2^{\beta_2}\sn_3^{\beta_3}.
\end{align}
Similarly for any multi-index $\nu = (\nu_1,\nu_2,\nu_3)\in \mathbb Z_{\ge0}^3$we denote
\begin{align}
\pa^\nu : = \pa_{y_1}^{\nu_1}\pa_{y_2}^{\nu_2}\pa_{y_3}^{\nu_3}, 
\end{align}
where $\pa_{y_i}$, $i=1,2,3$ denote the usual Cartesian derivatives.

\subsection{Main theorem in  Lagrangian coordinates} \label{S:MAINRESULT}

To state the main theorem in our modified Lagrangian coordinates, we first introduce the high-order weighted Sobolev norm that measures the size of the deviation $\uptheta$. For any $N\in\mathbb N$, let
\begin{align}
& \mathcal S^N(\uptheta,{\bf V})(\tau)  = \mathcal S^N(\tau)  \notag \\ 
&: =   \sum_{a+|\beta|\le N}\sup_{0\le\tau'\le\tau}
\Big\{\mu^{3\gamma-3}\left\|\der{\bf V}\right\|_{a+\alpha,\psi}^2 +\left\|\der{\uptheta}\right\|_{a+\alpha,\psi}^2 
 + \left\|\nabla_\eta\der \uptheta\right\|_{a+\alpha+1,\psi}^2
 +\left\|\text{div}_\eta\der \uptheta\right\|_{a+\alpha+1,\psi}^2 \Big\} \notag \\
&+\sum_{|\nu|\le N}\sup_{0\le\tau'\le\tau}
\Big\{\mu^{3\gamma-3}\left\|\pa^\nu{\bf V}\right\|_{\alpha,1-\psi}^2+\left\|\pa^\nu\uptheta\right\|_{\alpha,1-\psi}^2+\left\|\nabla_\eta\pa^\nu \uptheta\right\|_{\alpha+1,1-\psi}^2
+\left\|\text{div}_\eta\pa^\nu \uptheta\right\|_{\alpha+1,1-\psi}^2\Big\}  \label{E:SNORM} 
\end{align}
Additionally we introduce another high-order quantity measuring the modified vorticity of ${\bf V}$ which is a priori not controlled by the norm $\mathcal S^N(\tau)$: 
\begin{align}
\vortnorm[{\bf V}](\tau) & :=\sum_{a+|\beta|\le N} \sup_{0\le\tau'\le\tau}\left\|\text{Curl}_{\Lambda\A}\der{\bf V}\right\|_{a+\alpha+1,\psi}^2  +\sum_{|\nu|\le N} \sup_{0\le\tau'\le\tau} \left\|\text{Curl}_{\Lambda\A}\pa^\nu{\bf V}\right\|_{\alpha+1,1-\psi}^2.\label{E:VORTNORMDEF}
\end{align}
We also define the quantity $\vortnorm[\uptheta]$ analogously, with $\uptheta$ instead of ${\bf V}$ in the definition~\eqref{E:VORTNORMDEF}.

We are now ready to state our main result. 

%%%%%%%%%%%%%%%%%%%%%%%%%%%%%%%%%%%%%%

\begin{theorem}[Nonlinear stability of affine motions]\label{T:MAINLAGR}
Assume that $\gamma\in(1,\frac53]$ or equivalently $\alpha \in[\frac32,\infty)$. 
 Let $N\geq 2\lceil \alpha \rceil +12$ be fixed. 
Let $(A_0,A_1,\delta)\in\text{{\em GL}}^+(3)\times \mathbb M^{3\times3}\times \mathbb R_+$ be a given triple parametrizing an affine motion $A(\tau)$ from the set $\mathscr S$ and let 
\begin{align}
\mu(\tau) &= (\det A(\tau))^\frac13, \ \ \Lambda(\tau) =\det A(\tau)^{\frac23}A(\tau)^{-1}  A(\tau)^{-\top}, \notag \\
\mu_0 & : = \frac{3\gamma-3}2\mu_1, \ \ \mu_1:=\lim_{\tau\to\infty}\frac{\mu_\tau(\tau)}{\mu(\tau)}.  \label{E:MUZERODEF}
\end{align}
Then there exists an $\varepsilon_0>0$ such that for any $0\le\varepsilon\le\varepsilon_0$ and $(\uptheta_0,{\bf V}_0)$ satisfying $\mathcal S^N(\uptheta_0,{\bf V}_0)+\vortnorm({\bf V}_0)\le\varepsilon$,
there exists a global-in-time solution to the initial value problem~\eqref{E:THETAEQUATION}--\eqref{E:THETAINITIAL} and a constant $C>0$ such that 
\be
\mathcal S^N(\uptheta,{\bf V})(\tau)  \le C \varepsilon, \ \ 0\le\tau<\infty, \label{E:GLOBALBOUND} 
\ee
and
\be
\vortnorm[{\bf V}](\tau)   \le 
\begin{cases}
C \varepsilon e^{-2\mu_0\tau} &\text{if }1<\gamma<\frac53\\
C \varepsilon (1+\tau^2) e^{-2\mu_0\tau} &\text{if } \gamma=\frac53
\end{cases} ,  \ \ 0\le\tau<\infty. \label{E:VORTICITYDECAY}
\ee
Furthermore, there exists a constant $C_*>0$
such that 
\be\label{E:EXPONENTIALDECAY}
\begin{split}
&\sum_{a+|\beta|\le N}\|\der {\bf V}\|_{1+\alpha+a,\psi}
+\sum_{|\nu|\leq N} \|\pa^\nu {\bf V}\|_{1+\alpha,1-\psi}
\le C_*\sqrt{\varepsilon} e^{-\mu_0\tau},
\end{split}
\ee
and a $\tau$-independent asymptotic attractor $\uptheta_\infty$ such that 
\be\label{E:ASYMP}
\lim_{\tau\to\infty}\left( \sum_{a+|\beta|\le N}\left\|\der\uptheta(\tau,\cdot)-\der\uptheta_\infty(\cdot)\right\|_{a+\alpha,\psi}+\sum_{|\nu|\le N}\left\|\pa^\nu\uptheta(\tau,\cdot)-\pa^\nu\uptheta_\infty(\cdot)\right\|_{\alpha,1-\psi}\right) = 0.
\ee
\end{theorem}

%%%%%%%%%%%%%%%%%%%%%%%%%%%%%%%%%%%%%%

\begin{remark}
The exponential-in-$\tau$ decay stated in~\eqref{E:EXPONENTIALDECAY} implies an algebraic decay statement in $t$-variable with an upper bound 
$\frac{C_*}{(1+t)^{\mu_0}}$ for some $C_*>0$. Note that the larger $\gamma\in(1,\frac53]$ is, the faster the equilibration rate is. The same applies to the vorticity decay in \eqref{E:VORTICITYDECAY} except for $\gamma=\frac53$ where we get $\frac{\ln (1+t)}{(1+t)}$ instead of $\frac{1}{1+t}$. 
\end{remark}

\begin{remark} 
The assumption $N\geq 2\lceil \alpha \rceil +12=2\lceil \frac1{\gamma-1} \rceil +12$ is not an optimal lower bound for the number of derivatives necessary to prove Theorem~\ref{T:MAINLAGR}. 
We have not tried to find an optimal $N$, as we  
are primarily interested in displaying the phenomenology and developing a robust approach to the construction of global-in-time solutions to the E$_\gamma$-system. We do observe that the larger $\gamma\in(1,\frac53]$ is, the less smoothness is required in our norm.
\end{remark}

\begin{remark}[Constants $\mu_0$ and $\mu_1$]
As a simple consequence of the ODE analysis in~\cite{Sideris}, the limit defining $\mu_1$ in the statement of the theorem is well-defined,
see Lemma~\ref{L:GAMMASTARASYMP} for more details.  
Note that by definition~\eqref{E:SNORM} of $\norm$, we have that 
\begin{align}
\sum_{a+|\beta|\le N}\|\der {\bf V}\|_{1+\alpha+a,\psi}
+\sum_{|\nu|\leq N} \|\pa^\nu {\bf V}\|_{1+\alpha,1-\psi}
\lesssim e^{-\mu_0\tau}\norm(\tau)^{\frac12}. \label{E:DECAYBOUND}
\end{align}
Bound~\eqref{E:DECAYBOUND} will be commonly used in our analysis. We also observe that for any $\gamma\in(1,\frac53]$ the inequality $0<\mu_0=\mu_0(\gamma)\le\mu_1$ holds with
equality if and only if $\gamma=\frac53$.
\end{remark}

\begin{remark}[Initial density $\rho_0$]
Note that the density $\rho_0$ of the initial gas configuration relates to the density of the background affine motion via
the composition
\[
\rho_0(x) = \rho_A((\zeta_A(0)\circ\eta_0)^{-1}(x))\det[D(\zeta_A(0)\circ\eta_0)]^{-1},
\]
where the transformation $\eta_0:\Omega\to\Omega(0)$ is very close to the identity map and $\zeta_A$ is the flow map associated with the affine motion~\eqref{E:affine}.
\end{remark}
%%%%%%%%%%%%%%%%%%%%%%%%%%%%%%%%%%%%%%

\subsubsection*{Local well-posedness}

The local-well-posedness of~\eqref{E:THETAEQUATION}--\eqref{E:THETAINITIAL} follows from an adaption of the argument in~\cite{JaMa2015}. In fact, the method developed  in~\cite{JaMa2015}  ensures the well-posedness upon having a good a priori estimates for the full energy and the curl energy of the Euler equations via a duality argument. Although \eqref{E:THETAEQUATION} is different from the Euler equations due to the presence of the time weights such as $\Lambda$ (note that these weights are prescribed functions independent of the solutions), the estimates in Section \ref{S:VORTICITY} and Section \ref{S:ENERGY}, which will be obtained through several new ideas, give rise to the desired bounds on the energy functionals $\norm$ \eqref{E:SNORM} and $\vortnorm$ \eqref{E:VORTNORMDEF} that are sufficient to apply the methodology of \cite{JaMa2015} by using $\norm$ and  $\vortnorm$:  

\begin{theorem}[Local well-posedness \cite{JaMa2015}]\label{T:LWP} Let $N\geq 2\lceil \alpha \rceil +12$ be given. Then for given initial data $(\uptheta_0,{\bf V}_0)$ with $\norm(\uptheta_0, {\bf V}_0) + \vortnorm({\bf V}_0)<\infty$, there exists a $T>0$ and a unique solution $(\uptheta(\tau), {\bf V}(\tau)):\Omega \rightarrow \mathbb R^3\times \mathbb R^3$ for each $\tau\in [0,T]$ to~\eqref{E:THETAEQUATION}--\eqref{E:THETAINITIAL} such that $\norm (\uptheta, {\bf V})(\tau) + \vortnorm[{\bf V}](\tau) \leq 2(\norm(\uptheta_0, {\bf V}_0) + \vortnorm({\bf V}_0))$ for all $\tau\in[0,T]$. Moreover, the map $[0,T]\ni\tau\mapsto\norm(\tau)\in\mathbb R_+$ is continuous.
\end{theorem}

For the rest of the article, we will use this well-posedness result to guarantee the existence of a unique solution $(\uptheta, {\bf V}):\Omega \rightarrow \mathbb R^3\times \mathbb R^3$ with $\norm(\uptheta, {\bf V}) +\vortnorm[{\bf V}] <\infty$ on a finite interval $[0,T]$ where $T>0$ is fixed.  

\subsubsection*{A priori assumptions}

It is convenient to make the following a priori assumptions:  
\be\label{E:assumption}
\norm(\tau) < \frac 13
\ee
as well as 
\be
\label{E:APRIORI}
\|\A - \text{{\bf Id}}\|_{W^{1,\infty}(\Omega)}<\frac13, \quad 
\|D\uptheta\|_{W^{1,\infty}(\Omega)}<\frac13,\quad 
\|\mathscr J - 1\|_{W^{1,\infty}(\Omega)}<\frac13 
\ee 
for each $\tau\in [0,T]$ where $T>0$ is fixed.  The latter assumption ensures the non-degeneracy of the flow map. 
For an $\varepsilon_0\ll1$ in Theorem~\ref{T:MAINLAGR} the assumptions \eqref{E:assumption} and \eqref{E:APRIORI}  are initially true, and by the local well-posedness theory they remain true on some short time interval. By means of a continuity argument we will show in the proof of Theorem~\ref{T:MAINLAGR} (Section~\ref{S:MAINPROOF}) that our a-priori assumptions are in fact {\em improved}, thereby justifying them in the first place.

%%%%%%%%%%%%%%%%%%%%%%%%%%%%%%%%%%%%%%
%%%%%%%%%%%%%%%%%%%%%%%%%%%%%%%%%%%%%%

\section{Vorticity estimates}\label{S:VORTICITY}

In this section we exploit an effective decoupling in the modified Euler system~\eqref{E:THETAEQUATION} which allows us to 
bound the modified vorticity tensors $\text{Curl}_{\Lambda\A}{\bf V}$ and $\text{Curl}_{\Lambda\A}\uptheta$ in terms of the high-order norm $\norm$.
The main result of this section is given in Proposition~\ref{P:VORTICITYBOUNDS}, but to achieve it we first prove a lemma which demonstrates the above mentioned decoupling 
and gives us a simplified formula for the modified vorticity.

\begin{lemma} 
Let $(\uptheta,{\bf V}):\Omega\to\mathbb R^3\times \mathbb R^3$ be a unique solution to~\eqref{E:THETAEQUATION}--\eqref{E:THETAINITIAL} defined on some time interval $[0,T]$.
For any $\tau\in[0,T]$ the modified vorticity tensors $\text{{\em Curl}}_{\Lambda\A}{\bf V}$ and $\text{{\em Curl}}_{\Lambda\A}\uptheta$ satisfy 
\begin{align}
\text{{\em Curl}}_{\Lambda\A}{\bf V} (\tau)&= \frac{\mu(0)\text{{\em Curl}}_{\Lambda\A}{\bf V} (0)}{\mu} \notag \\
&+\frac{1}{\mu} \int_0^\tau \mu [\pa_\tau, \text{{\em Curl}}_{\Lambda\A}] {\bf V}d\tau'
-\frac{2}{\mu}\int_0^\tau \mu\text{{\em Curl}}_{\Lambda\A}\left(\Gamma^\ast{\bf V}\right)d\tau' \label{E:CURL3}
\end{align}
and 
\begin{align}
\text{{\em Curl}}_{\Lambda\A}\uptheta (\tau) & = \text{{\em Curl}}_{\Lambda\A}\uptheta (0)+\mu(0)\text{{\em Curl}}_{\Lambda\A}{\bf V} (0) \int_0^\tau \frac1{\mu(\tau')}\,d\tau' 
+ \int_0^\tau[\pa_\tau, \text{{\em Curl}}_{\Lambda\A}] \uptheta(\tau')\,d\tau' \notag \\
&+\int_0^\tau\frac{1}{\mu(\tau')} \int_0^{\tau'} \mu(\tau'') [\pa_\tau, \text{{\em Curl}}_{\Lambda\A}] {\bf V} \, d\tau''\,d\tau'
-\int_0^\tau\frac{2}{\mu(\tau')}\int_0^{\tau'} \mu(\tau'')\text{{\em Curl}}_{\Lambda\A}\left(\Gamma^\ast{\bf V}\right)\, d\tau''\,d\tau' \label{E:CURLTHETA}
\end{align}
\end{lemma}

\begin{proof}
Recalling that 
$
{\bf V} = \uptheta_\tau,
$
 equation~\eqref{E:VELOCITYLAGR2} is best rewritten in the form
\begin{align}\label{E:CURL0}
\mu^{3\gamma-3}\left({\bf V}_\tau + \frac{\mu_\tau}{\mu}{\bf V}+2\Gamma^\ast {\bf V}\right) 
+ \frac\gamma{\gamma-1}\Lambda\nabla_\eta(f^{\gamma-1}) +\delta\Lambda{\bf \eta}=0
\end{align}
Recalling the definition \eqref{E:BIGCURLDEF}, 
apply $\text{Curl}_{\Lambda\A}$ to~\eqref{E:CURL0} and obtain
\begin{align}\label{E:CURL1}
\mu^{3\gamma-3}\left(\text{Curl}_{\Lambda\A}{\bf V}_\tau + \frac{\mu_\tau}{\mu}\text{Curl}_{\Lambda\A}{\bf V}+2\text{Curl}_{\Lambda\A}\left(\Gamma^\ast{\bf V}\right) \right)
 +\delta\text{Curl}_{\Lambda\A}\Lambda{\bf \eta}=0
\end{align}
Note that 
\[
\text{Curl}_{\Lambda\A}\Lambda{\bf \eta} = 0
\]
and therefore~\eqref{E:CURL1} gives
\be\label{E:CURLBASIC}
\text{Curl}_{\Lambda\A}{\bf V}_\tau + \frac{\mu_\tau}{\mu}\text{Curl}_{\Lambda\A}{\bf V}+2\text{Curl}_{\Lambda\A}\left(\Gamma^\ast{\bf V}\right)=0
\ee

To obtain the identities~\eqref{E:CURL3} and~\eqref{E:CURLTHETA} we shall integrate~\eqref{E:CURLBASIC} with respect to $\tau$. 
We rewrite the first term as
\be
\text{Curl}_{\Lambda\A} {\bf V}_\tau  =\pa_\tau\left( \text{Curl}_{\Lambda\A}{\bf V}\right) - [\pa_\tau, \text{Curl}_{\Lambda\A}] {\bf V}
\ee
where 
\be\label{E:TAUCURLCOMMUTATOR}
[\pa_\tau, \text{Curl}_{\Lambda\A}] F^k_j = \pa_\tau \left(\Lambda_{jm}\A^s_m\right) F,_s^k - \pa_\tau \left(\Lambda_{km}\A^s_m\right) F,_s^j 
\ee
We may therefore rewrite~\eqref{E:CURLBASIC} in the form 
\be\label{E:CURL3.0}
\pa_\tau \left(\mu  \text{Curl}_{\Lambda\A}{\bf V}   \right) = \mu [\pa_\tau, \text{Curl}_{\Lambda\A}] {\bf V}  - 2 \mu \text{Curl}_{\Lambda\A}\left(\Gamma^\ast{\bf V}\right).
\ee
Integrating \eqref{E:CURL3.0} in $\tau$ we obtain~\eqref{E:CURL3}.
Using the commutator formula again 
we may write~\eqref{E:CURL3} in the form
\begin{align}
\pa_\tau\text{Curl}_{\Lambda\A}\uptheta (\tau) & = \frac{\mu(0)\text{Curl}_{\Lambda\A}{\bf V} (0)}{\mu} + [\pa_\tau, \text{Curl}_{\Lambda\A}] \uptheta \notag \\
&+\frac{1}{\mu} \int_0^\tau \mu [\pa_\tau, \text{Curl}_{\Lambda\A}] {\bf V}d\tau'
-\frac{2}{\mu}\int_0^\tau \mu\text{Curl}_{\Lambda\A}\left(\Gamma^\ast{\bf V}\right)d\tau'
\end{align}
and therefore~\eqref{E:CURLTHETA} follows by integrating the previous identity in $\tau$.
\end{proof}

\begin{lemma}\label{L:COMMUTATORBOUNDS}
Let $(\uptheta,{\bf V}):\Omega\to\mathbb R^3\times \mathbb R^3$ be a unique solution to~\eqref{E:THETAEQUATION}--\eqref{E:THETAINITIAL} defined on some time interval $[0,T]$
and satisfying the a priori assumptions~\eqref{E:assumption}--\eqref{E:APRIORI}.
Then for any $\tau\in [0,T]$ 
\begin{align}
&\sum_{a+|\beta|\le N} \Big\|\left[\der,\text{{\em Curl}}_{\Lambda\A}\right]{\bf V}(\tau)\Big\|_{1+\alpha+a,\psi}^2
+ \sum_{|\nu|\le N} \Big\|\left[\car,\text{{\em Curl}}_{\Lambda\A}\right]{\bf V}(\tau)\Big\|_{1+\alpha,1-\psi}^2 \notag  \\
& \ \ \ \ \  \lesssim e^{-2\mu_0\tau}\norm(\tau)\left(1+P(\norm(\tau))\right) \label{E:COMMUTATORBOUND1}  \\
&\sum_{a+|\beta|\le N} \Big\|\left[\der,\text{{\em Curl}}_{\Lambda\A}\right]\uptheta(\tau) \Big\|_{1+\alpha+a,\psi}^2
+ \sum_{|\nu|\le N} \Big\|\left[\car,\text{{\em Curl}}_{\Lambda\A}\right]\uptheta(\tau)\Big\|_{1+\alpha,1-\psi}^2 \notag  \\
& \ \ \ \ \  \lesssim \norm(0)+ \kappa\norm(\tau) +(1+P(\norm(\tau)))\int_0^\tau e^{-\mu_0\tau'} \norm(\tau')\,d\tau'  \label{E:COMMUTATORBOUND2}
\end{align}
where the polynomial $P$ is of degree at least 1, $0<\kappa\ll 1$ is a small constant, and $\mu_0$ is defined in~\eqref{E:MUZERODEF}.
\end{lemma}

\begin{proof}
{\em Proof of~\eqref{E:COMMUTATORBOUND1}.}
Recall that 
\begin{align}
&\left[\der,\text{{Curl}}_{\Lambda\A}\right]{\bf V}^k_i
 =  \der\left(\text{{Curl}}_{\Lambda\A}{\bf V}^k_i\right)-\text{{Curl}}_{\Lambda\A}(\der {\bf V})^k_i  \notag \\
& =\der\left(\Lambda_{im}\A^s_m{\bf V}^k,_s- \Lambda_{km}\A^s_m{\bf V}^i,_s\right)- \left(\Lambda_{im}\A^s_m(\der {\bf V})^k,_s- \Lambda_{km}\A^s_m(\der {\bf V})^i,_s \right)   \notag\\
&  =\Lambda_{km}\left(\A^s_m(\der {\bf V})^i,_s  - \der(\A^s_m{\bf V})^i,_s)\right)- \Lambda_{im}\left(\A^s_m(\der {\bf V})^k,_s - \der(\A^s_m{\bf V}^k,_s)\right)  \label{E:VCOMM}
\end{align}
As the two summands appearing on the right-most side above behave identically from the point-of-view of derivative count, we will restrict our attention only to the second term.
Applying the Leibniz rule,
\begin{align}
&\A^s_m(\der {\bf V})^k,_s - \der(\A^s_m{\bf V}^k,_s)  = - \underbrace{ \der\A^s_m {\bf V}^k,_s}_{=:A} \notag \\
 & - \sum_{1\le a'+|\beta'|\le N-1\atop a'\le a, \beta'\le\beta}\pa_r^{a'}\slashed\nabla^{\beta'}\A^s_m\pa_r^{a-a'}\slashed\nabla^{\beta-\beta'}{\bf V}^k,_s - 
\underbrace{\A^s_m [\der,\pa_s]{\bf V}^k}_{=:B} \label{E:CURLCOMM1}
\end{align}
where $[\der,\pa_s]$ stands for the commutator defined in~\eqref{E:DIFFCOMMUTATOR}.
Terms A and B are the two highest-order terms on the right-hand side of~\eqref{E:CURLCOMM1}.
To estimate A, we use formulas~\eqref{E:TOPORDER1} (when $\beta=(0,0,0)$) and otherwise~\eqref{E:TOPORDER2}. Assume without loss of generality that
$\beta\neq (0,0,0)$ since the other case can be done analogously. By~\eqref{E:TOPORDER2} for some $\ell\in\{1,2,3\}$ we have 
\begin{align}
A & = - \underbrace{\A^k_j(\der \uptheta^j),_m\A^m_i{\bf V}^k,_s}_{=:A_1} -  \A^k_r[\pa_r^a\sn^{\beta-e_\ell},\pa_m]\sn_\ell\uptheta^s\A^m_i \notag \\
& \ \ \ -{\bf V}^k,_s\Big(\sum_{0<a'+|\beta'|\le a+|\beta| \atop \beta'\le\beta-e_\ell }c_{a'\beta'}\pa_r^{a'}\slashed\nabla^{\beta'}(\A^k_j\A^m_i)\pa_r^{a-a'}\slashed\nabla^{\beta-e_\ell-\beta'}\left(\sn_\ell\uptheta^j\right),_m +\pa_r^a\slashed\nabla^{\beta-e_\ell}\left( \A^k_j[\pa_m,\sn_\ell]\uptheta^j\A^m_i \right) \Big). \label{E:ATERM}
\end{align}
Note that
\begin{align}
\|A_1\|_{1+\alpha+a,\psi}^2& \lesssim \|D{\bf V}\|_{L^\infty(\Omega)}^2\|\nabla_\eta\der\uptheta\|_{1+\alpha+a,\psi}^2 \notag \\
& \lesssim e^{-2\mu_0\tau} (\norm(\tau))^2\left(1+P(\norm(\tau))\right),
\end{align} 
where we have used~\eqref{embedding2} to bound $\|D{\bf V}\|_{L^\infty(\Omega)}$ by a constant multiple of $e^{-\mu_0\tau}\sqrt{\norm(\tau)}$ and the a priori bounds~\eqref{E:APRIORI}.
The remaining terms on the right-hand side of~\eqref{E:ATERM} are estimated using the standard Moser type estimates in conjunction with Proposition~\ref{P:EMBEDDING} and formula~\eqref{E:DIFFCOMMUTATOR}. 
The estimates for term $B$ and the remaining terms in~\eqref{E:CURLCOMM1} rely on the formula~\eqref{E:DIFFCOMMUTATOR} and Proposition~\ref{P:EMBEDDING}.
Namely, recalling~\eqref{E:DIFFCOMMUTATOR} we have the following schematic formula for $B:$
\begin{align*}
B =  \A^s_m \sum_{\ell=0}^{a+|\beta|}\sum_{a'+|\beta'|=\ell  \atop a'\leq a+1} C^s_{a',\beta',\ell}\pa_r^{a'}\sn^{\beta'} {\bf V}^k,
\end{align*}
for some universal functions $C_{a',\beta',\ell}^s$, smooth away from the origin. Therefore
\begin{align*}
\|B\|_{1+\alpha+a,\psi}^2 \lesssim e^{-2\mu_0\tau} \norm(\tau).
\end{align*}
Summing the above estimates we conclude that all the $\|\cdot\|_{1+\alpha+a,\psi}$-norms on the left-hand side of~\eqref{E:COMMUTATORBOUND1} are bounded by the right-hand side.
An entirely analogous strategy can be applied to $\|\cdot\|_{1+\alpha,1-\psi}$-norms to conclude the proof. \\

\noindent
{\em Proof of~\eqref{E:COMMUTATORBOUND2}.}
Using the formula~\eqref{E:VCOMM} with $\uptheta$ instead of ${\bf V}$, in analogy to~\eqref{E:CURLCOMM1} we reduce the problem to estimating the $\|\cdot\|_{1+\alpha+a,\psi}$-norm of the expression
\begin{align}
&\A^s_m(\der \uptheta)^k,_s - \der(\A^s_m\uptheta^k,_s)  = - \underbrace{ \der\A^s_m \uptheta^k,_s}_{=:A_\uptheta} \notag \\
 & - \sum_{1\le a'+|\beta'|\le N-1\atop a'\le a, \beta'\le\beta}\pa_r^{a'}\slashed\nabla^{\beta'}\A^s_m\pa_r^{a-a'}\slashed\nabla^{\beta-\beta'}\uptheta^k,_s - 
\underbrace{\A^s_m [\der,\pa_s]\uptheta^k}_{=:B_\uptheta}, \label{E:CURLCOMM2}
\end{align}

We illustrate the proof by estimating the term $A_\uptheta$. Estimating $\uptheta^k,_s$ in $L^\infty$ norm and using the fundamental theorem of calculus we obtain
\begin{align}
\|A_\uptheta\|_{1+\alpha+a,\psi}^2 & \lesssim \|\der\A\|_{1+\alpha+a,\psi}^2 \| D\uptheta\|_{L^\infty(\Omega)}^2 \notag \\
& \lesssim  \|\der\A\|_{1+\alpha+a,\psi}^2 \left(\|D\uptheta_0\|_{L^\infty(\Omega)} + \int_0^\tau \| D{\bf V}\|_{L^\infty(\Omega)}\right)^2 \notag \\
& \lesssim \norm(\tau) \left(\norm(0) + \left(\int_0^\tau e^{-\mu_0\tau'}\sqrt{\norm(\tau')}\,d\tau'\right)^2\right) \notag \\
&\lesssim  \norm(\tau) \left(\norm(0) + \int_0^\tau e^{-\mu_0\tau'}\norm(\tau')\,d\tau'\right), \notag
\end{align}
where we have used the Hardy-Sobolev embedding $\| D{\bf V}\|_{L^\infty(\Omega)} \lesssim e^{-\mu_0\tau}\sqrt{\norm(\tau)}$ in the second-to-last estimate, and  
the Cauchy-Schwarz inequality in the last line.
We have also used the already established fact $\|\der\A\|_{1+\alpha+a,\psi}^2\lesssim \norm(\tau)$ which has been explained in the proof of~\eqref{E:COMMUTATORBOUND1}.
We leave out the details for the remaining terms in~\eqref{E:CURLCOMM2} as they follow similarly.  
\end{proof}

\begin{proposition}[The vorticity estimates]\label{P:VORTICITYBOUNDS}
Let $(\uptheta,{\bf V}):\Omega\to\mathbb R^3\times \mathbb R^3$ be a unique solution to~\eqref{E:THETAEQUATION}--\eqref{E:THETAINITIAL} defined on some time interval $[0,T]$
and satisfying the a priori assumptions~\eqref{E:assumption}--\eqref{E:APRIORI}. 
Then for any $\tau\in [0,T]$ 
\begin{align}
&\vortnorm[{\bf V}](\tau) \lesssim 
\begin{cases}
 e^{-2\mu_0\tau}\left(\norm(0)+\vortnorm(0)\right)+ e^{-2\mu_0\tau}\norm(\tau)   & \text{if } \ 1<\gamma<\frac53 \\ 
e^{-2\mu_0\tau}\left(\norm(0)+\vortnorm(0)\right)+ (1+\tau^2)e^{-2\mu_0\tau}\norm(\tau) & \text{if } \ \gamma =\frac53
\end{cases}.
\label{E:CURLVBOUND} \\
&
\vortnorm[\uptheta](\tau) \lesssim  \norm(0)+\vortnorm(0) + \kappa\norm(\tau) + \int_0^\tau e^{-\mu_0\tau'} \norm(\tau')\,d\tau' \label{E:CURLTHETABOUND}
\end{align}
where 
$0<\kappa\ll 1$ is a small constant, $\mu_0$ is defined in~\eqref{E:MUZERODEF}, and $\vortnorm(\bf V)$, $\vortnorm[\uptheta]$ are defined in~\eqref{E:VORTNORMDEF}.
\end{proposition}

\begin{proof}
{\em Proof of~\eqref{E:CURLVBOUND}.}
Applying $\der$ to~\eqref{E:CURL3} we obtain   
\begin{align}
&\text{Curl}_{\Lambda\A}\der{\bf V} = \frac{\mu(0)\der\text{Curl}_{\Lambda\A}{\bf V} (0)}{\mu} + \left[\der,\text{Curl}_{\Lambda\A}\right]{\bf V} \notag \\
&+\frac{1}{\mu} \int_0^\tau \mu \der[\pa_\tau, \text{Curl}_{\Lambda\A}] {\bf V}d\tau'
-\frac{2}{\mu}\int_0^\tau \mu\der\text{Curl}_{\Lambda\A}\left(\Gamma^\ast{\bf V}\right)d\tau'. \label{E:CURL3COMM}
\end{align}
We shall exploit the exponential-in-$\tau$ growth of $\mu$ to show that for any pair of indices $a+|\beta|\le N$ the left-hand side satisfies a ``good" bound in $L^2$-sense.
Note that by definition~\eqref{E:SNORM} of $\norm$ it immediately follows that
\be\label{E:CURLEST0}
\left\|\frac{\mu(0)\der\text{Curl}_{\Lambda\A}{\bf V} (0)}{\mu} \right\|_{1+\alpha+a,\psi}^2 \lesssim e^{-2\mu_1\tau} (\norm(0)+\vortnorm(0)),
\ee
where we used the exponential growth~\eqref{E:MUASYMP0} of $\mu$.
The second term on the right-hand side of~\eqref{E:CURL3COMM} is bounded by Lemma~\ref{L:COMMUTATORBOUNDS}.
Concerning the third term on the right-hand side of~\eqref{E:CURL3COMM}, by~\eqref{E:TAUCURLCOMMUTATOR} we have 
\begin{align*}
\der[\pa_\tau, \text{Curl}_{\Lambda\A}] {\bf V}_j^k  =\der \left(\pa_\tau \left(\Lambda_{jm}\A^s_m\right) {\bf V},_s^k - \pa_\tau \left(\Lambda_{km}\A^s_m\right) {\bf V},_s^j \right) .
\end{align*}
Schematically speaking both terms inside the parenthesis on the right-hand side behave the same way from the point of view of derivative count. It is therefore sufficient to restrict our attention to one term only.
Note that 
\begin{align}
&\der \left(\pa_\tau \left(\Lambda_{jm}\A^s_m\right) {\bf V},_s^k\right) \notag \\
& = \der\left[\left(\pa_\tau\Lambda_{jm}\A^s_m + \Lambda_{jm}\pa_\tau \A^s_m \right){\bf V},_s^k\right] \notag \\
& = \pa_\tau\Lambda_{jm} \left(\der\A^s_m {\bf V},_s^k +\A^s_m\der{\bf V},_s^k\right) +  \Lambda_{jm}\left(\der\pa_\tau\A^s_m {\bf V},_s^k +\pa_\tau\A^s_m\der{\bf V},_s^k\right) \notag \\
& \ \ \ \ + \sum_{1\le a'+|\beta'|\le N-1} C_{a',\beta'}\left(\pa_\tau\Lambda_{jm}\pa_r^{a'}\sn^{\beta'}\A^s_m +\Lambda_{jm} \pa_r^{a'}\sn^{\beta'}\pa_\tau\A^s_m\right) 
 \pa_r^{a-a'}\sn^{\beta-\beta'} {\bf V}^k,_s
\label{E:CURLERROREXPANSION}
\end{align}
From the point of view of derivative count we may schematically rewrite the first two terms on the right-most side in the form
\be\label{E:CURLERROR1}
\underbrace{\pa_\tau\Lambda \,\der D\uptheta D{\bf V}}_{=:C_1} + \underbrace{\pa_\tau\Lambda  D\eta \, \der D{\bf V}}_{=:C_2} + \underbrace{\Lambda D{\bf V} \, \der D{\bf V}}_{=:C_3} 
\ee
For the term $C_1$, observe that 
\begin{align}
& \int_\Omega \psi w^{1+\alpha + a}\frac1{\mu^{2}}\Big|\int_0^\tau \mu \pa_\tau\Lambda \,\der D\uptheta D{\bf V} \,d\tau'\Big|^2\,dx \notag \\
& \lesssim e^{-2\mu_1\tau} \int_\Omega \psi w^{1+\alpha + a}\Big|\int_0^\tau \left|e^{\mu_1\tau'} \pa_\tau\Lambda \,\der D\uptheta D{\bf V} \right|\,d\tau'\Big|^2\,dx  \notag \\
& \lesssim e^{-2\mu_1\tau} \sup_{0\le\tau'\le\tau}\int_\Omega\psi w^{1+\alpha + a} |\der D\uptheta|^2 \Big|\int_0^\tau| e^{\mu_1\tau'}\left|\pa_\tau\Lambda D{\bf V} \right|\,d\tau'\Big|^2 \,dx \notag \\
& \lesssim  e^{-2\mu_1\tau}  \sup_{0\le\tau'\le\tau}\left( \|\der D\uptheta\|_{1+\alpha+a,\psi}^2\right) \Big| \int_0^\tau e^{\mu_1\tau'}\left|\pa_\tau\Lambda\right| e^{-\mu_0\tau'}  (\norm(\tau'))^{\frac12} d\tau'  \Big|^2  \notag \\
& \lesssim e^{-2\mu_1\tau} ( \norm(\tau))^2 \label{E:CURLEST1}
\end{align}
where we used Proposition~\ref{P:EMBEDDING} to bound $\|D{\bf V}\|_{L^\infty(\Omega)}$ by $e^{-\mu_0\tau'} (\norm(\tau'))^{\frac12}$, the sharp exponential behavior~\eqref{E:MUASYMP0} of $\mu(\tau)$, and the exponential decay~\eqref{E:LAMBDABOUNDS} of $|\Lambda_\tau|$.

To estimate the $\|\cdot\|_{1+\alpha+a,\psi}$-norm of term $C_2$ in~\eqref{E:CURLERROR1} we first note that the expression $\der D {\bf V}$ contains one too many derivative to belong to our energy space. To resolve this difficulty,
we must integrate-by-parts in $\tau$ first.
Observe that 
\begin{align}
\frac1\mu\int_0^\tau \mu \pa_\tau\Lambda  D\eta \, \der D{\bf V} d\tau'
= & \frac1\mu\left(\mu \pa_\tau\Lambda  D\eta \, \der D \uptheta\right)\big|^\tau_0 \label{tau2}\\ 
&  - \frac1\mu\int_0^\tau \mu(\frac{\mu_\tau}\mu\pa_\tau\Lambda D\eta+\pa_{\tau\tau}\Lambda  D\eta+\pa_\tau\Lambda D{\bf V}) \, \der D \uptheta\,d\tau'. \notag
\end{align}
Applying now the same line of reasoning as in the case of term $C_1$, using thereby the exponential decay of $\pa_\tau \Lambda$ and $\pa_{\tau\tau}\Lambda$ as stated in Lemma~\ref{L:GAMMASTARASYMP} we arrive at 
\begin{align}
\|\frac1{\mu}\int_0^\tau\mu\pa_\tau\Lambda  D\eta \, \der D{\bf V} d\tau' \|_{1+\alpha+a,\psi}^2  \lesssim e^{-2\mu_1\tau}\norm(0) + \left((1+\tau^2) e^{-2\mu_1\tau} + Q(\tau)\right)  \norm(\tau) 
\label{E:NOTCRITICAL}
\end{align}
where 
\[
Q(\tau)=
\begin{cases}
e^{-4\mu_0\tau} &\text{if } \ 1<\gamma<\frac43  \\
\tau^2e^{-4\mu_0\tau} &\text{if } \ \gamma=\frac43\\
e^{-2\mu_1\tau} &\text{if } \ \frac43<\gamma\le\frac53
\end{cases}. 
\]
The first $\tau^2$ on the right-hand side of \eqref{E:NOTCRITICAL} comes from estimating $|\int_0^\tau \mu_\tau \pa_\tau \Lambda d\tau' |^2$ and $Q(\tau)$ from estimating $|\int_0^\tau \mu \pa_{\tau\tau}\Lambda d\tau' |^2$ corresponding to the first two terms in the $\tau$ integral in the right-hand side of \eqref{tau2}. 

The estimate for the error term $C_3$ from~\eqref{E:CURLERROR1} is more subtle. 
Quantity $|\Lambda|$ does {\em not} decay exponentially but is merely bounded. 
Therefore when
estimating $\frac{1}{\mu(\tau)}\int_0^\tau \mu(\tau') C_3(\tau')\,d\tau'$ in $L^\infty_\tau$-norm, we must proceed with more caution. As in the above, since the expression $\der D {\bf V}$ contains one too many derivative to belong to our energy space, we integrate-by-parts.
As a result
\begin{align}
\frac1\mu\int_0^\tau \mu \Lambda  D{\bf V} \, \der D{\bf V} d\tau'
= & \frac1\mu \left[\mu \Lambda  D{\bf V} \, \der D \uptheta\right]^\tau_0 \notag \\ 
&  - \frac1\mu\int_0^\tau \mu(\frac{\mu_\tau}\mu\Lambda D{\bf V}+\pa_{\tau}\Lambda  D{\bf V}+\Lambda D{\bf V}_\tau) \, \der D \uptheta\,d\tau' \label{E:CRITICAL}.
\end{align}
The $\|\cdot\|_{1+\alpha+a,\psi}$-norm of the first term on the right-hand side of~\eqref{E:CRITICAL} is obviously bounded by 
\[
\mu^{-1} \norm(0) + \|D{\bf V}\|_{L^\infty(\Omega)}(\norm(\tau))^{\frac12} \lesssim e^{-\mu_1\tau} \norm(0)+ e^{-\mu_0\tau}\norm(\tau)
\]
where we used the Hardy-Sobolev embeddings of Appendix~\ref{A:HARDY} to obtain $\|D{\bf V}\|_{L^\infty(\Omega)}\lesssim e^{-\mu_0\tau}(\norm(\tau))^{\frac12}$.
By the sharp exponential behavior~\eqref{E:MUASYMP0} of $\mu$ we have 
\begin{align*}
\Big|\frac1\mu\int_0^\tau \mu_\tau\Lambda D{\bf V}  \der D \uptheta\,d\tau'\Big| & \lesssim \sup_{0\le\tau'\le\tau}| \der D \uptheta| e^{-\mu_1\tau}\int_0^\tau e^{\mu_1\tau'}\|D{\bf V}\|_{L^\infty(\Omega)}\,d\tau' \notag \\
 & \lesssim  \sup_{0\le\tau'\le\tau}|\der D \uptheta|(\norm(\tau))^{\frac12}\underbrace{ e^{-\mu_1\tau}\int_0^t e^{(\mu_1-\mu_0)\tau'} d\tau' }_{=: L} \notag 
\end{align*}
Note that 
\[
L \lesssim 
\begin{cases}
e^{-\mu_0\tau} & \text{ if } \  1<\gamma<\frac53 \\
\tau e^{-\mu_0\tau} & \text{ if } \ \gamma =\frac53
\end{cases}
\]
because $\mu_0=\mu_1$ when $\gamma=\frac53$. 
Therefore
\begin{align*}
\|\frac1\mu\int_0^\tau \mu_\tau\Lambda D\uptheta  \der D \uptheta\,d\tau'\|_{1+\alpha+a,\psi}^2\lesssim 
\begin{cases}
e^{-2\mu_0\tau} (\norm(\tau))^2 & \text{ if } \ 1<\gamma<\frac53 \\
\tau^2 e^{-2\mu_0\tau} (\norm(\tau))^2 & \text{ if } \ \gamma =\frac53
\end{cases}
\end{align*}
The remaining two terms inside the integral on the right-hand side of~\eqref{E:CRITICAL} can be estimated similarly.
The only difference is that after integrating-by-parts in $\tau$ one $\tau$-derivative will fall on $D{\bf V}$ in the third term on the right-hand side of~\eqref{E:CRITICAL}. 
Since $D{\bf V}_\tau$ is a priori not in our energy space, we must use equation~\eqref{E:THETAEQUATION} to express ${\bf V}_\tau$ in terms of ${\bf V}$ and purely spatial derivatives of $\uptheta$. The rest of the proof is then analogous and we obtain
\begin{align}\label{E:CURLEST3}
\|\frac1{\mu}\int_0^\tau \mu(\tau') C_3(\tau')\,d\tau'\|_{1+\alpha+a,\psi}^2 \lesssim
\begin{cases}
 e^{-2\mu_0\tau} \norm(\tau)   & \text{ if } \ 1<\gamma<\frac53 \\  
(1+ \tau^2) e^{-2\mu_0\tau} \norm(\tau) & \text{ if } \ \gamma =\frac53
\end{cases}. 
\end{align} 
The lower-order remainder appearing as the last term on the right-hand side of~\eqref{E:CURLERROREXPANSION} is straightforward to estimate using the above ideas, Moser-type estimates, 
Proposition~\ref{P:EMBEDDING}, Lemma~\ref{L:GAMMASTARASYMP}, and the definition of $\mathcal S^N$:
\begin{align}
& \Big\|\sum_{1\le a'+|\beta'|\le N-1} C_{a',\beta'}\left(\pa_\tau\Lambda_{jm}\pa_r^{a'}\sn^{\beta'}\A^s_m +\Lambda_{jm} \pa_r^{a'}\sn^{\beta'}\pa_\tau\A^s_m\right) 
 \pa_r^{a-a'}\sn^{\beta-\beta'} {\bf V}^k,_s\Big\|_{1+\alpha+a,\psi}^2 \notag \\
 &  \lesssim e^{-2\mu_0\tau} \norm(\tau)\left(1+P(\norm(\tau))\right)
 \label{E:CURLEST3.1}
\end{align}
where $P$ is a polynomial of degree at least 1. By \eqref{E:assumption}, it is in turn bounded by $e^{-2\mu_0\tau} \norm(\tau)$.

From~\eqref{E:CURLEST1}--\eqref{E:CURLEST3.1} and by using $\mu_0\leq \mu_1$ and \eqref{E:assumption}, we conclude that 
\begin{align}\label{E:CURLEST4}
\|\frac{1}{\mu} \int_0^\tau \mu \der[\pa_\tau, \text{Curl}_{\Lambda\A}] {\bf V}d\tau'\|_{1+\alpha+a,\psi}^2 \lesssim  
\begin{cases}
 e^{-2\mu_0\tau} \norm(\tau)   & \text{ if } \ 1<\gamma<\frac53 \\  
(1+ \tau^2) e^{-2\mu_0\tau} \norm(\tau) & \text{ if } \ \gamma =\frac53
\end{cases}. 
\end{align}
In an entirely analogous fashion we can estimate the last term on the right-hand side of~\eqref{E:CURL3COMM} to conclude
\begin{align}\label{E:CURLEST5}
\|\frac{2}{\mu}\int_0^\tau \mu\der\text{Curl}_{\Lambda\A}\left(\Gamma^\ast{\bf V}\right)d\tau'\|_{1+\alpha+a,\psi}^2 \lesssim 
\begin{cases}
 e^{-2\mu_0\tau} \norm(\tau)   & \text{ if } \ 1<\gamma<\frac53 \\  
(1+ \tau^2) e^{-2\mu_0\tau} \norm(\tau) & \text{ if } \ \gamma =\frac53
\end{cases}.
\end{align}
Combining~\eqref{E:CURL3COMM}--\eqref{E:CURLEST0}, Lemma~\ref{L:COMMUTATORBOUNDS}, bounds~\eqref{E:CURLEST4}--\eqref{E:CURLEST5}, and Proposition~\ref{P:NORMENERGY} 
we obtain~\eqref{E:CURLVBOUND}. 

\noindent
{\em Proof of~\eqref{E:CURLTHETABOUND}.}
Applying $\der$ to~\eqref{E:CURLTHETA} we obtain   
\begin{align}
&\text{Curl}_{\Lambda\A}\der{\bf \uptheta} = \der\text{Curl}_{\Lambda\A}\uptheta(0)+\mu(0)\der\text{Curl}_{\Lambda\A}{\bf V}(0)\int_0^\tau\frac{1}{\mu(\tau')}\,d\tau' \notag \\
& + \left[\der,\text{Curl}_{\Lambda\A}\right]\uptheta + \int_0^\tau \frac{1}{\mu(\tau')} \int_0^{\tau'} \mu(\tau'') \der[\pa_\tau, \text{Curl}_{\Lambda\A}] {\bf V}d\tau''\,d\tau' \notag \\
& -\int_0^\tau \frac{2}{\mu(\tau')}\int_0^{\tau'} \mu(\tau'') \der\text{Curl}_{\Lambda\A}\left(\Gamma^\ast{\bf V}\right)d\tau''\,d\tau'. \label{E:CURLTHETAHIGH}
\end{align}

The $\|\cdot\|_{1+\alpha+a,\psi}$ norm of the first two terms on the right-hand side of~\eqref{E:CURLTHETAHIGH} is bounded by a constant multiple of $\norm(0)+\vortnorm(0)$ since $\int_0^\tau\frac{1}{\mu(\tau')}\,d\tau'$
is bounded uniformly-in-$\tau$ due to the exponential growth of $\mu(\tau)$. The term $\|\left[\der,\text{Curl}_{\Lambda\A}\right]\uptheta\|_{1+\alpha+a,\psi}$ is bounded by Lemma~\ref{L:COMMUTATORBOUNDS}.
To bound the fourth term on the right-hand side of~\eqref{E:CURLTHETAHIGH} we use the Cauchy-Schwarz inequality and~\eqref{E:CURLEST4} to obtain
\begin{align}
&\Big\|\int_0^\tau \frac{1}{\mu(\tau')} \int_0^{\tau'} \mu(\tau'') \der[\pa_\tau, \text{Curl}_{\Lambda\A}] {\bf V}d\tau''\,d\tau' \Big\|_{1+\alpha+a,\psi}^2 \notag \\
& \lesssim \int_\Omega\left[\int_0^\tau \frac{(1+\tau')^2}{\mu(\tau')}\,d\tau'  \int_0^\tau \frac1{(1+\tau')^2\mu(\tau')  }\left(\int_0^{\tau'} \mu(\tau'') \der[\pa_\tau, \text{Curl}_{\Lambda\A}] {\bf V}d\tau''\right)^2\,d\tau'\right] 
 w^{1+\alpha+a}\psi\,dx\notag \\
& \lesssim \int_0^\tau \frac1{(1+\tau')^2\mu(\tau')} \Big\|\int_0^{\tau'} \mu(\tau'') \der[\pa_\tau, \text{Curl}_{\Lambda\A}] {\bf V}d\tau''\Big\|_{1+\alpha+a,\psi}^2\,d\tau' \notag \\
& \lesssim \int_0^\tau e^{-\mu_0\tau} \norm(\tau')\,d\tau', \notag 
\end{align}
where the last bound follows from a minor adaptation of the argument leading to~\eqref{E:CURLEST4}.
Using the Cauchy-Schwarz inequality and~\eqref{E:CURLEST5} instead, we analogously show that the last term on the right-hand side of~\eqref{E:CURLTHETAHIGH} satisfies 
\begin{align}
& \Big\|\int_0^\tau \frac{2}{\mu(\tau')}\int_0^{\tau'} \mu(\tau'') \der\text{Curl}_{\Lambda\A}\left(\Gamma^\ast{\bf V}\right)d\tau''\,d\tau' \Big\|_{1+\alpha+a,\psi}^2 \lesssim \int_0^\tau e^{-\mu_0\tau'} \norm(\tau')\,d\tau'  \notag
\end{align}
From the above estimates, the desired bound for 
$\sup_{\tau\in[0,T]}\sum_{a+|\beta|\le N}\|\text{Curl}_{\Lambda\A}\der\uptheta\|_{1+\alpha+a,\psi}^2$ in~\eqref{E:CURLTHETABOUND} is obtained. A completely analogous argument yields the remaining bound
on $\sum_{|\nu|\le N} \|\left[\car,\text{Curl}_{\Lambda\A}\right]\uptheta\|_{1+\alpha+a,1-\psi}^2$ and this concludes the proof of the proposition.  
\end{proof}

\section{Energy estimates}\label{S:ENERGY}

\subsection{High-order energies}

The high-order {\em energy} that emerges naturally from 
the problem is intricate and a priori different from the high-order norm $\mathcal S^N$. Nevertheless, they are equivalent as we shall establish shortly. Before we give the precise expression for the energy functional, we first identify some important quantities. 

Recall that $\Lambda$~\eqref{E:NEWQUANITITIES} is a positive symmetric matrix with the determinant 1 and therefore it is diagonalisable: 
\be\label{E:ORTHDEC}
\Lambda =  P^\top Q  P, \ \  P\in \text{SO}(3), \  Q = \text{diag}(d_1,d_2,d_3)
\ee
where $ P\in \text{SO}(3)$ is the orthogonal matrix that diagonalizes $\Lambda$ and $d_i>0$ are the eigenvalues of $\Lambda$. 
We then define $\mathscr M_{a,\beta}$ and $\mathscr N_\nu$  as the $ P-$conjugates of $\nabla_\eta\der\uptheta$ and $\nabla_\eta\pa^\nu\uptheta$: 
\begin{align}\label{E:MNDEF}
\mathscr M_{a,\beta} : =  P\,\nabla_\eta\der\uptheta\,  P^{\top} \ \text{ and } \ \mathscr N_\nu : =  P\,\nabla_\eta\pa^\nu\uptheta\,  P^{\top}
\end{align}
for any $(a,\beta)$, $0\leq a+|\beta|\le N$  and $\nu,$ $0\leq |\nu|\le N$ respectively. 
By denoting the standard inner product of two vectors  in $\mathbb{R}^3$ by $\langle \cdot , \cdot \rangle$: 
$$\langle u,v\rangle = u^\top v = v^\top u \;  \text{ for any } u, v\in \mathbb{R}^3$$
we  define the high-order energy functional: 
\begin{align}
&\energy(\uptheta,{\bf V})  = \energy(\tau)\notag \\
&=\frac12\sum_{a+|\beta|\le N}
\int_\Omega\psi\Big[\mu^{3\gamma-3} \left\langle\Lambda^{-1}\der{\bf V},\,\der{\bf V}\right\rangle+\delta\left\langle\Lambda^{-1}\der\uptheta,\,\der\uptheta\right\rangle\Big] w^{a+\alpha} \notag \\
& \ \ \ \  \ \ \ \ \  \ \ \ \ \  \ \ \ \ \  \ \ \ \ \  +\psi\J^{-\frac1\alpha}\Big[\sum_{i,j=1}^3 d_id_j^{-1}\left((\mathscr M_{a,\beta})^j_{i}\right)^2+\tfrac1\alpha\left(\text{div}_\eta\der\uptheta\right)^2\Big] w^{a+\alpha+1}\,dy \notag \\
%%%%%%%%%%%
&\ \ \ \  +\frac12\sum_{|\nu|\le N}
\int_\Omega (1-\psi)\Big[\mu^{3\gamma-3} \left\langle \Lambda^{-1} \pa^\nu{\bf V},\,\pa^\nu{\bf V}\right\rangle +\delta \left\langle\Lambda^{-1}\pa^\nu{\bf V},\,\pa^\nu{\bf V}\right\rangle\Big]  w^\alpha
\notag \\
& \ \ \ \  \ \ \ \ \ \ \ \ \ \ \ \ \ \ \ \ \ \ \ \ +(1-\psi)\J^{-\frac1\alpha}\Big[ \sum_{i,j=1}^3 d_id_j^{-1}\left((\mathscr{N}_{\nu})^j_{i}\right)^2 + \tfrac1\alpha\left(\text{div}_\eta\pa^\nu \uptheta\right)^2\Big] w^{1+\alpha} \,dy   \label{E:EDEF}
\end{align}
Note that when $a+|\beta|=0$ and $|\nu|=0$, the sum of the two integrals above produces the zeroth order energy that does not depend on the cutoff function $\psi$. 
 
\begin{remark}
Since the eigenvalues $d_i(\tau)$, $i=1,2,3$ are strictly positive for all $\tau\ge0$ and bounded from above and below uniformly-in-$\tau$ by Lemma~\ref{L:GAMMASTARASYMP} the ratios $\frac{d_i}{d_j}$
are strictly positive and uniformly bounded from above and below for all $\tau\ge0$.
\end{remark} 
 
We also introduce the {\em dissipation} functional: 
\begin{align*}
\mathcal D^N ({\bf V}) = \mathcal D^N(\tau) =\frac{5-3\gamma}{2}\mu^{3\gamma-3}\frac{\mu_\tau}{\mu} &\int_\Omega  \Big[\psi \sum_{a+|\beta|\le N} \left\langle\Lambda^{-1}\der{\bf V},\,\der{\bf V}\right\rangle  w^{a+\alpha}
\\
& \quad+ (1-\psi)\sum_{|\nu|\le N}\left\langle\Lambda^{-1}\pa^\nu{\bf V},\,\pa^\nu{\bf V}\right\rangle  w^{\alpha} \Big]\,dy
\end{align*}

Note that $\mathcal D^N ({\bf V}) \geq 0$ for $\gamma\leq \frac53$.

%%%%%%%%%%%%%%%%%%%%%%%%%%%%%%%

\begin{proposition}[Equivalence of norm and energy]\label{P:NORMENERGY}
Let $(\uptheta, {\bf V}):\Omega\to\mathbb R^3\times \mathbb R^3$ be a unique solution to~\eqref{E:THETAEQUATION}--\eqref{E:THETAINITIAL} defined on some time interval $[0,T]$ and satisfying the a priori assumptions~\eqref{E:assumption}--\eqref{E:APRIORI}. 
Then there exist constants $C_1,C_2>0$ depending only on $A_0,A_1,\delta$ such that 
\begin{align}
C_1\norm (\tau) \le \sup_{0\le\tau'\le\tau}\energy(\tau') %+\int_0^\tau\dissipation(\tau')\,d\tau'
\le C_2\norm(\tau).
\end{align}
\end{proposition}

\begin{proof}
The proof is a simple consequence of the definitions of $\norm$, $\energy$ and uniform boundedness of the matrix $\Lambda$ and its eigenvalues, see Lemma~\ref{L:GAMMASTARASYMP}.
\end{proof}

%%%%%%%%%%%%%%%%%%%%%%%%%%%%%%%

\subsection{Energy identities}

Before we derive the fundamental energy identity, we shall need the following lemma, which captures the key algebraic structure necessary for the extraction of a positive-definite energy.

\begin{lemma} \label{L:KEYLEMMA}
Let $(\uptheta,{\bf V}):\Omega\to\mathbb R^3\times \mathbb R^3$ be a unique solution to~\eqref{E:THETAEQUATION}--\eqref{E:THETAINITIAL} defined on some time interval $[0,T]$. 
Then for any $(a,\beta)$, $0\le a+|\beta|\le N$, and a multi-index $\nu$, $|\nu|\le N$, the following identities hold:
\begin{align}
 \Lambda_{\ell j}(\nabla_\eta\der\uptheta)^i_j\Lambda^{-1}_{im}(\nabla_\eta\der\uptheta)^m_{\ell,\tau}  & =  \frac12\frac{d}{d\tau}\left(\sum_{i,j=1}^3 d_id_j^{-1}(\mathscr M_{a,\beta})^j_{i})^2\right) + \mathcal T_{a,\beta} \label{E:KEYONE}\\
  \Lambda_{\ell j}(\nabla_\eta\pa^\nu\uptheta)^i_j\Lambda^{-1}_{im}(\nabla_\eta\pa^\nu\uptheta)^m_{\ell,\tau}  & =  \frac12\frac{d}{d\tau}\left(\sum_{i,j=1}^3 d_id_j^{-1}(\mathscr N_\nu)^j_{i})^2\right) + \mathcal T_{\nu} \label{E:KEYTWO}
\end{align}
where $\mathscr M_{a,\beta}$ and $\mathscr N_\nu$ are defined in~\eqref{E:MNDEF} and the error terms $\mathcal T_{a,\beta}$ and $\mathcal T_{\nu}$ are given by
\begin{align}
\mathcal T_{a,\beta} & = -  \frac12\sum_{i,j=1}^3\frac d{d\tau}\left(d_id_j^{-1}\right)(\left(\mathscr M_{a,\beta}\right)^j_i)^2 
 - \text{{\em Tr}}\left(Q\mathscr M_{a,\beta}Q^{-1}\left(\pa_\tau P P^\top \mathscr M_{a,\beta}^\top + \mathscr M_{a,\beta}^\top P\pa_\tau P^\top\right)\right) \label{E:TABETA}\\
\mathcal T_{\nu} & = -  \frac12\sum_{i,j=1}^3\frac d{d\tau}\left(d_id_j^{-1}\right)(\left(\mathscr N_\nu\right)^j_i)^2 
 -  \text{{\em Tr}}\left(Q\mathscr N_\nu Q^{-1}\left(\pa_\tau P P^\top \mathscr N_\nu^\top + \mathscr N_\nu^\top P\pa_\tau P^\top\right)\right).
\end{align}
\end{lemma}
\begin{proof}
For any matrix $\tau\mapsto M(\tau)$ let $\tilde M  = P M P^\top$ where $P\in \text{SO}(\mathbb R^3)$ is the orthogonal matrix introduced in~\eqref{E:ORTHDEC}. We then have the following identity
\begin{align}
\text{Tr}\left(\Lambda M \Lambda^{-1}M^\top_\tau\right) & = \text{Tr}\left(Q(PMP^\top)Q^{-1}PM^\top_\tau P^\top\right) \notag \\
& =  \text{Tr}\left(Q\tilde MQ^{-1}\pa_\tau\tilde M^\top\right) -  \text{Tr}\left(Q\tilde MQ^{-1}\left(\pa_\tau P P^\top \tilde M^\top + \tilde M^\top P\pa_\tau P^\top\right)\right)  \notag \\
& = \frac12\frac d{d\tau}\sum_{k,\ell=1}^3(d_kd_\ell^{-1}\tilde M_{k\ell}^2) -  \frac12\sum_{k,\ell=1}^3\frac d{d\tau}\left(d_kd_\ell^{-1}\right)\tilde M_{k\ell}^2 \notag \\
& \ \ \ \  - \text{Tr}\left(Q\tilde MQ^{-1}\left(\pa_\tau P P^\top \tilde M^\top + \tilde M^\top P\pa_\tau P^\top\right)\right). \label{E:MATRIXIDENTITY}
\end{align}
The matrix identity~\eqref{E:MATRIXIDENTITY} is a simple algebraic manipulation, where we used the identity
\begin{align*}
PM^\top_\tau P^\top  & =\pa_\tau(PM^\top P^\top) - \pa_\tau P M^\top P^\top - P M^\top \pa_\tau P^\top \\
& = \pa_\tau(\tilde M^\top) - \pa_\tau P P^\top \tilde M^\top - \tilde M^\top P\pa_\tau P^\top
\end{align*}
but it allows us to extract a positive-definite 
energy quantity modulo a small error term. Applying \eqref{E:MATRIXIDENTITY} to $M=\nabla_\eta\der\uptheta$, we have 
\begin{align}
& \Lambda_{\ell j}\Lambda^{-1}_{im}(\nabla_\eta\der\uptheta)^i_j(\nabla_\eta\der\uptheta)^m_{\ell,\tau} 
= (\Lambda (\nabla_\eta\der\uptheta))_{\ell i}(\Lambda^{-1} \pa_\tau (\nabla_\eta\der\uptheta)^\top)_{i\ell} \notag \\
& =  \frac12\frac{d}{d\tau}\left(\sum_{i,j=1}^3 d_id_j^{-1}((\mathscr M_{a,\beta})^j_{i})^2\right)-  \frac12\sum_{ij}\frac d{d\tau}\left(d_id_j^{-1}\right)(\left(\mathscr M_{a,\beta}\right)^j_i)^2 \notag \\
& \ \ \ \  - \text{Tr}\left(Q\mathscr M_{a,\beta}Q^{-1}\left(\pa_\tau P P^\top \mathscr M_{a,\beta}^\top + \mathscr M_{a,\beta}^\top P\pa_\tau P^\top\right)\right)
\end{align}
where we recall the definitions of $\mathscr M_{a,\beta}$ and $\mathscr N_\nu$ given in~\eqref{E:MNDEF}. This proves~\eqref{E:KEYONE} and the proof of~\eqref{E:KEYTWO} is completely analogous.
\end{proof}

%%%%%%%%%%%%%%%%%%%%%%%%%%%%%%%%%%%%%%%%%%%%%

The next lemma reveals how vector fields $\pa_r$ and $\slashed\nabla_i,$ $i=1,2,3$, commute with the degenerate (spatial) differential operator coming from the pressure gradient term in the momentum equation.
This commutation lemma in particular provides a clue as to why the strength of the weights used in our norm scales proportionally to the number of spatial derivatives used. 

\begin{lemma}\label{Lem:comm} 
Let $q \in \mathbb R_+$  be given. Then for a given $2$-tensor ${\bf T}:\Omega\to\mathbb M^{3\times3}$ and any $i,j\in\{1,2,3\}$ the following 
commutation identities hold
\begin{align}
\pa_r\left[ \frac{1}{ w^q} ( w^{1+q}  {\bf T}^k_i),_k  \right] & =  \frac{1}{ w^{1+q}} ( w^{2+q} \pa_r {\bf T}^k_i),_k + \mathcal C_i^{q+1}[{\bf T}], \label{Commr}\\ 
\sn_j\left[ \frac{1}{ w^q} ( w^{1+q}  {\bf T}^k_i),_k  \right] &=  \frac{1}{ w^{q}} ( w^{1+q} \slashed\nabla_j{\bf T}^k_i),_k + \mathcal C_{ij}^{q}[{\bf T}], \label{Commphi}
\end{align}
where the commutators $ \mathcal C_i^{q+1}[{\bf T}]$ and $ \mathcal C_{ij}^{q}[{\bf T}]$ are given as follows:
\begin{align}
\mathcal C_i^{q+1}[{\bf T}] &= \left(\pa_r w-\frac{ w}{r}\right)\sn_k {\bf T}^k_i+(1+q) \pa_r w,_k {\bf T}^k_i \label{Commrc} \\
\mathcal C_{ij}^{q}[{\bf T}]&=  w \left(\frac{y_j\sn_k}{r^2} + \frac{\delta_{kj}r^2- {y}_k {y}_j}{r^3}\pa_r \right){\bf T}^k_i  + (1+q)\sn_j w,_k{\bf T}^k_i \ \ \text{for } j=1,2,3. \label{Commphic}
\end{align}
\end{lemma}

\begin{proof} Note that 
\begin{align*}
&\pa_r\left[ \frac{1}{ w^q} ( w^{1+q}  {\bf T}^k_i),_k  \right] = \pa_r \left[  w {\bf T}^k_i,_k + (1+q)  w,_k {\bf T}^k_i\right] \\
&=  w \pa_r {\bf T}^k_i,_k + \pa_r  w {\bf T}^k_i,_k + (1+q)  w,_k \pa_r {\bf T}^k_i + (1+q) \pa_r w,_k {\bf T}^k_i \\
&=  w (\pa_r {\bf T}^k_i),_k + (2+q)   w,_k \pa_r {\bf T}^k_i  + w[\pa_r,\pa_k]{\bf T}^k_i  
+ \pa_r  w {\bf T}^k_i,_k -  w,_k \pa_r {\bf T}^k_i   + (1+q) \pa_r w,_k {\bf T}^k_i \\
&= \frac{1}{ w^{1+q}} ( w^{2+q} \pa_r {\bf T}^k_i),_k  - \frac{ w}{r} \sn_k {\bf T}^k_i 
+ \pa_r  w {\bf T}^k_i,_k -  w,_k \pa_r {\bf T}^k_i   + (1+q) \pa_r w,_k {\bf T}^k_i
\end{align*} 
Since $\pa_r = \frac{ y}{r}\cdot \nabla$ and $ w=\frac{\delta}{2(1+\alpha)}(1-| y|^2) = 
\frac{\delta}{2(1+\alpha)}(1-r^2) $, we deduce 
\begin{align*}
 \pa_r  w {\bf T}^k_i,_k -  w,_k \pa_r {\bf T}^k_i = \pa_r w\sn_k{\bf T}^k_i
\end{align*}
which finishes the proof of \eqref{Commr}. 
Using $\sn_j w=0$, we obtain
\begin{align*}
&\sn_j\left[ \frac{1}{ w^q} ( w^{1+q}  {\bf T}^k_i),_k  \right] = \sn_j\left[  w {\bf T}^k_i,_k + (1+q)  w,_k {\bf T}^k_i\right] \\
&=  w \sn_j {\bf T}^k_i,_k + (1+q)  w,_k \sn_j {\bf T}^k_i + (1+q) \sn_j  w,_k {\bf T}^k_i \\
&=  w (\sn_j {\bf T}^k_i),_k + w[\sn_j,\pa_k]{\bf T}^k_i +  (1+q)   w,_k \sn_j{\bf T}^k_i  
  + (1+q) \sn_j  w,_k {\bf T}^k_i\\
  &= \frac{1}{ w^q} ( w^{1+q} \sn_j {\bf T}^k_i),_k  + \mathcal C_{ij}^{q}[{\bf T}]
  \end{align*}
where we used the commutator formula~\eqref{E:PAISNJ} and the definition~\eqref{Commphic}. 
\end{proof}

\begin{remark}\label{R:Comm}
The commutators appearing in the previous lemma consist of lower order terms with respect to both the number of derivatives and the strength of weights in $ w$. 
Even though the term $\pa_r w\sn_k{\bf T}^k_i$ appearing in the expression $\mathcal C_i^{q+1}$~\eqref{Commrc} is potentially very dangerous as it looses a factor of $ w$, 
it is nevertheless of  lower order because it contains the term ${\bf T}^k_i$ with purely tangential derivatives. 
Additionally all the coefficients appearing in the commutators in Lemma~\ref{Lem:comm} are smooth 
 in the annulus $\frac14 \leq r \leq 1$ where the commutation formulas~\eqref{Commr}--\eqref{Commphi} will be actually applied. 
\end{remark}

\begin{proposition}\label{P:ENERGYESTIMATE1}
Let $(\uptheta,{\bf V} ):\Omega\to\mathbb R^3\times \mathbb R^3$ be a unique solution to~\eqref{E:THETAEQUATION}--\eqref{E:THETAINITIAL} defined on some time interval $[0,T]$
and satisfying the a priori assumptions~\eqref{E:assumption}--\eqref{E:APRIORI}. 
Then for any $\tau\in [0,T]$ the following energy inequality holds:
\begin{align}
& \mathcal E^N(\tau) +\int_0^\tau \mathcal D^N(\tau')\,d\tau' \lesssim  \norm(0)  +\vortnorm[\uptheta](\tau) + \int_0^\tau  (\norm(\tau'))^{\frac12} (\vortnorm[{\bf V}](\tau'))^{\frac12}\,d\tau'  \notag \\
& \ \ \ \  + \kappa\norm(\tau) + \int_0^\tau e^{-\mu_0\tau'} \norm(\tau') d\tau'\label{E:ENERGYMAIN0}
\end{align}
where $0<\kappa\ll 1$ is a small constant and the constant $\mu_0$ is defined in \eqref{E:MUZERODEF}.
\end{proposition}

\begin{proof}
To simplify notations, we shall denote by $\mathcal R$ any lower order error term satisfying an estimate of the form
\be\label{E:ERRORR}
\Big|\int_0^\tau \int_{\Omega}\mathcal R \,dx\,d\tau'\Big|
\lesssim  \norm(0) + \kappa\norm(\tau) + \int_0^\tau e^{-\mu_0\tau'} \norm(\tau') d\tau'.
\ee
{\bf Step 1. Zeroth order estimate.} We start with $a+|\beta|=0$ and $|\nu|=0$. In this case, there is no need to obtain two separate estimates by using the cutoff function $\psi$.  
Multiply \eqref{E:THETAEQUATION} by $\Lambda_{im}^{-1}\pa_\tau \uptheta^m $ and integrate over $\Omega$ to obtain 
\begin{align*}
\int_\Omega  w^\alpha\left(\mu^{3\gamma-3}\pa_{\tau\tau}\uptheta_i+\mu^{3\gamma-4}\mu_\tau\pa_\tau\uptheta_i+2\mu^{3\gamma-3}\Gamma^*_{ij}\pa_\tau\uptheta_j  +  \delta \Lambda_{i\ell}\uptheta_\ell \right) \Lambda_{im}^{-1}\pa_\tau \uptheta^m    dy\\
+\int_\Omega  \left( w^{1+\alpha} \Lambda_{ij}\left(\A^k_j\J^{-\frac1\alpha}-\delta^k_j\right)\right),_k \Lambda_{im}^{-1}\pa_\tau \uptheta^m  dy = 0. 
\end{align*}
The first integral can be rewritten as 
\begin{align*}
\frac12\frac{d}{d\tau}\left(\mu^{3\gamma-3} \int_\Omega  w^\alpha  \langle \Lambda^{-1}\pa_\tau\uptheta, \pa_\tau\uptheta \rangle dy  +\delta \int_\Omega w^\alpha |\uptheta |^2 dy\right)+\frac{5-3\gamma}{2}\mu^{3\gamma-4}\mu_\tau \int_\Omega  w^\alpha \langle \Lambda^{-1}\pa_\tau\uptheta, \pa_\tau\uptheta \rangle dy\\
-\frac{\mu^{3\gamma-3}}{2}\int_\Omega  w^\alpha \langle \pa_\tau\Lambda^{-1}\pa_\tau\uptheta, \pa_\tau \uptheta \rangle dy + 2\mu^{3\gamma-3} \int_\Omega  w^\alpha \langle \Lambda^{-1}\pa_\tau\uptheta, \Gamma^\ast \pa_\tau \uptheta \rangle  dy 
\end{align*}
For the second integral, by divergence theorem (integration by parts) and by using the following identities
\[
\A^k_j\J^{-\frac1\alpha} - \delta^k_j = ( \A^k_j-\delta^k_j) \J^{-\frac1\alpha} + \delta^k_j (\J^{-\frac1\alpha} -1) 
\]
\[
\A^k_j-\delta^k_j = \A^k_l  \delta^l_j  -\A^k_l [D\eta]^l_j  = \A^k_l (\delta^l_j - [D  y]^l_j - [D\uptheta]^l_j) = -\A^k_l  [D\uptheta]^l_j 
\]
we obtain 
\[
\begin{split}
&- \int_\Omega  w^{1+\alpha} \Lambda_{ij}\left(\A^k_j\J^{-\frac1\alpha}-\delta^k_j\right) \Lambda_{im}^{-1}\pa_\tau \uptheta^m,_k   dy  \\
&= \int_\Omega  w^{1+\alpha} \J^{-\frac1\alpha} \Lambda_{ij} \A^k_l \uptheta^l,_j  \Lambda_{im}^{-1}\pa_\tau \uptheta^m,_k   dy - \int_\Omega  w^{1+\alpha} (\J^{-\frac1\alpha} -1) \pa_\tau \uptheta^k,_k  dy\\
&=: (i)+(ii)
\end{split}
\]
Now for $(i)$, by rewriting $\Lambda_{ij}=\Lambda_{ip}\delta^j_p = \Lambda_{ip} \A^j_l\eta^l,_p= 
 \Lambda_{ip}(\A^j_p +  \A^j_l\uptheta^l,_p)$, we have 
 \[
 \begin{split}
 (i)&= \int_\Omega  w^{1+\alpha} \J^{-\frac1\alpha}  \Lambda_{ip} \A^j_p  \uptheta^\ell,_j  \Lambda_{im}^{-1} \A^k_\ell \pa_\tau \uptheta^m,_k   dy+ \int_\Omega  w^{1+\alpha} \J^{-\frac1\alpha} \Lambda_{ip} \A^j_l\uptheta^l,_p \A^k_\ell \uptheta^\ell,_j  \Lambda_{im}^{-1}\pa_\tau \uptheta^m,_k   dy \\
 &=\int_\Omega  w^{1+\alpha}  \J^{-\frac1\alpha} \Lambda_{\ell j}[\nabla_\eta\uptheta]^i_j  \Lambda_{im}^{-1} [\nabla_\eta \pa_\tau \uptheta ]^m_\ell dy +\int_\Omega  w^{1+\alpha}  \J^{-\frac1\alpha} [\text{{Curl}}_{\Lambda\A} \uptheta]^\ell_i  \Lambda_{im}^{-1} [\nabla_\eta \pa_\tau \uptheta ]^m_\ell dy \\
 &\quad+ \int_\Omega  w^{1+\alpha} \J^{-\frac1\alpha}  \A^j_l\uptheta^l,_m \A^k_\ell \uptheta^\ell,_j \pa_\tau \uptheta^m,_k   dy\\
 &=: (i)_1+ (i)_2 +(i)_3. 
 \end{split}
 \]
By Lemma \ref{L:KEYLEMMA}, we rewrite the first term 
\[
\begin{split}
(i)_1 &= \frac{1}{2}\frac{d}{d\tau} \int_\Omega  w^{1+\alpha} \J^{-\frac1\alpha}\sum_{i,j=1}^3 d_i d_j^{-1}(({\mathscr M}_{0,0})^j_i)^2 dy  + \int_\Omega  w^{1+\alpha} \J^{-\frac1\alpha} \mathcal T_{0,0} dy \\
&\quad+ 
\frac{1}{2\alpha}\int_\Omega  w^{1+\alpha} \J^{-\frac1\alpha-1}\J_\tau \sum_{i,j=1}^3 d_i d_j^{-1}(({\mathscr M}_{0,0})^j_i)^2 dy
\end{split}
\]
The first term contributes to the energy.  The second term contains $\tau-$derivative of $ P$ or $d_i$ from $\mathcal T_{0,0}$ decaying exponentially fast and the third term contains $D\uptheta_\tau$ from $\J_\tau$, and therefore they are bounded by $e^{-\mu_0\tau} \mathcal S^N(\tau)$. Now the term $(i)_2$ contributes to the second right-hand side term in \eqref{E:ENERGYMAIN0}. The term $(i)_3$ is bounded by $e^{-\mu_0\tau} \mathcal S^N(\tau)$. 

For $(ii)$, we first note that 
 \[
\J = \det [D\eta]= \det [\text{{\bf Id}} + D\uptheta] = 1 + \text{Tr} [D\uptheta] + O(|D\uptheta|^2). 
\]
By writing $1- \J^{-\frac1\alpha}=\frac1\alpha \text{Tr}[D\uptheta] + R_J$ where $R_J=O(|D\uptheta|^2)$, we deduce that 
\[
|(ii)|\lesssim e^{-\mu_0\tau} \mathcal S^N(\tau)
\]
 To be consistent with the definition of our energy $\mathcal E^N$, we may recover the zeroth order divergence energy in $\mathcal E^N$ by adding the following identity
  \[
 \begin{split}
 \frac{1}{2}\frac{d}{d\tau} \frac1\alpha \int_\Omega  w^{1+\alpha} \J^{-\frac1\alpha} |\text{div}_\eta \uptheta |^2 dy =  \frac{1}{2\alpha} \int_\Omega  w^{1+\alpha} \left( \pa_\tau (\J^{-\frac1\alpha} )|\text{div}_\eta \uptheta |^2
 + \J^{-\frac1\alpha} \pa_\tau(|\text{div}_\eta \uptheta |^2) \right) dy\\
  \end{split}
 \]
 where the right-hand side is clearly bounded by $e^{-\mu_0\tau} \mathcal S^N(\tau)$ due to the presence of $\tau-$derivatives of $\uptheta$. 
This complete the zeroth order estimate. 

\

\noindent{\bf Step 2. High order estimates.} In order to obtain high order estimates, we will need to deal with degeneracy caused by vacuum and having correct weights and commuting with right vector fields near the boundary is important. This is why we invoke the cutoff function $\psi$ to localize the difficulty. We focus on getting the estimates of $\mathcal E^N$ involving $\psi$. Here we will crucially use Lemma \ref{Lem:comm}. 

Let $(a,\beta)$ where $a+|\beta|\geq 1$ be given. We first rewrite~\eqref{E:THETAEQUATION}  as 
\be
\mu^{3\gamma-4} \left(\mu\pa_{\tau\tau}\uptheta_i+\mu_\tau\pa_\tau\uptheta_i+2\mu\Gamma^*_{ij}\pa_\tau\uptheta_j\right) +\delta \Lambda_{i\ell}\uptheta_\ell + \frac{1}{ w^\alpha} \left( w^{1+\alpha} \Lambda_{ij}\left(\A^k_j\J^{-\frac1\alpha}-\delta^k_j\right)\right),_k = 0. \label{E:THETAEQUATION00}
\ee
We apply $\pa_r^a\slashed\nabla^\beta$  to~\eqref{E:THETAEQUATION00} by using \eqref{Commr} and \eqref{Commphi} and multiply the resulting equation by the cutoff function $\psi $:
\begin{align}\label{E:ENERGYDERIVATION1}
&\psi \mu^{3\gamma-4} \left(\mu\pa_{\tau\tau}\der\uptheta_i+\mu_\tau\pa_\tau\der\uptheta_i+2\mu \Gamma^*_{ij}\pa_\tau\der\uptheta_j \right)+\delta \psi  \Lambda_{i\ell}\der\uptheta_\ell  \notag \\
& \ \ \ \ +  \psi \frac{1}{ w^{\alpha +a}}\left( w^{1+\alpha+a} \Lambda_{ij} \der \left(\A^k_j\J^{-\frac1\alpha}-\delta^k_j\right)\right),_k = - \psi  \mathcal R_i^{a,\beta},
\end{align}
where the error terms $\mathcal R^{a,\beta}_i$ can be written inductively as follows. For any given index $\beta$ with $|\beta|=n$ we represent it as $\beta=b^n:=(b_n,...,b_1)$, $b_i\in \{1,2,3\}$ so that $b^n=(b_n, b^{n-1})$ and $b^1=(b_1)$. Then for $a=0$ we have the following inductive formula for $\mathcal R^{a,\beta}_i$
\[
\begin{split}
&\mathcal R^{0,b^1}_i = \mathcal C_{ib_1}^{\alpha}[\A \J^{-\frac1\alpha} - \text{{\bf Id}} ] \\
&\mathcal R^{0,b^n}_i = \mathcal C_{ib_n}^{\alpha}[ \slashed\nabla^{b^{n-1}} (\A \J^{-\frac1\alpha} - \text{{\bf Id}} )] +  \slashed\nabla_{b_n} \mathcal R^{0,b^{n-1}}_i \text{ for }n\geq 2
\end{split}
\]
where $\mathcal C_{i j}^{\alpha}[\cdot]$ is defined in \eqref{Commphic} and we recall that $\alpha=\frac1{\gamma-1}$. 
Now in general for $a\geq 1$, by setting  $\mathcal R^{0,0}=0$, we have 
\[
\begin{split}
\mathcal R^{a,\beta}_i = \mathcal C_{i}^{\alpha+a}[ \pa_r^{a-1}\slashed\nabla^\beta (\A \J^{-\frac1\alpha} - \text{{\bf Id}})] +  \pa_r \mathcal R^{a-1, \beta}_i 
\end{split}
\]
where $\mathcal C_i^{\alpha+1}[\cdot]$ is defined in \eqref{Commrc}. 
 Note that $\mathcal R^{a,\beta}_i$ is a lower order term in the sense that it satisfies the error estimate
\begin{align}
\|\mathcal R^{a,\beta}_i\|^2_{a+\alpha,\psi} \lesssim \norm \left((1+P(\norm)\right),
\end{align}
where $P$ is a polynomial of order at least 1. The proof follows from the Proposition~\ref{P:EMBEDDING}. See also Remark \ref{R:Comm}. 

We will now derive the energy estimates from \eqref{E:ENERGYDERIVATION1}. 
Multiplying~\eqref{E:ENERGYDERIVATION1} by $ w^{\alpha+a}\Lambda^{-1}_{im}\pa_\tau\der\uptheta^m$ and integrating over $\Omega$ we arrive at
\begin{align}
&\frac12\frac{d}{d\tau}\left(\mu^{3\gamma-3} \int_\Omega \psi \, w^{\alpha+a} \langle\Lambda^{-1}\der{\bf V},\,\der{\bf V}\rangle \,dy+\delta \int_\Omega\psi  w^{\alpha+a}\, |\der\uptheta |^2 
\,dy \right)  \notag \\
&+\frac{5-3\gamma}{2}\mu^{3\gamma-3} \frac{\mu_\tau}{\mu}\int_\Omega \psi  w^{\alpha+a}\,\langle \Lambda^{-1}\der{\bf V},\,\der{\bf V}\rangle dy \notag \\
& - \frac12\mu^{3\gamma-3}\int_\Omega
\psi  w^{\alpha+a}\, \left\langle \left[\pa_\tau\Lambda^{-1}-4 \Lambda^{-1}\Gamma^*\right]\der{\bf V},\der{\bf V}\right\rangle dy\notag\\
& + \int_\Omega  \psi \left( w^{1+\alpha+a} \Lambda_{ij} \der \left(\A^k_j\J^{-\frac1\alpha}-\delta^k_j\right)\right),_k \Lambda^{-1}_{im} \pa_\tau\der\uptheta^m \,dy  \notag \\
&=\int_\Omega \psi  w^{\alpha+a}\mathcal R^i_{a,\beta} \Lambda^{-1}_{im} \der\uptheta^m_\tau \,dy. \label{E:ENERGYFIRST}
\end{align}
Note that 
\[
\begin{split}
\pa_\tau \Lambda^{-1}-4 \Lambda^{-1}\Gamma^* = \pa_\tau(O^\top O) - 4 O^\top OO^{-1}O_\tau = - \pa_\tau\Lambda^{-1}  + 2(O^\top_\tau O - O^\top O_\tau),  \\
\end{split}
\]
where the orthogonal matrix $O$ is defined in~\eqref{E:DECOMPOSITION}.
Since $(O^\top_\tau O - O^\top O_\tau) ^\top = -(O^\top_\tau O - O^\top O_\tau) $, the third line of~\eqref{E:ENERGYFIRST} reduces to 
\[
 \frac12\mu^{3\gamma-3}\int_\Omega
\psi  w^{\alpha+a}\, \left\langle \pa_\tau\Lambda^{-1} \der{\bf V},\der{\bf V}\right\rangle dy. 
\]
Hence, all lines except the fourth one in the left-hand side of~\eqref{E:ENERGYFIRST} are written in desirable forms. The fourth line will give rise to the full gradient energy at the expense of
a top order term, which will contain the favorable $\text{Curl}_{\Lambda\A}$-component. To see this, we start with the following key formula~\cite{JaMa2015}:
\begin{align}
\pa(\A^k_j\J^{-\frac1\alpha}) & = - \J^{-\frac1\alpha}\A^k_\ell\A^s_j\pa\eta^\ell,_s -\tfrac1\alpha \J^{-\frac1\alpha}\A^k_j\A^s_\ell\pa\eta^\ell,_s 
\end{align}
where $\pa$ is either $\pa_i$, $i=1,2,3$, or $\pa = \pa_\tau$. This formula is still valid for $\pa_r$ or $\sn_i$ if we write the formula with more care: 
\[
X(\A^k_j\J^{-\frac1\alpha}) = - \J^{-\frac1\alpha}\A^k_\ell\A^s_jX(\eta^\ell,_s) -\tfrac1\alpha \J^{-\frac1\alpha}\A^k_j\A^s_\ell X(\eta^\ell,_s)  
\]
for $X\in \mathcal X$ where $\mathcal X=\{\pa_i, \sn_j,\pa_r\}_{i,j=1,2,3}$. We also note that 
\be\label{E:XETA}
X(\eta^\ell,_s)=X( y^\ell,_s + \uptheta^\ell,_s)=  X( \uptheta^\ell,_s) = (X \uptheta^\ell ),_s +
[X,\pa_s]\uptheta^\ell, \ \ X\in\mathcal X, 
\ee
where $[X,\pa_s]$ is given in Lemma \ref{L:COMMUTATORS}. 
Hence, for any multi-index $(a,\beta)$ so that $a+|\beta|> 0$, we have the identity
\begin{align}
&\Lambda_{ij}\der(\A^k_j\J^{-\frac1\alpha}-\delta^k_j)   \notag \\
& = - \Lambda_{ij}\left( \J^{-\frac1\alpha}\A^k_\ell\A^s_j\der\uptheta^\ell,_s +\tfrac1\alpha \J^{-\frac1\alpha}\A^k_j\A^s_\ell \der\uptheta^\ell,_s\right) + \mathcal C^{a,\beta,k}_i(\uptheta)  \notag \\
& =  - \J^{-\frac1\alpha}\A^k_\ell \left(\Lambda_{ij}\A^s_j\der\uptheta^\ell,_s - \Lambda_{\ell j}\A^s_j\der\uptheta^i,_s\right) - \J^{-\frac1\alpha}\A^k_\ell\Lambda_{\ell j}\A^s_j\der\uptheta^i,_s \notag \\
& \ \ \ \ -\tfrac1\alpha\J^{-\frac1\alpha}\Lambda_{ij}\A^k_j\A^s_\ell \der\uptheta^\ell,_s  + \mathcal C^{a,\beta,k}_i(\uptheta) \label{E:ENERGYCRUCIAL} \\
& = -  \J^{-\frac1\alpha}\A^k_\ell  [\text{Curl}_{\Lambda\A}\der\uptheta]^\ell_i -  \J^{-\frac1\alpha}\A^k_\ell \Lambda_{\ell j}
[\nabla_\eta \der\uptheta]^i_j\notag\\
& \ \ \ \ - \tfrac1\alpha\J^{-\frac1\alpha}\Lambda_{ij}\A^k_j \text{div}_\eta (\der\uptheta) + 
\mathcal C_i^{a,\beta,k}(\uptheta).\notag
\end{align}
All that matters in our determination of the lower order commutator  $\mathcal C^{a,\beta,k}_i$ is that for any choice of $i,k\in\{1,2,3\}$ its schematic form is given by
\begin{align}\label{E:ERRORDECOMPOSITION}
\mathcal C^{a,\beta,k}_i(\uptheta) = \sum_{g\in G_{a,\beta}} C_g(\pa\uptheta)\prod_{\{(a_i,\beta_i)\}\in G_{a,\beta}} \pa_r^{a_i}\sn^{\beta_i}\uptheta,
\end{align}
where the set $G_{a,\beta}$ denotes the set of all decompositions into finite sequences of pairs $\{(a_i,\beta_i)\}_{i=1,\dots,m}$ such that 
$\sum_{i=1}^m(a_i+|\beta_i|)<a+|\beta|$, $\beta_i\le\beta$ for all $i=1,\dots,m$, and $C_g(\cdot)$ are some smooth universal functions on the exterior of any ball around the origin $r=0$.
Such a decomposition is a simple consequence of the Leibniz rule and the commutator properties given in  Lemma~\ref{L:COMMUTATORS}.
Using Proposition~\ref{P:EMBEDDING} and the standard Moser inequalities we obtain the bound
\begin{align}\label{E:LOWERORDERTERM}
\|\mathcal C^{a,\beta,k}_i(\uptheta) \|_{\alpha+a,\psi}^2 + \|D\mathcal C^{a,\beta,k}_i(\uptheta)\|_{1+\alpha+a,\psi}^2
\lesssim \mathcal S^N(\tau).
\end{align}
For the sake of completeness we briefly explain the proof of the above estimate. Without loss of generality assume that $a_1+|\beta_1|=\max_{i=1,\dots.m}\{a_i+|\beta_i|\}$. Then
for any $j=2,\dots,m$, $a_j+|\beta_j|\le [\frac N2]$ and therefore by Proposition~\ref{P:EMBEDDING}, estimate~\eqref{embedding1}, we have the bound
\begin{align}
\prod_{j=2}^m \| w^{a_j}\pa_r^{a_j}\sn^{\beta_j}\uptheta\|_{L^\infty(\Omega)}\lesssim \mathcal S^N(\tau)^{\frac{m-1}{2}}.
\end{align}
Therefore, using the Cauchy-Schwarz inequality, the definition of $\mathcal S^N$, decomposition~\eqref{E:ERRORDECOMPOSITION} we conclude that 
\begin{align*}
\|\mathcal C^{a,\beta,k}_i(\uptheta) \|_{\alpha+a,\psi}^2 \lesssim \mathcal S^N(\tau)
\end{align*}
as claimed. A similar bound using~\eqref{embedding2} leads to the second inequality in~\eqref{E:LOWERORDERTERM}.  
We now plug~\eqref{E:ENERGYCRUCIAL} into the last term on the left-hand side of~\eqref{E:ENERGYFIRST} and after integration by parts we obtain 
\begin{align}
&\int_\Omega  \psi \left( w^{1+\alpha+a} \Lambda_{ij} \der \left(\A^k_j\J^{-\frac1\alpha}-\delta^k_j\right)\right),_k \Lambda^{-1}_{im} \pa_\tau\der\uptheta^m \,dy \notag \\
&= \int_\Omega\psi  w^{1+\alpha+a}   \J^{-\frac1\alpha} \left(\A^k_\ell  [\text{Curl}_{\Lambda\A}\der\uptheta]^\ell_i + \A^k_\ell \Lambda_{\ell j}[\nabla_\eta \der\uptheta]^i_j+ \tfrac1\alpha
\Lambda_{ij}\A^k_j \text{div}_\eta (\der\uptheta)\right) \notag \\
& \ \ \ \  \qquad \Lambda^{-1}_{im}\pa_\tau\left(\der\uptheta^m\right),_k\, dy \notag \\
 &\quad+\int_\Omega\psi,_k  w^{1+\alpha+a}   \J^{-\frac1\alpha} \left(\A^k_\ell  [\text{Curl}_{\Lambda\A}\der\uptheta]^\ell_i + \A^k_\ell \Lambda_{\ell j}[\nabla_\eta \der\uptheta]^i_j+ \tfrac1\alpha
\Lambda_{ij}\A^k_j \text{div}_\eta (\der\uptheta)\right) \notag \\
& \ \ \ \  \qquad \Lambda^{-1}_{im}\pa_\tau\der\uptheta^m\, dy \notag \\
&\quad +\int_\Omega \psi \left(  w^{1+\alpha+a} \mathcal C^{a,\beta,k}_i \right),_k \Lambda^{-1}_{im}\pa_\tau\der\uptheta^m \,dy \label{E:ENERGYFIRST1} \\
&=: I_1 + I_2 +I_3.\notag
\end{align}
The rest of the proof will be devoted to the estimation of $I_1$, $I_2$ and $I_3$. 

\

\noindent\underline{Estimation of $I_3$ in \eqref{E:ENERGYFIRST1}.}
By the Cauchy-Schwarz inequality and~\eqref{E:LAMBDABOUNDS}, the last integral $I_3$ above satisfies the bound
\begin{align}
& \Big|\int_0^\tau \int_\Omega \psi \left(  w^{1+\alpha+a} \mathcal C^{a,\beta,k}_i \right),_k \Lambda^{-1}_{im}\pa_\tau\der\uptheta^m \,dyd\tau'\Big| \notag \\
& \lesssim \int_0^\tau \|\der{\bf V}^m\|_{1+\alpha+a,\psi} \left(\|\mathcal C^{a,\beta,k}_i(\uptheta) \|_{\alpha+a,\psi}+\|D\mathcal C^{a,\beta,k}_i \|_{1+\alpha+a,\psi}\right)\,d\tau' \notag \\
& \lesssim \int_0^\tau e^{-\mu_0\tau'} \mathcal S^N(\tau')\,d\tau',
\end{align}
where we used the definition of $\mathcal S^N$ and~\eqref{E:LOWERORDERTERM}. 

\

\noindent\underline{Estimation of $I_2$ in \eqref{E:ENERGYFIRST1}.} We now explain the second-to-last integral $I_2$ on the right-hand side of~\eqref{E:ENERGYFIRST1}. Note that $\psi,_k $ is supported in $B_{\frac34}({\bf 0})\setminus B_{\frac14}({\bf 0})$ and it is not bounded by $\psi$, and therefore the integral cannot be controlled by using only $\psi$ norms. On $B_{\frac34}({\bf 0})\setminus B_{\frac14}({\bf 0})$, however, $\der\uptheta$ and $\der {\bf V}$ can be expressed in terms of cartesian derivatives $\pa^\nu$ of $\uptheta$ and $\bf{V}$ for $|\nu|\leq a+|\beta|$. Hence we use both norms involving $\psi$ and $1-\psi$ to control the integral:
\begin{align}
& \Big|\int_0^\tau \int_{B_{\frac34}({\bf 0})\setminus B_{\frac14}({\bf 0})}\psi,_k  w^{1+\alpha+a}   \J^{-\frac1\alpha} \left(\A^k_\ell  [\text{Curl}_{\Lambda\A}\der\uptheta]^\ell_i + \A^k_\ell \Lambda_{\ell j}[\nabla_\eta \der\uptheta]^i_j+ \tfrac1\alpha
\Lambda_{ij}\A^k_j \text{div}_\eta (\der\uptheta)\right) \notag \\
& \ \ \ \  \qquad \Lambda^{-1}_{im}\pa_\tau\der\uptheta^m\, dy \,d\tau' \Big| \notag \\
& \lesssim \int_0^\tau\|D\psi\|_{L^\infty(\Omega)} \Big(   \|\der{\bf V}^m\|_{\alpha+a,\psi} \|\nabla_\eta \der\uptheta\|_{1+\alpha+a,\psi} + \sum_{|\nu|\le a+|\beta|}\|\pa^\nu{\bf V}\|_{\alpha,1-\psi} \sum_{|\nu|\le a+|\beta|}\|\nabla_\eta\pa^\nu\uptheta\|_{1+\alpha,1-\psi}\Big)\,d\tau' \notag \\
& \lesssim \int_0^\tau e^{-\mu_0\tau'} \mathcal S^N(\tau')\,d\tau'. 
\end{align}
The last two integrals $I_2$ and $I_3$ on the right-hand side of~\eqref{E:ENERGYFIRST1} therefore contribute to the error term $\mathcal R$ satisfying the bound~\eqref{E:ERRORR}.

\

\noindent\underline{Estimation of $I_1$ in \eqref{E:ENERGYFIRST1}.} We evaluate now the first integral  on the right-hand side of~\eqref{E:ENERGYFIRST1}: 
\begin{align}
&\int_\Omega\psi  w^{1+\alpha+a}   \J^{-\frac1\alpha} \left(\A^k_\ell  [\text{Curl}_{\Lambda\A}\der\uptheta]^\ell_i + \A^k_\ell \Lambda_{\ell j}[\nabla_\eta \der\uptheta]^i_j+ \tfrac1\alpha
\Lambda_{ij}\A^k_j \text{div}_\eta (\der\uptheta)\right) \notag \\
& \ \ \ \  \qquad \Lambda^{-1}_{im}\pa_\tau\der\uptheta^m,_k\, dy \notag \\
&= \int_\Omega \psi  w^{1+\alpha+a}  \J^{-\frac1\alpha}[\text{Curl}_{\Lambda\A}\der\uptheta]^\ell_i \A^k_\ell\Lambda^{-1}_{im}\pa_\tau\left(\der\uptheta^m\right),_k\,dy\notag \\
& \ \ \ \  +\int_\Omega \psi  w^{1+\alpha+a}  \J^{-\frac1\alpha} \left(\A^k_\ell \Lambda_{\ell j}[\nabla_\eta \der\uptheta]^i_j+ \tfrac1\alpha
\Lambda_{ij}\A^k_j \text{div}_\eta (\der\uptheta)\right)\Lambda^{-1}_{im}\pa_\tau\left(\der\uptheta^m\right),_k\, dy \label{E:ENERGYTWO}
\end{align}
We first focus on the first integral on the right-hand side of~\eqref{E:ENERGYTWO}. We integrate-by-parts with respect to $\tau$ and obtain
\begin{align}
& \int_0^\tau \int_\Omega w^{1+\alpha+a}  \J^{\frac1\alpha} \A^k_\ell\Lambda^{-1}_{im}
\psi[\text{Curl}_{\Lambda\A}\der\uptheta]^\ell_i \pa_\tau\left(\der\uptheta^m\right),_k \,dx\,d\tau' \notag \\
& =  \int_\Omega w^{1+\alpha+a}  \J^{\frac1\alpha} \A^k_\ell\Lambda^{-1}_{im}
\psi[\text{Curl}_{\Lambda\A}\der\uptheta]^\ell_i \left(\der{\bf\uptheta}^m\right),_k \,dx\Big|^\tau_0 \notag \\
& \ \ \ \ -\int_0^\tau \int_\Omega w^{1+\alpha+a}  \J^{\frac1\alpha} \A^k_\ell\Lambda^{-1}_{im}
\psi[\text{Curl}_{\Lambda\A}\der{\bf V}]^\ell_i \left(\der{\bf\uptheta}^m\right),_k \,dx\,d\tau' \notag\\
& \ \ \ \ -\int_0^\tau \int_\Omega w^{1+\alpha+a}  \pa_\tau\left(\J^{\frac1\alpha} \A^k_\ell\Lambda^{-1}_{im}\right)
\psi[\text{Curl}_{\Lambda\A}\der{\bf \uptheta}]^\ell_i \left(\der{\bf\uptheta}^m\right),_k \,dx\,d\tau'  \notag \\
& \ \ \ \ -\int_0^\tau \int_\Omega w^{1+\alpha+a}  \J^{\frac1\alpha} \A^k_\ell\Lambda^{-1}_{im}
\psi[\text{Curl}_{\Lambda_\tau\A}\der{\bf \uptheta}]^\ell_i \left(\der{\bf\uptheta}^m\right),_k \,dx\,d\tau'.  \label{E:INTINTAU}
\end{align} 
The first term on the right-hand side of~\eqref{E:INTINTAU} is estimated by using the Young inequality and Proposition~\ref{P:VORTICITYBOUNDS}:
\begin{align}
&\int_\Omega w^{1+\alpha+a}  \J^{\frac1\alpha} \A^k_\ell\Lambda^{-1}_{im}
\psi\sum_{a+|\beta|\le N}[\text{Curl}_{\Lambda\A}\der\uptheta]^\ell_i \left(\der{\bf\uptheta}^m\right),_k \,dx\Big|^\tau_0 \notag \\
&\lesssim \norm(0) + \kappa\|\nabla_\eta\der\uptheta\|_{1+\alpha+a,\psi}^2 + \frac1\kappa \|\text{Curl}_{\Lambda\A}\der\uptheta\|_{1+\alpha+a,\psi}^2 \notag\\
&\lesssim \norm(0) +\kappa\norm(\tau) + \vortnorm[\uptheta](\tau) \label{E:CURLBOUNDFIRST}
\end{align}
where we used the Young inequality in the first estimate.  
To bound the second term on the right-hand side of~\eqref{E:INTINTAU}
we use the Cauchy-Schwarz inequality and the a priori bounds~\eqref{E:APRIORI},  
\begin{align}
&\Big|\int_0^\tau \int_\Omega w^{1+\alpha+a}  \J^{\frac1\alpha} \A^k_\ell\Lambda^{-1}_{im}
\psi[\text{Curl}_{\Lambda\A}\der{\bf V}]^\ell_i \left(\der{\bf\uptheta}^m\right),_k \,dx\,d\tau' \Big| \notag \\
&\lesssim \int_0^\tau\|\nabla_\eta\der\uptheta\|_{1+\alpha+a,\psi} \|\text{Curl}_{\Lambda\A}\der{\bf V}\|_{1+\alpha+a,\psi} \,d\tau' \notag \\
&\lesssim \int_0^\tau (\norm(\tau'))^{\frac12} (\vortnorm[{\bf V}](\tau'))^{\frac12}\,d\tau'.\label{E:CURLBOUNDSECOND} 
\end{align}   
The last two integrals on the right-hand side of~\eqref{E:INTINTAU} are easily estimated by using the Young inequality, definition of $\mathcal S^N$, and Lemma~\ref{L:GAMMASTARASYMP} 
and they  therefore contribute to the error term $\mathcal R$ satisfying the bound~\eqref{E:ERRORR}.

We next move onto the second integral on the right-hand side of~\eqref{E:ENERGYTWO}, which will give rise to our energy.  
Note that 
\begin{align}
&\int_\Omega \psi  w^{1+\alpha+a}  \J^{-\frac1\alpha} \left(\A^k_\ell \Lambda_{\ell j}[\nabla_\eta \der\uptheta]^i_j+ \tfrac1\alpha
\Lambda_{ij}\A^k_j \text{div}_\eta (\der\uptheta)\right)\Lambda^{-1}_{im}\pa_\tau\left(\der\uptheta^m\right),_k\, dy \notag \\
&=\int_\Omega \psi  w^{1+\alpha+a}  \J^{-\frac1\alpha} \left( \Lambda_{\ell j}[\nabla_\eta \der\uptheta]^i_j \Lambda^{-1}_{im}[\nabla_\eta\pa_\tau \der\uptheta ]_\ell^m+\tfrac1\alpha \text{div}_\eta (\der\uptheta) \text{div}_\eta (\pa_\tau \der\uptheta) \right)dy\notag 
\\
&=\int_\Omega \psi  w^{1+\alpha+a}  \J^{-\frac1\alpha} \left( \Lambda_{\ell j}[\nabla_\eta \der\uptheta]^i_j \Lambda^{-1}_{im}\pa_\tau [\nabla_\eta \der\uptheta ]_\ell^m+\tfrac1\alpha\text{div}_\eta (\der\uptheta) \pa_\tau \text{div}_\eta ( \der\uptheta) \right)dy\notag 
\\
&- \int_\Omega \psi  w^{1+\alpha+a}  \J^{-\frac1\alpha} \left( \Lambda_{\ell j}[\nabla_\eta \der\uptheta]^i_j \Lambda^{-1}_{im}\pa_\tau\A^k_\ell \left(\der\uptheta^m\right),_k+ \tfrac1\alpha\text{div}_\eta (\der\uptheta)  \pa_\tau \A^k_j \left(\der\uptheta^j\right),_k  \right) dy
\label{E:ALMOSTE}
\end{align}  
The last integral consists of lower order terms and they will contribute to the error term $\mathcal R$ satisfying~\eqref{E:ERRORR}. The source of the exponential-in-$\tau$ decay
are the the terms of the form $\pa_\tau\A$ which at the top order behave like $D{\bf V}$.   
By Lemma~\ref{L:KEYLEMMA}, the first integral in the right hand side of~\eqref{E:ALMOSTE} can be rewritten in the form
\begin{align}
&\int_\Omega \psi  w^{1+\alpha+a}  \J^{-\frac1\alpha} \left( \Lambda_{rj}[\nabla_\eta \der\uptheta]^i_j \Lambda^{-1}_{im}\pa_\tau [\nabla_\eta \der\uptheta ]_\ell^m\, + \tfrac1\alpha\text{div}_\eta (\der\uptheta) \pa_\tau \text{div}_\eta ( \der\uptheta) \right)dy \notag \\
& = \frac12\frac{d}{d\tau}\left\{\int_\Omega \psi  w^{1+\alpha+a}  \J^{-\frac1\alpha}\left(\sum_{i,j=1}^3 d_id_j^{-1}(\mathscr M_{a,\beta})^j_{i})^2 + \frac1\alpha\left(\text{div}_\eta\der\uptheta\right)^2\right)\,dy\right\} \notag\\
&\ \ \ \ + \int_\Omega \psi  w^{1+\alpha+a}  \J^{-\frac1\alpha} \mathcal T_{a,\beta} dy  \notag \\
& \ \ \ \ -  \frac12\int_\Omega \psi  w^{1+\alpha+a} \pa_\tau\left(\J^{-\frac1\alpha}\right)\left(\sum_{i,j=1}^3 d_id_j^{-1}(\mathscr M_{a,\beta})^j_{i})^2 + \frac1\alpha\left(\text{div}_\eta\der\uptheta\right)^2\right)\,dy, \label{E:ALMOSTE2}
\end{align}
where we recall that $\mathscr M_{a,\beta}$ is defined in~\eqref{E:MNDEF}.
Using the a priori bounds~\eqref{E:APRIORI}, definition~\eqref{E:TABETA}, and~\eqref{E:LAMBDABOUNDS2} it follows that 
\[
\Big| \int_\Omega \psi  w^{1+\alpha+a}  \J^{-\frac1\alpha} \mathcal T_{a,\beta} dy \Big| \lesssim e^{-\mu_0\tau}\norm(\tau).
\]
Similar argument applies to the last term on the right-hand side of~\eqref{E:ALMOSTE2}, where the source of an exponentially decaying term is $\pa_\tau\J$ and we rely on Lemma~\ref{L:GAMMASTARASYMP}.
Therefore, the last two terms on the right-hand side of~\eqref{E:ALMOSTE2} will contribute to error term $\mathcal R$ and are bounded by $e^{-\mu_0\tau} \mathcal S^N(\tau)$. 

From~\eqref{E:ENERGYFIRST},~\eqref{E:ENERGYFIRST1},~\eqref{E:ENERGYTWO},~\eqref{E:ALMOSTE},~\eqref{E:CURLBOUNDFIRST},~\eqref{E:CURLBOUNDSECOND}, 
and~\eqref{E:ALMOSTE2} we infer that 
\begin{align}
&\frac12\left(\mu^{3\gamma-3} \int_\Omega \psi \, w^{\alpha+a} \langle\Lambda^{-1}\der{\bf V},\,\der{\bf V}\rangle \,dy+\delta \int_\Omega\psi  w^{\alpha+a}\, |\der\uptheta |^2 
\,dy \right)  \notag \\
& + \frac12\left(\int_\Omega \psi  w^{1+\alpha+a}  \J^{-\frac1\alpha}\left(\sum_{i,j=1}^3 d_id_j^{-1}(\mathscr M_{a,\beta})^j_{i})^2 +\frac1\alpha \left(\text{div}_\eta\der\uptheta\right)^2\right)\,dy\right) \notag \\
&+\int_0^\tau \frac{5-3\gamma}{2}\mu^{3\gamma-3} \frac{\mu_\tau}{\mu}\left(\int_\Omega \psi  w^{\alpha+a}\,\langle \Lambda^{-1}\der{\bf V},\,\der{\bf V}\rangle dy\right)d\tau' \notag \\
&\lesssim \norm(0)  +\vortnorm[\uptheta](\tau) + \int_0^\tau \norm(\tau')^{\frac12} \vortnorm[{\bf V}](\tau')^{\frac12}\,d\tau'   + \kappa\norm(\tau) + \int_0^\tau e^{-\mu_0\tau'} \norm(\tau') d\tau'.
\label{E:ENERGY1}
\end{align}

For any index $\nu = (\nu_1,\nu_2,\nu_3)\in\mathbb Z_{\ge0}^3$ we now commute~\eqref{E:THETAEQUATION} with the operator $(1-\psi)\car$, where we recall that the cut-off
function $\psi$ is defined in~\eqref{E:PSIDEF}. Here the estimates will be obtained away from the boundary and therefore, the weight $ w$ has a positive lower bound in $B_{\frac34}({\bf 0})$. Hence the change of the weights is not critical to closing the estimates. 
A similar computation yields the following energy inequality:
\begin{align}
&\frac12\left(\mu^{3\gamma-3} \int_\Omega (1-\psi) \, w^{\alpha} \langle\Lambda^{-1}\car{\bf V},\,\car{\bf V}\rangle \,dy+\delta \int_\Omega(1-\psi ) w^{\alpha}\, |\car\uptheta |^2 
\,dy \right)  \notag \\
& + \frac12\left(\int_\Omega (1- \psi)  w^{1+\alpha}  \J^{-\frac1\alpha}\left(\sum_{i,j=1}^3 d_id_j^{-1}(\mathscr N_\nu)^j_{i})^2 +\frac1\alpha \left(\text{div}_\eta\car\uptheta\right)^2\right)\,dy\right) \notag \\
&+\int_0^\tau \frac{5-3\gamma}{2}\mu^{3\gamma-3} \frac{\mu_\tau}{\mu}\left(\int_\Omega (1-\psi)  w^{\alpha}\,\langle \Lambda^{-1}\car{\bf V},\,\car{\bf V}\rangle dy \right) d\tau'\notag \\
&\lesssim \norm(0)  +\vortnorm[\uptheta](\tau) + \int_0^\tau \norm(\tau')^{\frac12} \vortnorm[{\bf V}](\tau')^{\frac12}\,d\tau'   + \kappa\norm(\tau) + \int_0^\tau e^{-\mu_0\tau'} \norm(\tau') d\tau'.
\label{E:ENERGY2}
\end{align}
Summing~\eqref{E:ENERGY1} and~\eqref{E:ENERGY2} we arrive at~\eqref{E:ENERGYMAIN0}. 
\end{proof}

%%%%%%%%%%%%%%%%%%%%%%%%%%%%%%%
%%%%%%%%%%%%%%%%%%%%%%%%%%%%%%%

\section{Nonlinear energy inequality and proof of the main theorem}\label{S:MAINPROOF}

By the local-in-time well-posedness theory from~\cite{JaMa2015} there exists a unique solution to the initial value problem~\eqref{E:THETAEQUATION}--\eqref{E:THETAINITIAL} on some
time interval $[0,T], T>0.$ See Theorem \ref{T:LWP}. 
From Propositions~\ref{P:VORTICITYBOUNDS} and~\ref{P:ENERGYESTIMATE1} with $\kappa>0$ chosen small enough and by using the equivalence of the norm and energy in Proposition \ref{P:NORMENERGY}, we conclude that there exist universal constants $c_1,c_2,c_3\geq 1$ such that 
for any $0\le\sigma\le\tau\le T$
\begin{align}\label{E:ENERGYBOUND}
\sup_{\sigma\le \tau'\le \tau} \energy(\tau') 
 \le c_1\energy(\sigma) + c_2 \left(\norm(0)+\vortnorm[{\bf V}](0)\right) + c_3\int_\sigma^\tau e^{-\frac{\mu_0}{2}\tau'} \norm (\tau';\sigma) d\tau'.
\end{align}
Here $\norm(\tau;\sigma)$ denotes the sliced norm of $\norm$ from $\sigma$ to $\tau$ with $\sup_{0\le \tau' \le\tau}$ replaced by $\sup_{\sigma\le \tau' \le\tau}$  in the definition of $\norm$ \eqref{E:SNORM}. Note that $\norm(\tau;0)=\norm(\tau)$ and $\norm(\tau)=\max\{ \norm(\sigma),  \norm(\tau;\sigma)\}$ for any $\sigma\in[0,\tau]$. 
When $\sigma=0$, the estimate \eqref{E:ENERGYBOUND} is a direct consequence of Propositions~\ref{P:VORTICITYBOUNDS} and~\ref{P:ENERGYESTIMATE1}, while for general $\sigma>0$ the estimate follows by an identical argument, applying our analysis to the interval $[\sigma,\tau]$.

By a standard well-posedness estimate, we deduce that the time of existence $T$ is inversely proportional to the size of the initial data, i.e.: $T\sim (\norm(0)+\vortnorm[{\bf V}](0))$.
Choose $\varepsilon>0$ so small that the time of existence $T$ satisfies 
\be\label{E:LOCALTIME}
e^{-\mu_0T/4}\le\frac{\kappa \mu_0}{c_3}, \ \ \sup_{\tau\le T}\norm(\tau)\le C\left(\norm(0)+\vortnorm[{\bf V}](0)\right) 
\ee
where 
$C>0$ is a universal constant provided by the local-in-time well-posedness theory. Let 
\[
C_\ast =\frac{3(c_1C_2C+c_2)}{C_1}
\]
where $C_1, C_2$ ($C_1\leq C_2$) are the constants appearing in the equivalence of the energy and the norm given in Proposition \ref{P:NORMENERGY}. 
Define 
\[
\mathcal T : = \sup_{\tau\ge0}\{\text{ solution to~\eqref{E:THETAEQUATION}--\eqref{E:THETAINITIAL} exists on $[0,\tau)$ and} \ \norm(\tau)\le  C_\ast \left(\norm(0)+\vortnorm[{\bf V}](0)\right)\} . 
\]
Observe that $\mathcal T\geq T$. 
Letting $\sigma=\frac T2$ in~\eqref{E:ENERGYBOUND} for any $\tau\in[\frac T2,\mathcal T]$ we obtain 
\[
\sup_{\tau'\in[\frac T2,\tau]} \energy(\tau') 
\le c_1\energy(\frac T2)+c_2\left(\norm(0)+\vortnorm[{\bf V}](0)\right) 
+  c_3\int_{\frac T2}^\tau e^{-\frac{\mu_0}{2}\tau'}\norm(\tau';\frac{T}{2})\,d\tau'.
\]
Therefore, using~\eqref{E:LOCALTIME} we conclude that for any $\tau\in[\frac T2,\mathcal T]$
\begin{align}
\sup_{\tau'\in[\frac T2,\tau]} \energy(\tau') 
& \le  c_1\energy(\frac T2)+c_2\left(\norm(0)+\vortnorm[{\bf V}](0)\right)   + \frac{c_3}{\mu_0}e^{-\mu_0 T/4}\norm(\tau;\frac{T}{2}) \notag \\
& \le c_1C_2\norm(\frac T2)+c_2\left(\norm(0)+\vortnorm[{\bf V}](0)\right)  +\kappa\norm(\tau;\frac{T}{2}). \label{E:CONT1}
\end{align}
Since by~\eqref{E:LOCALTIME}, $\norm(\frac T2)\le C(\norm(0)+\vortnorm[{\bf V}](0))$ we conclude from~\eqref{E:CONT1} that
\begin{align*}
C_1\norm(\tau;\frac{T}{2}) \le (c_1C_2C+ c_2)\left(\norm(0)+\vortnorm[{\bf V}](0)\right) +\kappa\norm(\tau;\frac{T}{2}),
\end{align*}
which for sufficiently small $\kappa$ gives
\begin{align*}
\norm(\tau;\frac{T}{2}) & \le \frac{2(c_1C_2C+c_2)}{C_1}(\norm(0)+\vortnorm[{\bf V}](0))  
< C_\ast \left(\norm(0)+\vortnorm[{\bf V}](0)\right)
\end{align*}
and hence $\norm(\tau)< C_\ast \left(\norm(0)+\vortnorm[{\bf V}](0)\right)$. It is now easy to check the a priori bounds in \eqref{E:assumption} and \eqref{E:APRIORI} are in fact improved. 
For instance, by the fundamental theorem of calculus 
\[
\|D\uptheta\|_{W^{1,\infty}} = \|\int_0^\tau D{\bf V}\|_{W^{1,\infty}} \le \int_0^\tau e^{-\mu_0\tau'} \norm(\tau')\,d\tau' \lesssim \varepsilon <\frac16, \ \ \tau\in[0,\mathcal T) 
\]
for $\varepsilon>0$ small enough. Similar arguments apply to the remaining a priori assumptions.

From the continuity of the map $\tau\mapsto \norm(\tau')$ and the definition of $\mathcal T$ we conclude that 
$\mathcal T = \infty$ and the solution to~\eqref{E:THETAEQUATION}--\eqref{E:THETAINITIAL} exists globally-in-time. 
Together with Proposition~\ref{P:VORTICITYBOUNDS} the global bound~\eqref{E:GLOBALBOUND} follows.

Note that  by Lemma~\ref{L:GAMMASTARASYMP} $\mu(\tau)\sim_{\tau\to\infty} e^{\mu_1\tau}$ where
$\mu_1>0$ is defined in \eqref{E:MUZERODEF}, see also Lemma~\ref{L:GAMMASTARASYMP}. 
Statement~\eqref{E:EXPONENTIALDECAY} follows from the already established global bound $\norm(\tau)\le C\varepsilon$ and the presence of the exponentially growing term 
$\mu^{3\gamma-3}$ in the definition~\eqref{E:SNORM} of $\norm$. Estimate~\eqref{E:VORTICITYDECAY} is a simple consequence of~\eqref{E:GLOBALBOUND} and Proposition~\ref{P:VORTICITYBOUNDS}.

From the global-in-time boundedness of $\norm$, there exists a weak limit $\uptheta_\infty$ independent of $\tau$ such that 
$\norm(\uptheta_\infty,0) \lesssim \norm (\tau) \le C\varepsilon$. 
Observe that for any $0<\tau_2<\tau_1$, 
\begin{align*}
&\|  \der\uptheta (\tau_1) - \der\uptheta(\tau_2) \|^2_{a+\alpha,\psi}  = \int_\Omega\psi  w^{\alpha+a} \left| \int_{\tau_2}^{\tau_1} \der\uptheta_\tau d\tau \right|^2\,dy\\
&\le \left( \int_{\tau_2}^{\tau_1} \mu^{\frac{3-3\gamma}2} d\tau \right) \left( \int_{\tau_2}^{\tau_1}  \mu^{\frac{3\gamma-3}2}\int_\Omega\psi  w^{\alpha+a} \left| \der{\bf V}\right|^2 dy \,d\tau   \right)\\
&  \lesssim ( e^{-\frac{3\gamma-3}2\mu_0\tau_2} -e^{-\frac{3\gamma-3}2\mu_0\tau_1}  ) 
\int_{\tau_2}^{\tau_1}  \mu^{-\frac{3\gamma-3}{2}}\mu^{3\gamma-3}\int_\Omega\psi  w^{\alpha+a} \left| \der{\bf V}\right|^2 dy \,d\tau  \\
& \lesssim \varepsilon( e^{- \frac{3\gamma-3}2\mu_0\tau_2} -e^{- \frac{3\gamma-3}2\mu_0\tau_1}  ),
\end{align*}
where the last line follows from the integrability of $\int_{\tau_2}^{\tau_1}  \mu^{-\frac{3\gamma-3}{2}}\,d\tau$ and the fact that $\sup_{0\le \tau'\le\tau} \left(\mu^{3\gamma-3}\|\der{\bf V}(\tau')\|_{L^2_{\alpha+a,\psi}}^2\right) \lesssim\norm(\tau)\lesssim\varepsilon$ for any $\tau\ge0$.
Therefore, given a strictly increasing sequence $\tau_n \rightarrow \infty$, the sequence $\{\der\uptheta (\tau_n)\}_{n=1}^\infty$  is Cauchy in $L^2_{a+\alpha,\psi}$ for any $(a,\beta)$ satisfying $a+|\beta|\le N$.
By an analogous argument we can show that the sequence $\{\pa^\nu\uptheta (\tau_n)\}_{n=1}^\infty$  is Cauchy in $L^2_{\alpha,1-\psi}$. This completes the proof of~\eqref{E:ASYMP}.

\section*{Acknowledgements}

The authors express their gratitude to P. Rapha\"el for fruitful discussions and for pointing out connections to the treatment of self-similar singular behavior for nonlinear Schr\"odinger equations. They also thank C. Dafermos for his feedback and pointing out important references. 
JJ is supported in part by NSF grants DMS-1608492 and DMS-1608494 and a von Neumann fellowship of the Institute for Advanced Study through the NSF grant DMS-1128155. 
 MH acknowledges the support of the EPSRC Grant EP/N016777/1.

\appendix
\renewcommand{\theequation}{\Alph{section}.\arabic{equation}}
\setcounter{theorem}{0}\renewcommand{\theorem}{\Alph{section}.\??arabic{prop}}

\section{Asymptotic-in-$\tau$ behavior of affine solutions}\label{A:ASYMPTOTIC}

In this section we collect some of the basic properties of affine motions that are used at many places in our estimates. Their proofs are rather straightforward and follow directly
from the description of the asymptotic behavior of the solutions of~\eqref{E:SIDERISODE} from~\cite{Sideris}.

We remind the reader that for any $M\in\mathbb M^{3\times3}$ we denote by $\|M\|$ the Hilbert-Schmidt norm of the matrix $M$.

\begin{lemma}[Asymptotic behavior of $A$, $\Gamma^*= O^{-1} O_\tau$, and $\Lambda =  O^{-1} O^{-\top}$] \label{L:GAMMASTARASYMP}
For any $\gamma\in(1,\frac53]$ and any pair of initial conditions $(A(0),\dot A(0))\in  \text{{\em GL}}^+(3)\times \mathbb M^{3\times3}$ there exist matrices $A_0,A_1,M(t)$ such that 
the unique solution $A(t)$ to the Cauchy problem
\begin{align}
A_{tt} &= \delta \det{A}^{1-\gamma} A^{-\top} \\
A(0) & = A(0),  \ \ A_t(0) = \dot A(0)
\end{align}
can be written in the form
\begin{align}\label{E:AASYMP}
A(t) = A_0 + t A_1 + M(t), \ \ t\ge0,
\end{align}
where $A_0,A_1$ are both time-independent and $M(t)$ satisfies the bounds
\begin{align}\label{E:MASYMP}
\|M(t)\| = o_{t\to\infty}(1+t), \ \ \|\pa_t M(t)\| \lesssim (1+t)^{3-3\gamma}.
\end{align}
Moreover
\begin{align}
e^{\mu_1\tau} \lesssim \mu(\tau)  \lesssim  e^{\mu_1\tau}, \ \ \tau\ge0, \label{E:MUASYMP0}\\
\lim_{\tau\to\infty} \|\Gamma^* - e^{-\mu_1\tau} \mu_1 A_0 A_1^{-1} \| = 0,\label{E:GAMMASTARASYMPTOTICS}
\end{align}
where 
\[
\Gamma^* =  O^{-1} O_\tau, \ \ \ \mu_1 = (\det A_1)^{\frac13}>0.
\]
Furthermore there exists a constant $C>0$ such that 
\begin{align}\label{E:LAMBDABOUNDS}
\|\Lambda_\tau\|\le C e^{-\mu_1\tau}, \ \  \|\Lambda_{\tau\tau}\| \le C e^{-2\mu_0\tau}
, \ \ \|\Lambda\| + \|\Lambda^{-1}\| \le C,  \\
\sum_{i=1}^3\left(d_i+\frac 1{d_i}\right) \le C \label{E:EIGENVALUESBOUND}\\
\sum_{i=1}^3|\pa_\tau d_i| + \|\pa_\tau P\| \le C e^{-\mu_1\tau} \label{E:LAMBDABOUNDS2} \\
\frac1 C|{\bf w}|^2 \le \langle \Lambda^{-1}{\bf w}, {\bf w}\rangle  \le C |{\bf w}|^2 , \ {\bf w}\in \mathbb R^3. \label{E:LAMBDAPOSDEF}
\end{align}
where  
$d_i$, $i=1,2,3$, are the eigenvalues of the matrix $\Lambda$ and $P\in \text{SO}(3)$ satisfies
$\Lambda = P^\top Q P,$ $Q = \text{{\em diag}}(d_i)$.
\end{lemma}

\begin{proof}
Asymptotic behavior~\eqref{E:AASYMP}--\eqref{E:MASYMP} and bound~\eqref{E:MUASYMP0} are a  consequence of Theorem 3 and Lemma 6 from~\cite{Sideris}.

\

\noindent
{\em Proof of~\eqref{E:GAMMASTARASYMPTOTICS}.}
Since $ O = \frac{A}{\mu}$ we have
\begin{align}
 O_t  = \frac{A_t}{\mu} - \frac{A\mu_t}{\mu^2} 
 = \frac{A_t}{\mu} - \frac{A\text{Tr}(A^{-1}A_t)}{3\mu},  \label{E:DTBARO}
\end{align}
where we used the formula
\[
\mu_t =\pa_t\left((\det A)^{\frac13}\right)  = \frac13 \mu^{-2} \pa_t\det A =  \frac13 \mu^{-2} \mu^3 \text{Tr}(A^{-1}A_t) =  \frac13 \mu \text{Tr}(A^{-1}A_t).
\]
Therefore, from~\eqref{E:AASYMP} and~\eqref{E:MASYMP} 
it is easy to obtain the following asymptotics:
\begin{align}
\text{Tr}(A^{-1}A_t)&\sim \text{Tr}((A_0+tA_1 + M(t))^{-1}(A_1+\pa_t M(t))) = t^{-1}\text{Tr}((\frac{A_0}{t}+A_1)^{-1}(A_1+O(t^{3-3\gamma})) \notag \\
&\sim t^{-1}\text{Tr}(A_1^{-1}(A_1+O(t^{3-3\gamma})) = \frac{3}{t} + O(t^{2-3\gamma}).  \label{E:TA}
\end{align}
This implies that 
\[
 O_t\sim \frac{A_1}{\mu} - \frac{(A_0+tA_1)\frac3t}{3\mu} = \frac{A_0}{t\mu},
\]
where we made use of~\eqref{E:MASYMP} again.
Recalling that 
$
\frac{d\tau}{dt} = \frac1{\mu}
$
we obtain the following asymptotic behavior 
$
 O_\tau\sim_{\tau\to\infty} \frac{A_0}{t(\tau)}.
$
Using~\eqref{E:AASYMP} again it follows that 
\be\label{E:MUASYMP}
\det A\sim \det (A_0+t A_1) = t^3 \det (\frac{A_0}t + A_1)\sim t^3\det A_1, \ \ \mu \sim t (\det A_1)^{\frac13}.
\ee
Since 
\[
t = t(\tau) \sim_{\tau\to\infty} e^{(\det A_1)^{\frac13}\tau}.
\]
and $ O \sim \frac{A_0+tA_1}{\mu}$ we conclude that 
\[
 O^{-1}\sim \frac{\mu}{t} (\frac{A_0}{t}+A_1)^{-1}\sim \frac{\mu}{t} A_1^{-1} \sim  (\det A_1)^{\frac13}A_1^{-1},
\]
where we used~\eqref{E:MUASYMP}.
Therefore
\[
\Gamma^* =  O^{-1} O_\tau \sim (\det A_1)^{\frac13}A_1^{-1}\frac{A_0}{t(\tau)} = e^{-\mu_1\tau} \mu_1 A_0 A_1^{-1},
\]
where 
$
\mu_1 = (\det A_1)^{\frac13},
$
and this completes the proof of~\eqref{E:GAMMASTARASYMPTOTICS}. Proof of~\eqref{E:LAMBDABOUNDS} is similar.

\

\noindent
{\em Proof of~\eqref{E:LAMBDABOUNDS}--\eqref{E:LAMBDABOUNDS2}.}
From the definition of $\Lambda$ we have
\begin{align}
\Lambda_\tau = -  O^{-1} O_\tau  O^{-1}  O^{-\top} -  O^{-1} O^{-\top}  O^\top_\tau  O^{-\top} 
= -\Gamma^*\Lambda - \Lambda(\Gamma^*)^\top = -2 \Lambda(\Gamma^*)^\top, \label{E:DTAUBARLAMBDA}
\end{align}
where we used the symmetry of $\Lambda$ in the last equality.
Since $\|\Lambda\|\lesssim 1$ it follows by part (i) that $\|\Lambda_\tau\|\lesssim e^{-\mu_1\tau}$.
To bound $\Lambda_{\tau\tau}$ we note that by~\eqref{E:DTAUBARLAMBDA}
$
\Lambda_{\tau\tau} = - 2 \Lambda_\tau (\Gamma^*)^\top - 2\Lambda (\Gamma_\tau^*)^\top.
$
Since both $\|\Lambda_\tau\|$ and $\|\Gamma^*\|$ decay exponentially, it remains to prove the decay of $\|\Gamma_\tau^*\|$.
From $\Gamma^* = O^{-1}O_\tau$ it follows that $\Gamma_\tau^*=-(\Gamma^*)^2 + O^{-1}O_{\tau\tau}$, and therefore it remains to prove the decay of $\|O_{\tau\tau}\|$.
A simple calculation shows that
\[
O_{\tau\tau} = \mu\left( A_{tt} - \frac{A_t \text{Tr}(A^{-1}A_t)+A \pa_t\text{Tr}(A^{-1}A_t)}{3} \right).
\]
Using the asymptotic behavior~\eqref{E:TA},~\eqref{E:SIDERISODE}, we can refine the asymptotics~\eqref{E:TA} to show that 
$A_t \text{Tr}(A^{-1}A_t)+A \pa_t\text{Tr}(A^{-1}A_t) = O(t^{2-3\gamma})$ and therefore from the above equation it follows that
\[
\|O_{\tau\tau}\|\lesssim (1+t)^{3-3\gamma} \lesssim e^{-2\mu_0\tau}.
\]
This yields the second bound in~\eqref{E:LAMBDABOUNDS}.  
Bounds~\eqref{E:EIGENVALUESBOUND}~\eqref{E:LAMBDABOUNDS2} follow by similar arguments using~\eqref{E:AASYMP}, while~\eqref{E:LAMBDAPOSDEF}
is a direct consequence of~\eqref{E:EIGENVALUESBOUND}.
\end{proof}

\section{Commutators}\label{A:COMM}

In order to evaluate various commutator terms that arise from commuting differential operators $\der$ with the usual Cartesian derivatives or apply the Leibniz rule we shall 
rely on the fact that the high-order Sobolev norms expressed in polar coordinates  are equivalent to the usual high-order Sobolev norms on the support of function $\psi$.

\begin{lemma}\label{L:COMMUTATORS}
Let $\mathcal X: =\{\pa_i, \sn_j,\pa_r\}_{i,j=1,2,3}$ be a collection of the standard Cartesian, normal, and tangential vector-fields.
For any two vector fields $X_k,X_\ell\in\mathcal X$, $k,\ell = 1,\dots,7$ there commutator satisfies the following relationship
\[
[X_k,X_\ell] = \sum_{m=1}^3c^m_{k\ell}\sn_m + c_{k\ell}\pa_r, \ \ k,\ell =1,\dots 7,
\]
where the functions $c^m_{k\ell}$, $c_{k\ell}$, are $C^\infty$ on the exterior of  any ball around the origin $r=0$.
\end{lemma}

\begin{proof}
The proof is a simple consequence of the following direct calculations. For any $i,j\in\{1,2,3\}$ we have
\begin{align}
[\pa_r,\pa_i] &= [\pa_r,\sn_i] = - \frac1r\sn_i, \ \ [\pa_i,\pa_j]  = 0, \notag \\
[\sn_i,\sn_j] & = \frac{y_i\sn_j-y_j\sn_i}{r^2}, \notag \\
[\pa_i,\sn_j] & = -\frac{y_j\sn_i}{r^2} + \triangle_{ij}\pa_r, \ \ \triangle_{ij}:= \frac{ y_i y_j-\delta_{ij}r^2 }{r^3}. \label{E:PAISNJ} 
\end{align}
\end{proof}

Lemma~\ref{L:COMMUTATORS} is a technical tool allowing us to bound the lower order commutators in our energy estimates.

Using the product rule and the relationship $\A=[D\eta]^{-1}$ it is easy to see that the following formulas hold
\begin{align}
\pa_r\A^k_i  & = -\A^k_s\pa_r\pa_m\eta^s\A^m_i =  -\A^k_s\pa_r\left(\pa_m\uptheta^s+\delta^s_m\right)\A^m_i \notag \\
& = -\A^k_s(\pa_r\uptheta^s),_m\A^m_i  + \A^k_s [\pa_m,\pa_r]\uptheta^s\A^m_i. \label{E:COMMFORMULA1}
\end{align} 
Similarly, for any $j=1,2,3$ we have
\begin{align}\label{E:COMMFORMULA2}
\sn_j\A^k_i & =  -\A^k_s(\sn_j\uptheta^s),_m\A^m_i  +\A^k_s[\pa_m,\sn_j]\uptheta^s\A^m_i. 
\end{align}
Therefore, if $\beta=(0,0,0)$ we have the formula
\begin{align}
\der \A^k_i  = \pa_r^a\A^k_i &= - \A^k_s(\pa_r^a \uptheta^s),_m\A^m_i - \A^k_r[\pa_r^{a-1},\pa_m]\pa_r\uptheta^s\A^m_i \notag \\
&-\sum_{1\le a'\le a } c_{a'}\pa_r^{a'}(\A^k_s\A^m_i)\pa_r^{a-a'}\left(\pa_r\uptheta^s\right),_m 
 +\pa_r^{a-1}\left(\A^k_s[\pa_m,\pa_r]\uptheta_{,j}^s\A^m_i \right). \label{E:TOPORDER1}
\end{align}
If $|\beta| > 0$, then with $e_1=(1,0,0)$, $e_2 = (0,1,0)$, and $e_3=(0,0,1)$ we obtain
\begin{align}
\der\A^k_i & =- \A^k_s(\der \uptheta^s),_m\A^m_i - \A^k_r[\pa_r^a\sn^{\beta-e_j},\pa_m]\sn_j\uptheta^s\A^m_i \notag \\
&  \ \ \ \ -\sum_{0<a'+|\beta'|\le a+|\beta| \atop \beta'\le\beta-e_j }c_{a'\beta'}\pa_r^{a'}\slashed\nabla^{\beta'}(\A^k_s\A^m_i)\pa_r^{a-a'}\slashed\nabla^{\beta-e_j-\beta'}\left(\sn_j\uptheta^s\right),_m \notag \\
&  \ \ \ \ -\pa_r^a\slashed\nabla^{\beta-e_j}\left( \A^k_s[\pa_m,\sn_j]\uptheta^s\A^m_i \right), \label{E:TOPORDER2}
\end{align}
where $c_{a'\beta'},c_{\beta'}$ are positive universal constants and the Einstein summation convention does not apply to the index $j$. 
Finally, the high-order commutators appearing on the right-hand side in the identities~\eqref{E:TOPORDER1}--\eqref{E:TOPORDER2} can be expressed as a linear combination of the elements of $\mathcal X$ with smooth coefficients away from zero. In other words, for any $j=1,2,3$ the following commutator identity holds:
\begin{align}\label{E:DIFFCOMMUTATOR}
[\der,\pa_j] = \sum_{\ell=0}^{a+|\beta|}\sum_{a'+|\beta'|=\ell  \atop a'\leq a+1} C^j_{a',\beta',\ell}\pa_r^{a'}\sn^{\beta'} 
\end{align}
for some universal coefficients $C^j_{a',\beta',\ell}$ which are smooth on the exterior of any ball around the origin $r=0$. Formula~\eqref{E:DIFFCOMMUTATOR} is a direct 
consequence of Lemma~\ref{L:COMMUTATORS}.

\section{Weighted spaces, Hardy inequalities, and Sobolev embeddings}\label{A:HARDY}

In this section, we recall Hardy inequalities and embedding results of weighted function spaces.  
First of all, we state the Hardy inequality near $r=1$. 

\begin{lemma} \label{hardy} (Hardy inequality \cite{KMP}) 
Let  $k $ be a real number and $g$ a function satisfying $\int_0^1  (1-r)^{k +2}(g^2 + g'^2) dr < \infty $.

If $k > -1$, then we have $ \int_0^1  (1-r)^{k}  g^2 dr  \leq C \int_0^1  (1-r)^{k+2} (g^2 + |g'|^2) dr   $. 

If $k < -1$, then $g$ has a trace at $r=1$  and 
$ \int_0^1  (1-r)^{k}  (g - g(1))^2 dr  \leq C \int_0^1  (1-r)^{k+2}   |g'|^2  dr$. 
 \end{lemma}

Since $ w$ depends only on $r$ and $ w$ behaves like a distance function $1-r$, using Lemma \ref{hardy}, we in particular get 
  \begin{equation}  \label{hard1}
 \int_{B_1({\bf 0})\setminus B_{\frac14}({\bf 0})} \psi  w^{k} | u  |^2 dy \,  
 \lesssim  \int_{B_1({\bf 0})\setminus B_{\frac14}({\bf 0})} \psi  w^{k+2} \left(  |\partial_r u  |^2       +    |u|^2 \right)    dy 
\end{equation} 
for any nonnegative real number $k> -1$ and for any $u\in C^\infty(B_1({\bf 0})\setminus B_{\frac14}({\bf 0}))$. We can apply \eqref{hard1} to $\der u$ and thereafter apply \eqref{hard1} repeatedly  
to the right-hand side. As a consequence  for any $-1< k < \alpha+a$,  
\begin{equation}\label{hard2}
\begin{split}
 \int_{B_1({\bf 0})\setminus B_{\frac14}({\bf 0})}  \psi  w^{k} | \der u  |^2 dy \,  
&  \lesssim  \int_{B_1({\bf 0})\setminus B_{\frac14}({\bf 0})} \psi  w^{k+2} \left(  |\partial_r\der u  |^2       +    |\der u|^2 \right)    dy  \\
&  \lesssim  \sum_{j=0}^m  \int_{B_1({\bf 0})\setminus B_{\frac14}({\bf 0})} \psi  w^{k+2j} |\partial_r^j\der u  |^2   dy. 
  \end{split}
\end{equation}
Upon choosing $m = \lceil a+\alpha-k \rceil$ where $\lceil \ \rceil$ is the ceiling function, we obtain 
\begin{equation}\label{hard3}
\|  \der u  \|^2_{k,\psi}  \lesssim  \sum_{j=0}^{\lceil a+ \alpha-k\rceil} \|  \partial_r^j\der u \|^2_{\alpha+a+j,\psi}
\end{equation}
for any $u\in C^\infty(B_1({\bf 0})\setminus B_{\frac14}({\bf 0}))$.

As a consequence of \eqref{hard3}, we obtain the weighted Sobolev-Hardy inequality: 

\begin{proposition} \label{P:EMBEDDING}
For any $u\in C^\infty(B_1({\bf 0}))$, we have 
\begin{align}
&\sup_{B_1({\bf 0})\setminus B_{\frac14}({\bf 0})}\left| w^{\frac{a_1}{2}} \pa_r^{a_1}\slashed\nabla^{\beta_1} u \right| \notag\\
& \lesssim \sum_{a+|\beta|\leq a_1+|\beta_1|+\lceil \alpha\rceil+6} \| \der u \|_{a+\alpha,\psi} 
+ \sum_{|\nu|\leq a_1+|\beta_1| +2 } \| \pa^\nu u\|_{\alpha,1-\psi} ,\label{embedding1} \\
&\sup_{B_1({\bf 0})\setminus B_{\frac14}({\bf 0})}\left| w^{\frac{a_1}{2}} D \pa_r^{a_1}\slashed\nabla^{\beta_1} u \right|\notag \\
& \lesssim \sum_{a+|\beta|\leq a_1+|\beta_1|+\lceil \alpha\rceil+6} \| \nabla_\eta \der u \|_{a+\alpha+1,\psi} 
+ \sum_{|\nu|\leq a_1+|\beta_1|+2} \|\nabla_\eta\pa^\nu u\|_{\alpha+1,1-\psi}. \label{embedding2}
\end{align}
 \end{proposition}

We omit the technical details of the proof, as it follows from standard estimates relying on the $H^2(\Omega)\hookrightarrow L^\infty(\Omega)$ continuous embedding and the Hardy inequality~\eqref{hard3}.

\section{Starting from Lagrangian formulation}\label{A:LAG}

Consider the Lagrangian formulation of the E$_\gamma$-system~\cite{CoSh2012, JaMa2015}: 
\be\label{euler_lag}
\zeta_{tt} + \tfrac{\gamma}{\gamma-1}\A_\zeta^T \nabla (w \J_\zeta^{1-\gamma})=0
\ee
where $\A_\zeta$ and $\J_\zeta$ are induced by the flow map $\pa_t\zeta(t,y)={\bf{u}}(t,\zeta(t,y))$ (${\bf u}$ is the original fluid velocity appearing in \eqref{E:EULER}) and $w=(\rho_0 \J_\zeta(0))^{\gamma-1}\geq 0$ is an enthalpy function that depends only on the initial data.   
To find affine motions one makes the ansatz $ \zeta (t,y)= A(t) y$~\cite{Sideris}. With $\A_{\zeta}^\top=[D\zeta]^{-\top}= A(t)^{-\top}$ and $\J_{\zeta}=\det A$, by plugging this ansatz into \eqref{euler_lag}, we obtain 
\be\label{eq}
A_{tt} y + \tfrac{\gamma}{\gamma-1}  (\det A)^{1-\gamma} A^{-\top} \nabla w=0
\ee
Since $w$ is independent of $t$, \eqref{eq} will be satisfied if we demand 
\be\label{eqA}
A_{tt}= \delta  (\det A)^{1-\gamma} A^{-\top} , \quad \delta y = -  \tfrac{\gamma}{\gamma-1}  \nabla w, \quad \delta >0,  
\ee
which precisely yields the affine motions~\eqref{E:RHOAUA} discovered in~\cite{Sideris}, which form a set $\mathscr S$. 

In order to study the stability of elements of $\mathscr S$ we want to realize them as time-independent background solutions. 
Given such an affine motion $ A$ we modify the flow map $\zeta$ and define $\eta:=  A^{-1} \zeta$. Since $\A_\zeta^\top =  A^{-\top} \A_\eta^\top$ 
and $\J_\zeta= (\det A) \J_\eta$, we obtain 
\[
 \eta_{tt} + 2 A^{-1}  A_t \eta_t + A^{-1}  A_{tt}\eta + \tfrac{\gamma}{\gamma-1}  (\det A)^{1-\gamma} A^{-1} A^{-\top} \A_\eta^\top
\nabla (w \J_\eta^{1-\gamma}) =0
\]
By using \eqref{eqA} we can rewrite the previous equation in the form
\be\label{zetat}
 (\det A)^{\gamma-\frac13} \eta_{tt} + 2 (\det A)^{\gamma-\frac13}  A^{-1} A_t \eta_t +\delta \Lambda \eta + \tfrac{\gamma}{\gamma-1} \Lambda  \A_\eta^\top
\nabla (w \J_\eta^{1-\gamma}) =0
\ee
where $\Lambda =\det A^{\frac23}  A^{-1}A^{-T}$. 

We rescale the time variable $t$ so that $1+t\sim e^{\mu_1\tau}$ by setting $\frac{d\tau}{d t} = (\det A)^{-\frac13}$. Then \eqref{zetat} can be written as 
\be\label{zetatau}
 (\det A)^{\gamma-1} \eta_{\tau\tau}  -\tfrac13 (\det  A)^{\gamma-2} (\det A)_\tau \eta_\tau + 2 (\det A)^{\gamma-1} A^{-1}  A_\tau \eta_\tau   +\delta \Lambda \eta + \tfrac{\gamma}{\gamma-1} \Lambda  \A_\eta^\top
\nabla (w \J_\eta^{1-\gamma}) =0
\ee
We now recall $A^{-1}  A_\tau = \mu^{-1} \mu_\tau I + O^{-1}O_\tau$ where $ A= \mu O$ and $\mu= (\det  A)^{\frac 13}$ (see Section~\ref{S:SS}). The equation for $\eta$ reads
\be\label{zetatau2}
 \mu^{3\gamma-3}\eta_{\tau\tau} + \mu^{3\gamma -4}\mu_\tau \eta_\tau + 2 \mu^{3\gamma-3}\Gamma^\ast \eta_\tau  +\delta \Lambda \eta + \tfrac{\gamma}{\gamma-1} \Lambda  \A_\eta^\top
\nabla (w \J_\eta^{1-\gamma}) =0
\ee
where $\Gamma^\ast=O^{-1}O_\tau$.  It is clear that $\eta(y)\equiv y$ corresponds to Sideris' affine motions, and  equation \eqref{zetatau2} is nothing but~\eqref{E:VELOCITYLAGR2}. By considering $\uptheta=\eta -y$, we obtain the $\uptheta$-equation~\eqref{E:THETAEQUATION}.

\end{document}